\documentclass[11pt,a4paper]{article}

\usepackage[margin=25mm]{geometry}
\usepackage[T1]{fontenc}
\usepackage{lmodern}
\usepackage[round,authoryear]{natbib}
\setcitestyle{citesep={;},aysep={,},yysep={;}}

\usepackage{amsmath,amssymb,amsthm,mathtools,bm}
\usepackage{booktabs,graphicx,microtype,placeins}
\usepackage{tikz}
\usetikzlibrary{arrows.meta,calc}
\usepackage{capt-of}
\usepackage{xcolor}
\usepackage[hidelinks]{hyperref}
\usepackage{url}
\usepackage[nameinlink,capitalize,noabbrev]{cleveref}
\usepackage{enumitem}
\usepackage{needspace}
\usepackage{enotez}
\usepackage{comment}
\setenotez{backref=true,totoc=section}

\DeclareMathOperator{\tr}{tr}
\DeclareMathOperator{\rank}{rank}
\DeclareMathOperator{\supp}{supp}
\DeclareMathOperator{\divergence}{div}
\newcommand{\mmd}{\operatorname{MMD}^2}
\newcommand{\R}{\mathbb{R}}
\newcommand{\E}{\mathbb{E}}
\newcommand{\Id}{\mathrm{Id}}

\newcommand{\dd}{\mathrm{d}}
\newcommand{\g}{\mathsf{g}}
\newcommand{\eps}{\varepsilon}

\newcommand{\nab}{\nabla}
\newcommand{\hal}{\frac12}
\newcommand{\cb}{\color{blue}}
\newcommand{\p}{\partial}
\renewcommand{\d}{\delta}
\newcommand{\V}{\mathcal{V}}
\newcommand{\wh}{\widehat}

\newtheoremstyle{arxivstatement}
  {10pt plus 2pt minus 2pt}{10pt plus 2pt minus 2pt}
  {\itshape}{}{\bfseries}{.}{.6em}{}
\theoremstyle{arxivstatement}
\newtheorem{theorem}{Theorem}[section]
\newtheorem{proposition}[theorem]{Proposition}
\newtheorem{lemma}[theorem]{Lemma}
\newtheorem{corollary}[theorem]{Corollary}
\theoremstyle{remark}
\newtheorem{remark}[theorem]{Remark}

\title{Mean--Fluctuation Dynamics at the Edge of Stability}

\author{Antonin Chodron de Courcel\thanks{École Normale Supérieure, 45, rue d'Ulm, 75005 Paris; email: decourcel@ihes.fr}}
\date{}

\begin{document}
\maketitle

\begin{abstract}
    We study the dynamics of gradient descent in the Edge of Stability regime, where the learning rate is large enough to induce persistent oscillations in the trajectory, which has been linked to better generalization performance.
    We introduce the \emph{mean--fluctuation} dynamics, a tractable continuous-time model coupling the window-averaged trajectory to its fluctuation covariance.
    Among our contributions, we rigorously derive this model from gradient descent in a sharp-valley framework, characterize its stationary states and their linear stability, and establish precise connections with other effective dynamics.
    Numerical experiments illustrate these predictions and their finite-time limitations.
    We also study our model in the overparametrized regime of wide two-layer networks at a fixed learning rate, where we rigorously derive a kinetic equation describing weights and their fluctuations as a Wasserstein-2 gradient flow, for which we prove well-posedness, a mean-field limit, and conditional convergence results.
\end{abstract}

\section{Introduction}
\label{sec:introduction}

Many theoretical works on machine learning concern convergence guarantees for gradient-descent (GD) algorithms in the idealized regime of small learning rates. Such a regime is mathematically appealing because it reduces to the (already complex) study of gradient flows, for which we have a vast theory \citep{jordan1998variational,otto2001geometry,ambrosio2008gradient}.
However, practitioners usually operate in regimes where the learning rate is not arbitrarily small with respect to the other parameters \citep{lecun1998efficient}, a practice which has been associated with better generalization properties \citep{li2019large,lewkowycz2020catapult}. Yet, the implicit bias generated by such a procedure remains unclear.
\bigbreak
\noindent
\begin{minipage}[t]{0.73\textwidth}
The regime of optimization we are interested in is therefore the so-called
\emph{edge of stability} (EoS) regime \citep{cohen2021eos}, in which the
largest Hessian eigenvalues hover around the known stability threshold
$2/\eta$, where $\eta$ is the learning rate.  As described by
\citet{damian2023self}, GD iterates dynamically stabilize around this threshold through the following mechanism: the sharpness first increases until it reaches the threshold $2/\eta$ (``progressive sharpening''), at which point GD jumps out of sharp wells to go back to flatter regions (``curvature-related drift''). Once in this region, the sharpness increases again as in the first stage.

\end{minipage}%
\hfill%
\begin{minipage}[t]{0.27\textwidth}
\small
\raggedleft
{\color{green} $\lambda_1 < 2/\eta$}, \quad 
{\color{red} $\lambda_1 > 2/\eta$}, \\
{\cb $\lambda_1 \ll 2/\eta$: gradient flow}
\includegraphics[scale=0.2]{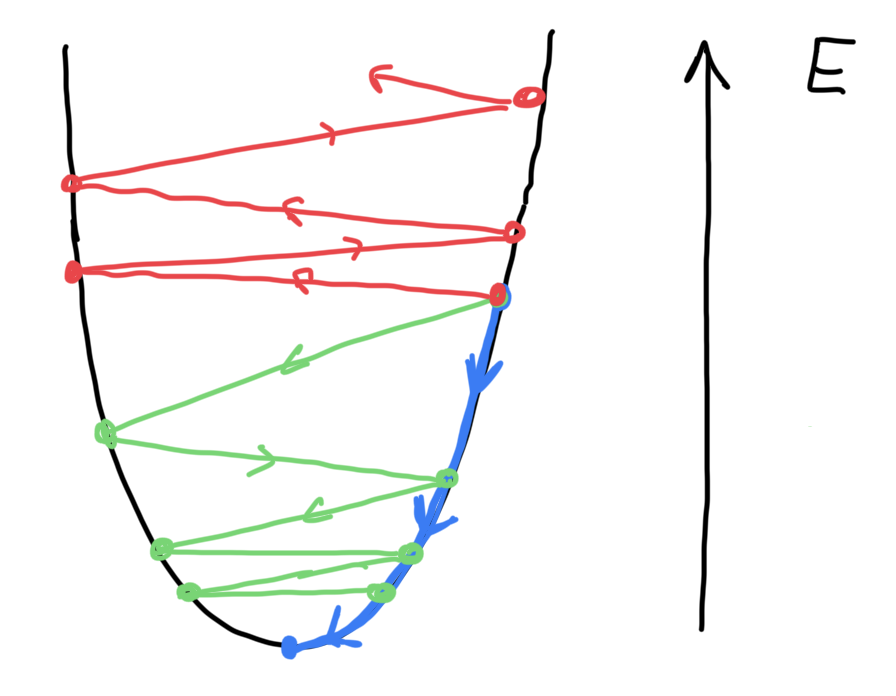}
\end{minipage}
\bigbreak
\noindent
Our goal is to derive a mathematical framework that is both amenable to analysis and quantitatively faithful to GD in this edge of stability regime, at least for a non-negligible timescale.
Specifically, our model is a rather simple yet rich coupled ordinary differential equation (ODE) for window-averaged parameters and the covariance of their oscillations. We therefore call our model \emph{mean--fluctuation} (MF) dynamics.
Its distinguishing feature is an explicit covariance evolution, obtained from first principles, that allows several oscillating directions and unstable Hessian eigenvalues without imposing instantaneous covariance equilibration. 

\begin{figure}[!ht]
    \centering
    \includegraphics[width=\textwidth]{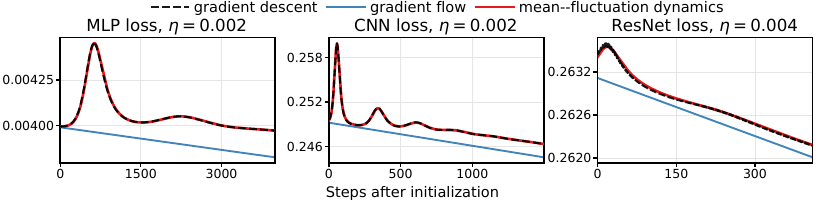}
    \caption{One EoS eigenvalue: our flow agrees with the GD iterates across architectures (\cref{fig:architecture-eig1}).}
    \label{fig:intro-one-mode-loss}
\end{figure}

\paragraph{Edge of stability and implicit bias.}
Flat minima have long been associated with better generalization
\citep{hochreiter1997flat}, whether they are dynamically reached by some implicit bias (e.g. using small batches of data \citep{keskar2017largeminima}) or hard-coded in the training procedure \citep{foret2021sharpnessaware}. 
Taking large learning rates is therefore believed to favor flatter regions 
\citep{lewkowycz2020catapult} by driving the dynamics at the edge of stability \citep{cohen2021eos,arora2022eos,damian2023self}.
Further analyses address less general situations in order to characterise as much as possible the edge of stability mechanism \citep{zhu2022minimalist,chen2023beyond}.

\paragraph{Continuous-time models.}
The idea of tracking the time-averaged position of parameters has been used in \cite{cohen2024centralflow},
where the covariance of the oscillations is prescribed through a complementarity problem at each time. This does not capture the transient dynamics of the unstable modes and the aforementioned stabilization mechanism,
but provides a very good model for the coarse GD behavior over longer optimization timescales.

Perhaps closer to our model, the Edge flow \citep{marion2026edge} and Rod flow \citep{regis2026rod,regis2026adam} ---of which we were initially unaware---were developed concurrently with the first draft of this paper.
We will detail the differences between all these flows in \cref{subsec:other-models}. 
Our model has the advantage of being mathematically analyzable while being faithful to gradient descent at the Edge of Stability. 

\paragraph{Mean-field descriptions.} 
Mean-field limits describe the training dynamics of wide two-layer networks by transport
equations for neuron distributions
\citep{mei2018mean,chizat2018global,rotskoff2018parameters,mei2019meanfield}.
Yet, these models concern non-oscillating optimization regimes.
We show that a similar overparametrized limit can be obtained for some oscillatory regimes (namely the ``flat landscape limit''; \cref{prop:mf-limits}), yielding a partial differential equation on the space of averaged parameters and their oscillation covariance factors (see \cref{sec:kinetic}).
Importantly, GD oscillations need not survive for long times as width increases in this regime.

\paragraph{Main contributions.} We derive our new model from gradient descent over $O(\eta^{-1})$ iterations under sharp-valley assumptions (\cref{thm:mf-sharp-valley-validity}). We then classify stationary covariances and their linear stability in the locally convex regime (\cref{prop:mf-locally-convex-stability}), and quantify the relation of our model to the Rod flow \citep{regis2026rod} near Edge of Stability (\cref{prop:equivalence}).
Finally, for a bounded-kernel two-layer model, we obtain a kinetic PDE which we rigorously derive (\cref{thm:kinetic-mean-field-limit}) and analyse its Wasserstein gradient-flow structure and conditional long-time behavior (\cref{thm:kinetic-target-convergence}). 
Experiments test the finite-time predictions and limitations of the mean--fluctuation model, including cases with several near-EoS Hessian modes, which are the most difficult to approximate. 

\section{Mean--fluctuation dynamics}
\label{sec:model}

At the edge of stability, gradient descent iterates can be decomposed into a slow mean and fast oscillations. The \emph{mean--fluctuation} (MF) dynamics couples the slow covariance of these fast oscillations to the mean trajectory.

\subsection{Model derivation, stability analysis, and landscape asymptotics}
\label{subsec:mf-model}

We start from the gradient descent (GD) update rule 
$$\tilde \theta_{t+1} = \tilde \theta_t - \eta \nabla E(\tilde \theta_t),$$ 
with learning rate $\eta>0$. 
Since we are interested in the slow mean and its coupling to the slow covariance of the fluctuations, we introduce an even window length $m\ge2$, and define, for $t\ge m-1$,
\begin{equation}\label{eq:windowedQ}
    \theta_t := \frac{1}{m} \sum_{j=0}^{m-1} \tilde \theta_{t-j}, \quad
    \Sigma_t := \frac{1}{m} \sum_{j=0}^{m-1} (\tilde \theta_{t-j} - \theta_t)(\tilde \theta_{t-j} - \theta_t)^T.
\end{equation}
For any family $(a_j)_{j=0}^{m-1}$, write
$
    \displaystyle \langle a\rangle_m:=\frac{1}{m}\sum_{j=0}^{m-1}a_j
$
and define the fast oscillations by
$
    \xi_{t,j}:=\tilde\theta_{t-j}-\theta_t,
$
for $j\in \{0,\ldots, m-1\}$.
Every iterate in the window can then be written \emph{exactly} as
$$
    \tilde\theta_{t-j}=\theta_t+ \xi_{t,j},
$$
where 
$$
    \langle\xi_t\rangle_m=0, \qquad 
    \langle\xi_t\xi_t^T\rangle_m=\Sigma_t.
$$
Since each $\tilde\theta_{t-j}$ advances to
$\tilde\theta_{t+1-j}$ under GD, we have exactly
\begin{align}
    \theta_{t+1}
    &=\theta_t-\eta\, \left\langle
        \nabla E(\theta_t+ \xi_t)
      \right\rangle_m,
    \label{eq:exact-mf-mean}\\
    \xi_{t+1,j}
    &=\xi_{t,j}- \eta\left(
        \nabla E(\theta_t+ \xi_{t,j})
        -\left\langle\nabla E(\theta_t+ \xi_t)\right\rangle_m
    \right).
    \label{eq:exact-mf-oscillation}
\end{align}
In particular, $\langle\xi_{t+1}\rangle_m=0$ and
$\langle\xi_{t+1}\xi_{t+1}^T\rangle_m=\Sigma_{t+1}$ exactly, with
the same average over $j$. 
Taylor expanding in $\xi$ as in \cref{app:mf-derivation}, we formally arrive at the \emph{mean--fluctuation} (MF) dynamics:
\begin{equation}\label{eq:MF}
    {\begin{cases}
        \displaystyle \frac{d\theta}{dt} = - \eta \nabla E(\theta) - \frac12 \eta \nab H(\theta ) : \Sigma, \qquad H(\theta) = \nab^2 E(\theta),\\[10pt]
        \displaystyle \frac{d\Sigma}{dt} = -\eta (\Sigma H(\theta) +  H(\theta) \Sigma)+ \eta^2 H(\theta) \Sigma H(\theta),
    \end{cases}}
\end{equation}
where $A:B \equiv \tr(AB^T)$ is the Frobenius inner product (throughout the article, unsubscripted matrix norms will denote the Frobenius norm) and $(\nabla H : \Sigma)_i = \tr((\p_i H)\Sigma)$.

\medbreak\noindent
\begin{minipage}[c]{0.54\linewidth}
To account for the contribution of fluctuations to the loss, we introduce the effective free energy 
$$ F(\theta,\Sigma):=E(\theta)+\frac12 H(\theta):\Sigma,$$ 
corresponding to the second-order Taylor approximation of the window-averaged loss. The MF mean equation then reads 
$$\frac{d\theta}{dt}=-\eta\nabla_\theta F(\theta,\Sigma).$$
\end{minipage}\hfill
\begin{minipage}[c]{0.43\linewidth}
    \centering
    \includegraphics[width=\linewidth]{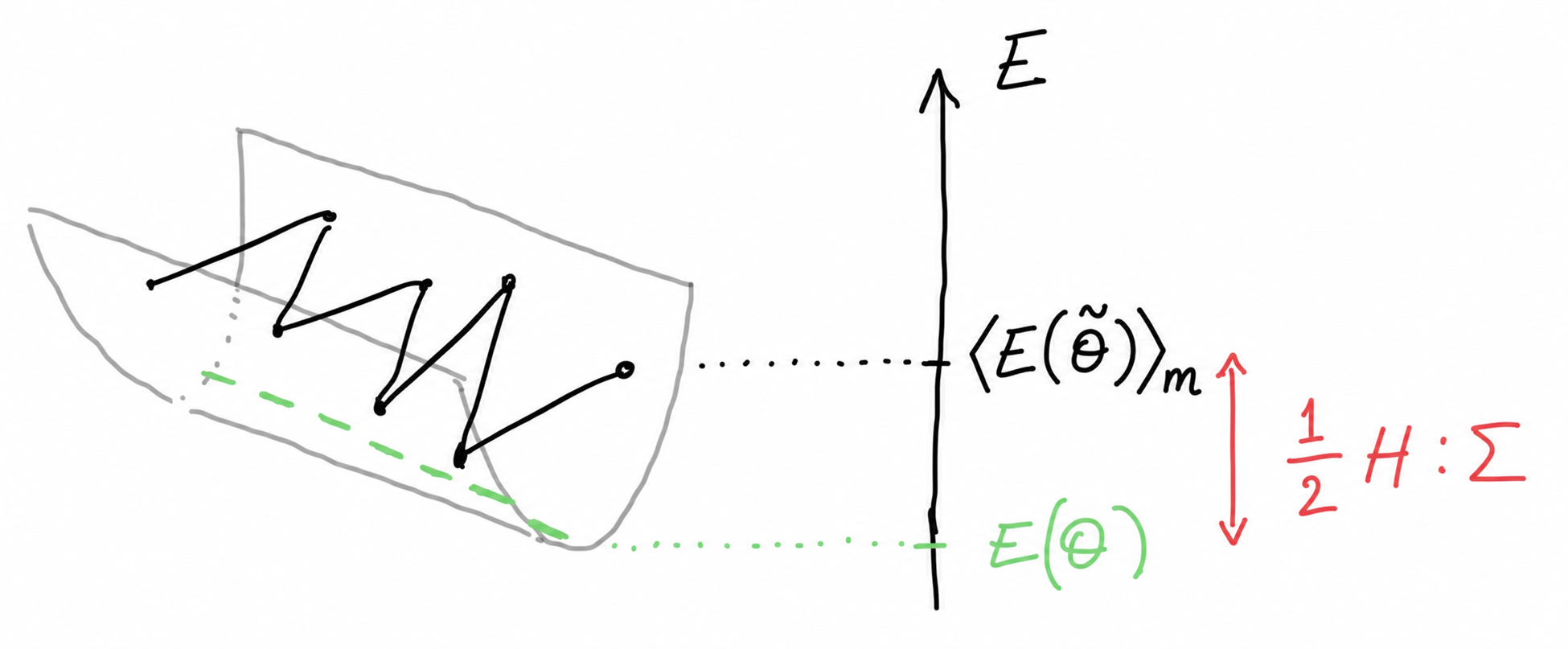}
    \small
    \captionof{figure}{{Fluctuation correction to the loss in a sharp valley.
    To second order,
    $\langle E(\tilde\theta)\rangle_m\simeq F(\theta,\Sigma)
    =E(\theta)+\frac12H(\theta):\Sigma$.
    The green level is the loss at the window mean.}}
    \label{fig:free-energy-schematic}
\end{minipage}
\medbreak

Under a more restricted ``sharp-valley'' framework (where sharp directions are orthogonal to the drift: see \cite{berthier2026timescales} which uses this framework to formally derive the Central flow), we can actually 
prove that the mean--fluctuation dynamics indeed describes gradient descent on algorithmic timescales of order $t=O(\eta^{-1})$.
To our knowledge, \cref{thm:mf-sharp-valley-validity} is the first rigorous approximation of GD at the EoS by a continuous-time model that tracks both the central trajectory and the oscillation covariance.

\begin{center}
\begin{minipage}{\linewidth}
    \centering
\begingroup
\definecolor{svmesh}{RGB}{50,106,129}
\definecolor{svsection}{RGB}{30,104,118}
\definecolor{svbottom}{RGB}{190,92,28}
\begin{tikzpicture}[
    x={(1.15cm,.12cm)},y={(.65cm,-.52cm)},z={(0cm,.9cm)},
    font=\small,line cap=round,line join=round,
    mesh/.style={draw=svmesh!50,line width=.3pt},
    axis/.style={draw=black!65,line width=.55pt,
        -{Stealth[length=1.7mm,width=1.2mm]}},
    leader/.style={draw=black!50,line width=.45pt}
]
    \path[use as bounding box] (-5.45cm,-1.85cm) rectangle (5.45cm,3.35cm);

    \foreach \j in {0,...,19} {
        \pgfmathsetmacro{\svy}{-1+.1*\j}
        \pgfmathsetmacro{\svnexty}{\svy+.1}
        \pgfmathsetmacro{\svshade}{11+13*abs(\svy+.05)}
        \fill[svmesh!\svshade,opacity=.78]
            (-3.2,\svy,{.5+.256+2.1*\svy*\svy}) --
            (3.2,\svy,{.5-.256+2.1*\svy*\svy}) --
            (3.2,\svnexty,{.5-.256+2.1*\svnexty*\svnexty}) --
            (-3.2,\svnexty,{.5+.256+2.1*\svnexty*\svnexty}) -- cycle;
    }

    \foreach \svy in {-1,-.75,-.5,-.25,0,.25,.5,.75,1} {
        \draw[mesh]
            (-3.2,\svy,{.5+.256+2.1*\svy*\svy}) --
            (3.2,\svy,{.5-.256+2.1*\svy*\svy});
    }
    \foreach \svx in {-3.2,-2.4,-1.6,-.8,0,.8,1.6,2.4,3.2} {
        \draw[mesh] plot[domain=-1:1,samples=41,variable=\s]
            (\svx,\s,{.5-.08*\svx+2.1*\s*\s});
    }
    \foreach \svy in {-1,1} {
        \draw[svmesh!70,line width=.55pt]
            (-3.2,\svy,{.5+.256+2.1*\svy*\svy}) --
            (3.2,\svy,{.5-.256+2.1*\svy*\svy});
    }
    \foreach \svx in {-3.2,3.2} {
        \draw[svmesh!70,line width=.55pt]
            plot[domain=-1:1,samples=51,variable=\s]
            (\svx,\s,{.5-.08*\svx+2.1*\s*\s});
    }

    \draw[svsection,line width=1.05pt,
        preaction={draw=white,line width=2pt,opacity=.75}]
        plot[domain=-1:1,samples=61,variable=\s]
        (2.4,\s,{.5-.08*2.4+2.1*\s*\s});
    \draw[svbottom,line width=1.4pt,
        preaction={draw=white,line width=2.6pt,opacity=.85}]
        (-3.2,0,{.5+.256}) -- (3.2,0,{.5-.256});
    \coordinate (svbottomanchor) at (.72,0,{.5-.08*.72});
    \coordinate (svnormalanchor) at (2.4,-.85,{.5-.08*2.4+2.1*.85*.85});

    \begin{scope}[x=1cm,y=1cm]
        \coordinate (svo) at (-4.72,-1.05);
        \draw[axis] (svo) -- ++(1.65,.172)
            node[right,inner sep=2pt] {$\theta_1$};
        \draw[axis] (svo) -- ++(.73,-.584)
            node[below right,inner sep=2pt] {$\theta_2$};
        \draw[axis] (svo) -- ++(0,1.55)
            node[above,inner sep=2pt] {$E_\eta$};

        \node[anchor=north,text=svbottom,align=center] (svlabel)
            at (.25,-.83) {valley bottom\\$\theta_2=0$};
        \draw[leader] (svlabel.north) -- (svbottomanchor);
        \node[anchor=south west,text=svsection] (svnormal)
            at (3.04,2.88) {normal section};
        \draw[leader] (svnormal.west) -- (svnormalanchor);
    \end{scope}
\end{tikzpicture}
\endgroup

    \captionof{figure}{{Local schematic of the sharp-valley framework
    for $d=r=1$. The bottom $\theta_2=0$ satisfies
    $\nabla_2E_\eta(\theta_1,0)=0$, while the loss may vary along the
    tangential direction $\theta_1$. The curvature in the sharp direction
    $\theta_2$ is $2/\eta+O(1)$.}}
    \label{fig:mf-sharp-valley-landscape-main}
\end{minipage}
\end{center}

{
We distinguish tangential coordinates $\theta_1$, along the valley, from
sharp coordinates $\theta_2$, across it
(\cref{fig:mf-sharp-valley-landscape-main}). The scale $a$ measures the sharp
oscillations, while $\delta=a^2/\eta$ is the scale of the accumulated
tangential displacement over $O(\eta^{-1})$ iterations.
\cref{fig:mf-pair-scales} illustrates these scales for a two-point window;
the theorem below also covers larger even windows.
}

\Needspace{4\baselineskip}
{
\begin{theorem}[Derivation of the mean--fluctuation model in a sharp valley]
\label{thm:mf-sharp-valley-validity}
Fix $C,C_0\ge1$, independently of $\eta,a$, and orthogonal coordinates
$\theta=(\theta_1,\theta_2)\in\R^d\times\R^r$.
Let $\eta,a\in(0,1)$, $\delta:=a^2/\eta\in(0,1)$, and let $m\ge2$ be even.

\smallskip\noindent\textnormal{\textbf{Landscape assumptions.}}
Assume $E_\eta\in C^4(\R^{d+r})$ and
\[
 \|D^3E_\eta\|_{\infty,\rm op}
 +\|D^4E_\eta\|_{\infty,\rm op}\le\frac C\eta.
\]
For $|\theta_1|\le\delta/C_0$, assume further that
\[
 \begin{aligned}
 &\nabla_2E_\eta(\theta_1,0)=0,
 &&\left\|D_2^2E_\eta(\theta_1,0)-\frac2\eta\Id_r\right\|_{\rm op}\le C,\\[3pt]
 &|\nabla_1E_\eta(\theta_1,0)|\le C\delta,
 &&\|D_1^2E_\eta(\theta_1,0)\|_{\rm op}\le C.
 \end{aligned}
\]

\smallskip\noindent\textnormal{\textbf{Initial data and compared dynamics.}}
Start GD with $|\tilde\theta_{1,0}|\le\delta/C_0$, and assume its first pair satisfies
\[
 \frac{\tilde\theta_{1,0}+\tilde\theta_{1,1}}2=0,\qquad
 \left|\frac{\tilde\theta_{2,0}+\tilde\theta_{2,1}}2\right|\le Ca^2,
 \qquad |\tilde\theta_{2,0}|\le Ca.
\]
Let $(\theta_k,\Sigma_k)$ be the moments in \eqref{eq:windowedQ},
reindexed so that $k=0$ corresponds to the first complete window
$(\tilde\theta_0,\ldots,\tilde\theta_{m-1})$.
Let $(\theta(t),\Sigma(t))$ solve \eqref{eq:MF} from the common initial
state $(\theta_0,\Sigma_0)$.

\smallskip\noindent\textnormal{\textbf{Approximation on the slow timescale.}}
There exist $T_*,C_*>0$, independent of $\eta,a,m$, such that
for sufficiently small $m(\eta+\delta)$ both processes exist on
the required interval and
\begin{equation}\label{eq:mf-sharp-valley-main-error}
 \sup_{0\le t\le T_*/\eta}\left(
 \frac{|\theta_{\lfloor t\rfloor}-\theta(t)|}{\delta}
 +\frac{\|\Sigma_{\lfloor t\rfloor}-\Sigma(t)\|_F}{a^2}\right)
 \le C_*m(\eta+\delta).
\end{equation}
\end{theorem}
}
\begin{figure}[!ht]
    \centering
\begingroup
\definecolor{mfscalegd}{RGB}{53,103,155}
\definecolor{mfscalemean}{RGB}{190,92,28}
\begin{tikzpicture}[
    x=1cm,y=1cm,font=\small,
    axis/.style={black!65,line width=.55pt,-{Stealth[length=1.6mm]}},
    guide/.style={black!40,densely dashed,line width=.4pt},
    measure/.style={black!80,line width=.55pt,
        {Stealth[length=1.15mm]}-{Stealth[length=1.15mm]}}
]
    \foreach \j in {0,...,30} {
        \coordinate (p\j) at
            ({.25+.27*\j},{.36+.002*\j+(-1)^\j*sqrt(1.44-.0048*(\j-15)^2)});
    }
    \foreach \j in {0,...,29} {
        \pgfmathtruncatemacro{\next}{\j+1}
        \coordinate (m\j) at ($(p\j)!.5!(p\next)$);
    }

    \draw[axis] (.1,0) -- (11.05,0) node[right] {$\theta_1$};
    \draw[axis] (.385,-1.1) -- (.385,1.93) node[above] {$\theta_2$};
    \node[below left,inner sep=2pt] at (.385,0) {$0$};

    \draw[mfscalegd!70,densely dashed,line width=.8pt]
        plot[domain=0:30,samples=61,variable=\j]
        ({.25+.27*\j},{.36+.002*\j+sqrt(1.44-.0048*(\j-15)^2)});
    \draw[mfscalegd!70,densely dashed,line width=.8pt]
        plot[domain=0:30,samples=61,variable=\j]
        ({.25+.27*\j},{.36+.002*\j-sqrt(1.44-.0048*(\j-15)^2)});

    \foreach \j in {0,...,29} {
        \pgfmathtruncatemacro{\next}{\j+1}
        \draw[mfscalegd!75,line width=.6pt,
            -{Stealth[length=1.1mm,width=.8mm]}] (p\j) -- (p\next);
    }
    \foreach \j in {0,...,30} {
        \fill[mfscalegd] (p\j) circle[radius=1.2pt];
    }
    \draw[mfscalemean,line width=1pt,
        preaction={draw=white,line width=2.6pt}]
        (m0) \foreach \j in {1,...,29} {-- (m\j)};
    \foreach \j in {0,...,29} {
        \fill[mfscalemean] (m\j) circle[radius=1pt];
    }

    \draw[mfscalegd,line width=1.15pt,
        -{Stealth[length=1.7mm]}] (p29) -- (p30);
    \fill[mfscalemean] (m29) circle[radius=2pt];
    \node[above left,inner sep=3pt,text=mfscalemean] at (m29) {$\theta_t$};

    \draw[guide] (m0) -- (m0 |- 0,-1.25);
    \draw[guide] (m29) -- (m29 |- 0,-1.25);
    \draw[measure] (m0 |- 0,-1.25) -- (m29 |- 0,-1.25)
        node[midway,below=3pt] {$\theta_1\sim a^2/\eta$};

    \draw[guide] (m29) -- (m29 |- 0,1.92);
    \draw[guide] (p30) -- (p30 |- 0,1.92);
    \draw[black!80,line width=.55pt] (7.8,1.84) -- (8.78,1.84);
    \draw[black!80,line width=.55pt,-{Stealth[length=1.15mm]}]
        (7.8,1.84) -- (m29 |- 0,1.84);
    \draw[black!80,line width=.55pt,-{Stealth[length=1.15mm]}]
        (8.78,1.84) -- (p30 |- 0,1.84);
    \node[above=3pt] at (8.2825,1.84) {$|\xi_1|\sim a^2$};

    \draw[guide] (p30) -- (9.2,0 |- p30);
    \draw[guide] (m29) -- (9.2,0 |- m29);
    \draw[measure] (9.2,0 |- m29) -- (9.2,0 |- p30)
        node[midway,right=5pt] {$|\xi_2|\sim a$};
    \draw[measure] (9.2,0) -- (9.2,0 |- m29)
        node[midway,right=5pt] {$\theta_2\sim a^2$};

    \node[text=mfscalegd,inner sep=2pt] (growthlabel) at (2.05,1.78)
        {growth};
    \draw[mfscalegd!80,line width=.55pt,-{Stealth[length=1.2mm]}]
        (growthlabel.south) -- (p6);
    \node[text=mfscalegd,inner sep=2pt] (dampinglabel) at (6.7,1.78)
        {damping};
    \draw[mfscalegd!80,line width=.55pt,-{Stealth[length=1.2mm]}]
        (dampinglabel.south) -- (p24);

    \draw[mfscalegd,line width=.8pt] (1.2,2.18) -- (1.65,2.18);
    \fill[mfscalegd] (1.425,2.18) circle[radius=1.2pt];
    \node[right] at (1.7,2.18) {GD iterates};
    \draw[mfscalemean,line width=1pt] (4,2.18) -- (4.45,2.18);
    \fill[mfscalemean] (4.225,2.18) circle[radius=1.2pt];
    \node[right] at (4.5,2.18) {pair means};
\end{tikzpicture}
\endgroup
    \caption{{Schematic almost two-period GD oscillations for $m=2$
    and a loss constant along the valley bottom.
    Labels indicate upper scales over $O(\eta^{-1})$ iterations.
    The dashed envelope illustrates a possible $O(a)$ change in
    $|\xi_2|$: growth followed by damping as the sharp curvature
    decreases through $2/\eta$.}}
    \label{fig:mf-pair-scales}
\end{figure}
In \cref{app:mf-sharp-valley-derivation}, we comment on the scaling and assumptions considered here and prove the theorem.
The key part of the proof is to control the remainders on the long timescale $t=O(\eta^{-1})$, which is done in \cref{lem:mf-pair-small-components}.
For this, we exploit the stabilization mechanism of gradient descent at the Edge of Stability, which can be controlled within the sharp-valley framework.
In this regime of established oscillations, the fluctuation covariance is rank one at leading order (\cref{eq:rank-one-first-order}) as observed in our MLP runs (\cref{app:mf-sharp-valley-derivation}, after \cref{cor:mf-even-window-direct-validity}). By contrast, near the onset of EoS, the window covariance is dominated by the drift, and the full covariance of MF is what allows the small critical component to grow (\cref{subsec:other-models,fig:rod-onset-floor}).

The advantage of our model \eqref{eq:MF} is that it both captures the Edge of Stability behavior and is amenable to the following mathematical analysis.
Some of the statements and all proofs of our results are given in \cref{app:mf-analysis}.

\paragraph{Well-posedness, positivity, and dissipation (\cref{prop:mf-well-posedness}).} 
Under bounded second and third derivatives of the loss, the mean--fluctuation equations \cref{eq:MF} possess a unique global solution. The covariance matrix $\Sigma_t$ is symmetric positive semidefinite for all times, and the free energy $F(\theta, \Sigma) := E(\theta) + 1/2 H(\theta): \Sigma$ is decreasing along the flow in stable regions.

\begin{remark}[Rank of the fluctuation covariance]
\label{rem:mf-rank}
    The MF dynamics does not preserve the rank of $\Sigma$. Indeed, the counter-example in \cref{app:mf-rank-proof} starts from a rank-one
    covariance which becomes full rank immediately. Rank is preserved in what we call the
    ``flat-landscape limit'' (see \cref{prop:mf-limits} and \cref{sec:kinetic} below), but not by the coupled MF system in general.
\end{remark}

\paragraph{Stability/instability of stationary states (\cref{prop:mf-locally-convex-stability}).}
We also analyse the stability of stationary states of the MF dynamics. Specifically, a stationary state is \emph{stable} if
the spectrum of the linearized operator is contained in the closed left
half-plane, \emph{strictly stable} if it is contained in the open left
half-plane, and \emph{unstable} if its linearization has an eigenvalue with positive real
part.

\Needspace{4\baselineskip}
{
\begin{theorem}[Stationary covariance and EoS stability]
\label{prop:mf-locally-convex-stability}
Let $E\in C^4$ be locally convex near a stationary MF state
$(\theta_*,\Sigma_*)$ with $\Sigma_*\ge0$. Set
\[
 H_*=H(\theta_*),\qquad V_0=\ker H_*,\qquad
 V_2=\ker(2\Id-\eta H_*),
\]
and write $P_0,P_2$ for the corresponding orthogonal projections.
\begin{enumerate}[label=\textnormal{(\roman*)},leftmargin=*,itemsep=.6em]
\item \textnormal{\textbf{Zero covariance.}}
If $\Sigma_*=0$, then $\nabla E(\theta_*)=0$.
The state is stable exactly when $0\le\eta H_*\le2\Id$, and
strictly stable exactly when $0<\eta H_*<2\Id$.

\item \textnormal{\textbf{Stationary decomposition.}}
If $\Sigma_*\ne0$, then
\[
 \operatorname{Ran}\Sigma_*\subset V_0\oplus V_2,
 \qquad P_0\Sigma_*P_2=0.
\]
Thus the fluctuations in flat and EoS directions are decorrelated,
and $(\theta_*,P_2\Sigma_*P_2)$ is also stationary.

\item \textnormal{\textbf{Stability of EoS fluctuations.}}
For a nonzero covariance supported in $V_2$, assume additionally that
\[
 \nabla_\theta^2F(\theta_*,\Sigma_*)\ge0.
\]
If all eigenvalues of $\eta H_*|_{V_2^\perp}$ lie in $[0,2)$,
the state is stable. An eigenvalue larger than $2$ produces an
exponentially growing covariance perturbation and hence instability.
\end{enumerate}
\end{theorem}
}
\begin{proof}[Sketch of proof]
    The proof uses the linearization of the stationary equations around $(\theta_*,\Sigma_*)$ and is split into three lemmas.
    \cref{lem:mf-stationary-decomposition} addresses the decomposition of $\operatorname*{Ran}\Sigma_*$ into flat and oscillating modes, which allows us to reduce the analysis for $\Sigma_*\neq 0$ to purely oscillating cases, i.e. $\operatorname*{Ran}\Sigma_*\subseteq V_2$. 
    \cref{lem:mf-zero-covariance-stability} handles the non-oscillating case $\Sigma_*=0$. 
    Finally, \cref{lem:mf-pure-eos-stability} concerns the oscillating states $\Sigma_*\ne 0$, for which we introduce a Lyapunov function to show linear stability.
\end{proof}

In \cref{app:mf-rank-proof}, we show that under a modified Polyak-\L ojasiewicz inequality, the free energy $F(\theta, \Sigma) := E(\theta) + 1/2H(\theta):\Sigma$ converges at an exponential rate.

\paragraph{Gradient-flow and flat-landscape limits (\cref{prop:mf-limits}).}

There are two interesting asymptotic regimes for the MF dynamics. The first one is the gradient-flow limit, where the learning rate $\eta$ is small and the covariance is of order $\eta$ (which is expected from scaling analysis, see \cref{app:mf-scaling}). The second one is the flat-landscape limit, where the energy is scaled by a small parameter $\eps$. 
In both cases, the MF dynamics reduces to a simpler system of equations. 
In the rescaled time $\tau=\eta\eps t$, the flat-landscape limit is the
gradient flow of $F$ for the product of the
Euclidean metric on $\theta$ and the Bures--Wasserstein metric on covariance
matrices \citep{malago2018wasserstein}.  Its covariance evolves as
$\Sigma_\tau=P_\tau\Sigma_0P_\tau^T$ with $P_\tau$ invertible, and hence
preserves rank throughout its existence interval.

\subsection{Comparison to other models}
\label{subsec:other-models}
In this section, we compare our mean--fluctuation model to three other models: the Central flow \citep{cohen2024centralflow}, the Edge flow \citep{marion2026edge} and the Rod flow \citep{regis2026rod}. 
We focus on the Rod flow model, which we show is the only one among these models, together with MF, to capture the \emph{transient} fluctuation covariance dynamics when several Hessian eigenvalues are unstable. 
In the asymptotic validity regime of the mean--fluctuation model, we show that the covariance evolutions of Rod flow and our model are equivalent when restricted to rank-one $\Sigma$ (\cref{prop:equivalence}).
We also provide an explanation for the comparative robustness of the mean--fluctuation model to Rod flow's failure to grow oscillations when initialized at the onset of EoS.

\paragraph{Central flow \citep{cohen2024centralflow}.}

The Central flow is a coarser description than our model.
\citet{cohen2024centralflow} decompose the discrete trajectory $\tilde \theta_k$ into a slow central path ${\theta}_k$ and fast oscillations $\d\theta_k$. The mean path $\theta$ satisfies the same equation as for our mean--fluctuation model, while 
their covariance matrix $\Sigma$ is hard-coded through the hypothesized relation
$
    \operatorname*{Ran}\, \Sigma \subseteq \text{Ker}\, (\nab^2 E(\theta) - 2/\eta \, \text{Id})
$
at each time, which assumes the instantaneous relaxation to an equilibrium value that balances the ``progressive sharpening'' with the ``curvature-related drift''. 
In particular, the Central flow is not designed to capture the transient, self-stabilizing, fluctuations of the loss and unstable Hessian eigenvalues, contrary to our mean--fluctuation model. 
As shown in \cref{fig:cnn-mf-central-relaxation}, our model can nevertheless relax towards the Central flow on longer optimization timescales,
which is consistent with the idea that Central flow is a coarser description of gradient descent. 

\paragraph{Edge flow \citep{marion2026edge}.} The Edge flow model is another model for the Edge of Stability dynamics whose equations couple $(\theta, u, x)$, with $u$ the oscillation direction and $x$ its amplitude. 
Contrary to the Central flow, Rod flow and our MF model, its derivation is not exactly based on first principles, but rather on the assumption that the oscillation direction follows the top eigenvector of the Hessian by a Rayleigh quotient gradient flow. 
In our experiments, Algorithm~1 of \cite{marion2026edge} seems unable to capture the Edge of Stability dynamics when several Hessian eigenvalues reach the stability threshold. 

\paragraph{Rod flow \citep{regis2026rod,regis2026adam}.}

The derivation of Rod flow can initially be viewed as the restriction of our model to the case of a window of size $m=2$ (therefore restricting the derivation of their model to a rank-one covariance). 
In the numerical implementation of the Rod flow, the authors relax the rank-one constraint on $\Sigma$ by instead \emph{defining} the rod extent $\xi$ as the top direction in the SVD of $\Sigma$. This allows them to obtain stronger numerical results, but it is not clear how to justify this from the point of view of the derivation of the model. 
In contrast, our model represents fluctuations in multiple directions from the start, but makes an additional closure assumption which makes it less precise than Rod flow in general, yet mathematically analyzable.
The Rod flow equations are given by
\begin{align}\label{eq:rawRod}
    \frac{d\theta}{dt} &= - \frac{\eta}{2}\big( \nabla E(\theta + \xi) +  \nabla E(\theta - \xi) \big) - \frac{\eta^2 }{8} \big(H(\theta+\xi) + H(\theta -\xi)\big)\big( \nabla E(\theta+\xi) + \nabla E(\theta -\xi) \big), \nonumber \\
    \frac{d\Sigma}{dt} &= \frac{\eta^2}{4}\big(\nabla E(\theta+ \xi) \otimes  \nabla E(\theta+\xi) + \nabla E(\theta - \xi) \otimes  \nabla E(\theta-\xi)\big) - 2\Sigma,
\end{align}
where $\xi$ is obtained from the SVD of $\Sigma$.
Arguably, our model is simpler to analyse than the Rod flow equations above, but this is because they are not exactly the right models to compare yet.
Taylor expansion in $\xi$ leads to the following \emph{coarse Rod flow}:
\begin{equation}\label{eq:coarseRod}
    \frac{d\theta}{dt} = - \eta\nabla E(\theta) - \frac{\eta}{2} \nab H(\theta) : {\xi\xi^T}, \qquad
    \frac{d\Sigma}{dt} = \frac{\eta^2}{2} H (\theta){\xi\xi^T}H(\theta) - 2\Sigma.
\end{equation}
The following proposition quantifies this reduction and shows that the
coarse Rod and MF vector fields, denoted by $\mathcal V_{\rm R}$ and
$\mathcal V_{\rm MF}$, are equivalent for small fluctuations around the EoS for rank-one oscillations. 
The proof of this proposition is in \cref{app:mf-rod-comparaison}.

\Needspace{4\baselineskip}
{
\begin{proposition}[Coarse reduction of Rod flow and comparison with MF]
\label{prop:equivalence}
\label{prop:rod-raw-coarse-remainder}

\textnormal{\textbf{(i) Raw-to-coarse reduction.}}
Fix $T>0$ and $\eta,a\in(0,1)$ with $\delta:=a^2/\eta\in(0,1)$.
Let $(\theta_t,\Sigma_t)$ solve \eqref{eq:rawRod} on $[0,T/\eta]$,
with $E=E_\eta$ of class $C^5$. Assume, uniformly along this trajectory,
\[
 |\nabla E_\eta(\theta_t)|\le C\delta,\qquad
 C^{-1}a^2\le|\xi_t|^2\le Ca^2,\qquad
 \eta\|D^jE_\eta\|_{\rm op}\le C\quad(2\le j\le5),
\]
where $C\ge1$ is independent of $\eta,a,t$, and the derivative norms
are taken on the Taylor segments $\theta_t+s\xi_t$, $|s|\le1$.
Use this same principal extent $\xi_t$ in the coarse field
$\mathcal V_{\rm R}$. For all sufficiently small $\eta$,
\begin{equation}
 \begin{aligned}
    \int_0^{T/\eta}
      \left|\frac{d\theta_t}{dt}
        -\mathcal V_{\rm R}^\theta(\theta_t,\Sigma_t)\right|\,dt
      &\le C_T\delta,\\[4pt]
    \frac1{a^2}\int_0^{T/\eta}
      \left\|\frac{d\Sigma_t}{dt}
        -\mathcal V_{\rm R}^\Sigma(\theta_t,\Sigma_t)\right\|_F\,dt
      &\le C_T\delta.
 \end{aligned}
\end{equation}
Here $C_T$ depends only on $C,T$. These are integrated vector-field
defects along the raw solution: the mean defect is absolute, whereas
the covariance defect is normalized by $a^2$.

\smallskip\noindent\textnormal{\textbf{(ii) Comparison with MF.}}
At any common state $(\theta,\Sigma)$ with $\Sigma\ge0$,
put $D=\eta H(\theta)-2\Id$. Then
\[
    \mathcal V_{\rm MF}^\Sigma-\mathcal V_{\rm R}^\Sigma
    =\frac12D\Sigma D+\frac{\eta^2}{2}H(\theta)(\Sigma-\xi\xi^T)H(\theta).
\]
If $P_\Sigma$ is the orthogonal projection onto $\operatorname{Ran}\Sigma$,
then $\|DP_\Sigma\|_{\rm op}\le\eps$ implies
\[
    \|\mathcal V_{\rm MF}^\Sigma-\mathcal V_{\rm R}^\Sigma\|_F
    \le\frac{\eps^2}{2}\|\Sigma\|_F
       +\frac{(2+\eps)^2}{2}\|\Sigma-\xi\xi^T\|_F.
\]
\end{proposition}
}

\paragraph{The onset failure mode of Rod flow.}

When initialized substantially before the Edge of Stability regime, the Rod flow fails to stabilize at EoS and follows the gradient flow instead \citep{marion2026edge}.
In \cref{app:onset-failure}, we reproduce this failure even starting at
$\eta\lambda_1\simeq1.999$, whereas MF recovers the growing fluctuations.
We provide an explanation which is that, at the onset of EoS, the mean--fluctuation dynamics can actually grow oscillations along the top Hessian eigenvector even when this covariance component is not dominant,
contrary to Rod flow, which retains a rank-one flavor by defining the extent $\xi$ from the top SVD direction of $\Sigma$.
\providecommand{\rodonsetgeometrylabel}{fig:rod-onset-geometry}
\begin{figure}[htbp]
    \centering
    \includegraphics[width=\textwidth]{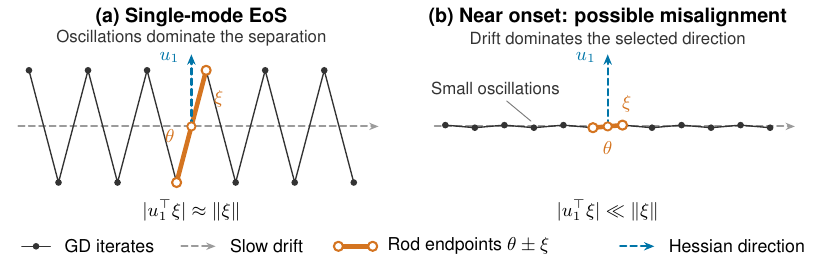}
    \caption{Rod's onset failure. The selected displacement is
    $\xi=\sqrt{s_1}v_1$, where $(s_1,v_1)$
    is the leading eigenpair of $\Sigma$.
    (a) At single-mode EoS, the oscillations are along the leading Hessian eigenvector $u_1$
    \citep[see also Fig.~2]{marion2026edge}.
    (b) Near onset, the selected direction can instead be nearly orthogonal
    to $u_1$, even if the full covariance contains smaller fluctuations
    along $u_1$. MF can amplify these smaller fluctuations.}
    \label{\rodonsetgeometrylabel}
\end{figure}

\paragraph{Raw counterpart of the MF dynamics: slow evolution hypothesis of the covariance.}
Finally, we compare MF to the following raw version of our model (dubbed \emph{raw mean--fluctuation dynamics}):
\begin{equation}\label{eq:rawMF-previous-main}
    \frac{d\theta}{dt} = -\eta \nabla E(\theta) - \frac{\eta}{2} \nab H(\theta) : \Sigma_t, \quad
    \Sigma_{t+1} = (\Id - \eta H(\theta_t))\Sigma_t(\Id - \eta H(\theta_t)).
\end{equation}
This will numerically test the closure assumption together with the continuous approximation of the mean to establish the importance of the continuous approximation of the covariance evolution. 
In the numerical experiments, this model is more accurate than \cref{eq:MF}. Agreement between the raw and coarse models does not, however, imply agreement with GD:
\cref{eq:MF} can be closer to \cref{eq:rawMF-previous-main} than to gradient descent over long timescales. Indeed, for a common prescribed $C^1$ Hessian $H(t)\ge0$, suppose a fixed invariant subspace $V$ supports $\Sigma_0$, with $\|(\eta H-2\Id)|_V\|_{\rm op}\le\varepsilon$ and $\|\dot H|_V\|_{\rm op}\le L$. If $\eta\le\varepsilon^2$, \cref{prop:mf-validity} bounds the raw/coarse covariance gap by $O(\varepsilon)\|\Sigma_0\|$ over times $O(\varepsilon^{-1})$ as $\varepsilon\to0$, 
while a toy example in \cref{app:mf-closure-validity} shows that the relative covariance error against GD can already be of order one on this $O(\varepsilon^{-1})$ timescale.

\section{Experiments}
\label{sec:experiments}

We validate our model, and compare it to the others, on three architectures (MLP, CNN, and ResNet) trained on a subset of CIFAR-10 (around 3,000--8,000 images).
For each plot, the flows are initialized when the first (or second) Hessian eigenvalue reaches the edge of stability.
Every flow starts from the same stored GD window. We compare free energy $F$ with GD loss, Hessian eigenvalues with those at the GD-window center, and the relative Euclidean center error
$e_\theta(p)=\|\theta_p-\overline\theta_{\rm GD}\|_2/\|\overline\theta_{\rm GD}\|_2,$
where $\overline\theta_{\rm GD}$ is the window mean. \cref{app:experiments} details the protocol.

\begin{figure}[ht!]
    \centering
    \includegraphics[width=\textwidth]{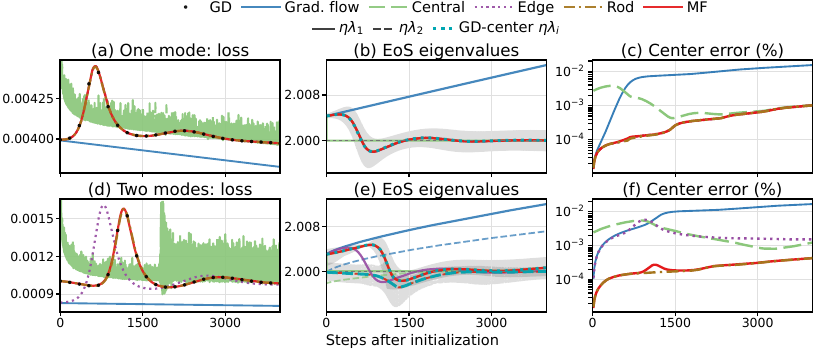}
    \caption{MLP on CIFAR-10, $\eta=0.002$.}
    \label{fig:mlp-eos}
\end{figure}

The reference runs show close agreement when there is one unstable eigenvalue, for which we give the results in \cref{app:architecture-eig1} (see also \cref{fig:intro-one-mode-loss}). MF and Rod have similar MLP errors when there are two unstable modes, while Rod is better for two-mode CNN and ResNet experiments. 
Raw MF has a similar validity timescale to Rod on the two-mode CNN experiment. 
In \cref{app:architecture-closure}, we plot the closure diagnostics for all these runs. 
Additional runs can be found in \cref{app:mlp-robustness}.
Since the plotted curves are small-rank numerical approximations, we perform numerical validation in \cref{app:solver}.  

\begin{figure}[ht!]
    \centering
    \includegraphics[width=\textwidth]{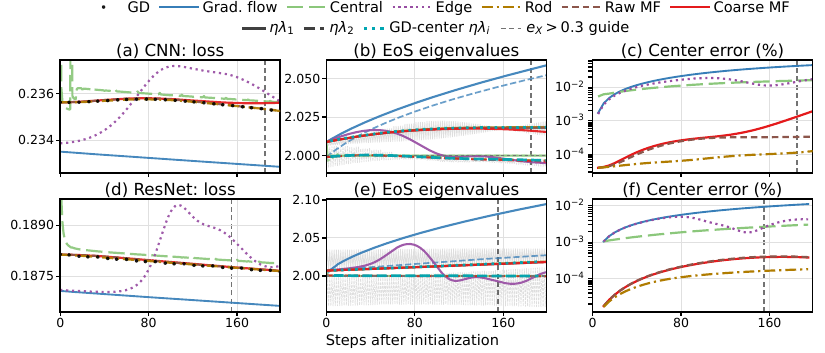}
    \caption{Two-mode CNN ($\eta=0.002$) and ResNet ($\eta=0.004$) on
    CIFAR-10. The vertical dashed line corresponds to accumulated covariance error above $0.3$ (see \cref{app:architecture-closure}).}
    \label{fig:architecture-eig12}
\end{figure}

\begin{figure}[ht!]
    \centering
    \includegraphics[width=\textwidth]{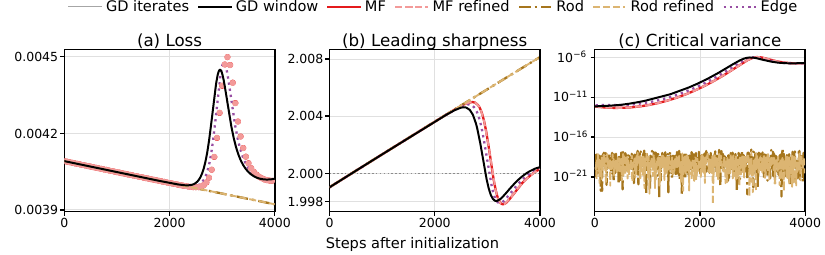}
    \caption{MLP initialized near $\eta\lambda_1=1.999$, $\eta= 0.002$. Writing $u_1$ for the leading
    GD-center Hessian eigenvector, the critical variance is defined as $|u_1^T\xi|^2$, where $\xi$ is obtained from the SVD of $\Sigma$ for Rod, and $u_1^T\Sigma u_1$ for Edge and MF. Full protocol: \cref{app:onset-failure}.}
    \label{fig:rod-onset-floor}
\end{figure}
\cref{fig:rod-onset-floor} supports our claim that MF can grow oscillations along the top Hessian eigenvector even when this covariance component is not dominant, while Rod fails to do so because of its remaining rank-one flavor. The full discussion is available in \cref{app:onset-failure}.
\par\medskip
\noindent\begin{minipage}{\textwidth}
\noindent\begin{minipage}[c]{0.35\textwidth}
In the longer CNN run, the covariance errors accumulate (\cref{app:architecture-closure}).
Nevertheless, an interesting observation is that the MF and Central centers eventually become close to each other (\cref{fig:cnn-mf-central-relaxation}).
\end{minipage}\hfill
\begin{minipage}[c]{0.63\textwidth}
    \centering
    \includegraphics[scale=0.8]{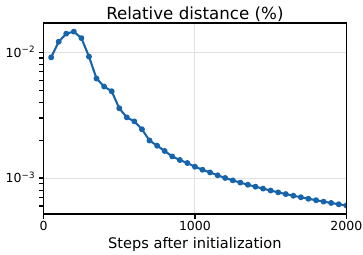}
    \captionof{figure}{MF--Central relative distance for the two-mode CNN.}
    \label{fig:cnn-mf-central-relaxation}
\end{minipage}\par
\end{minipage}\par

\section{Kinetic theory for wide two-layer networks}
\label{sec:kinetic}

Practical machine learning architectures are overparametrized, meaning that the number of parameters is very large with respect to the number of training samples.
We therefore study our model for wide two-layer networks in the mean-field scaling. 
Specifically, we consider $N$ neurons with weights $(\theta_1,\ldots,\theta_N) \in (\R^d)^N$ and the associated empirical measure $ \rho_N := N^{-1}\sum_{i=1,\ldots,N}\delta_{\theta_i}$. We apply our model to tasks of the form
$
    E(\rho) = \E_{(x,y) \sim \mathcal{D}} \left[ \ell\left(y, \int \phi(\theta, x) d\rho(\theta)\right) \right],
$
where $\phi(\theta, x)$ denotes the output of a single neuron with weights $\theta$ and input $x$, and $\mathcal{D}$ is the data distribution.
Rescaling time by a factor $N/\eta$, \cref{eq:MF} becomes
\begin{equation}\label{eq:particleMF}
    \begin{cases}
    \displaystyle \frac{d\theta_i}{dt} = -  N \nab_{\theta_i} E(\rho_N) - \hal N \nab_{\theta_i} \nab^2 E(\rho_N) : \Sigma, \\[10pt]
    \displaystyle\frac{d\Sigma_{i,j}}{dt} = -  N (\Sigma  \nab^2 E(\rho_N))_{i,j} -  N (\nab^2 E(\rho_N)  \Sigma)_{i,j} + \eta N (\nab^2 E(\rho_N) \Sigma \nab^2 E(\rho_N))_{i,j}.
    \end{cases}
\end{equation}
As we will show in \cref{thm:kinetic-mean-field-limit}, the term quadratic in the Hessian in the covariance equation is of order $O(N^{-1})$ when $\eta$ is fixed. Therefore, this term disappears as $N\to\infty$, which is the flat landscape regime (\cref{prop:mf-limits}). 
It remains to show that, starting from $\Sigma^N=V^N(V^N)^T$ with a rank $r$ independent of $N$, on every fixed
time interval, the full system remains $O(N^{-1})$-close in block norm to the
rank-$r$ system obtained by omitting the quadratic term.  The block rows
$v_i^N\in\R^{d\times r}$ of $V^N$ then live in a fixed space, so their empirical
law has an ordinary mean-field limit. 

To fix ideas, we can write the equation in the case of a least-squares loss $\ell(r) = r^2/2$, in the teacher-student framework where $ y = \int \phi(\theta, x) \, d\mu(\theta)$ for some target distribution $\mu \in \mathcal{P}(\R^d)$. In the reproducing kernel Hilbert space scenario where we have a translation-invariant kernel $\g: \R^d\to \R$ such that $ \g(\theta-\theta') = \E_{\mathcal{D}}\big[ \phi(\theta,x)\phi(\theta',x)\big]$,
the risk functional now reads 
$
    E(\rho) = \hal \int \g\ast (\rho - \mu) \, d(\rho-\mu),
$
which is one half of the squared maximum mean discrepancy between $\rho$ and $\mu$ \citep{arbel2019maximum}.
Abusing notation, we also write $\mu$ for the lifted measure $\mu(d\theta)\otimes \d_{0}(d v)$ in $\mathcal{P}(\R^d \times \R^{d\times r})$. We rigorously obtain (see \cref{thm:kinetic-mean-field-limit}) the closed-form equation
\begin{equation}\label{eq:MMD}
    {\begin{cases}
        &\displaystyle \frac{\p f}{\p t} - \divergence_\theta (f\nab\g\ast(\rho - \mu) )- \frac{1}{2} \divergence_\theta (f\nab_\theta (\nab^2 \g : vv^T) \ast(f - \mu) ) \\[3pt]
       & \displaystyle-  \divergence_v (f(\nab^2\g \, v)\ast (f-\mu)) = 0, \qquad\rho := \int f\, dv. 
    \end{cases}}
\end{equation}
We could derive a more general PDE accounting for different losses $\ell$, tasks $E$ and data structures.
From the point of view of analysis, this equation is the gradient flow for the
Wasserstein-2 distance on the space $\mathcal{P}_2(\R^d\times \R^{d\times r})$, associated with the macroscopic free energy
\begin{equation}\label{eq:kinetic-free-energy}
    \mathcal{F}(f):= \frac{1}{2}\int\g\ast (\rho-\mu )\, d(\rho-\mu) + \frac{1}{4}\int (\nab^2\g : v v^T) \ast (f-\mu )\, d(f-\mu).
\end{equation}
The fluctuation term has no sign, so $\mathcal F$ is
generically unbounded from below. In
\cref{app:kinetic-variational} we prove a uniform lower bound under a
velocity second-moment constraint.  After adding the uniform integrability
needed, we also obtain minimizers and identify their
small-variance limit.

\paragraph{Well-posedness and regularity (\cref{thm:well-posedness}).} 
For even kernels $\g\in C^{3,1}_b$, target $\mu \in \mathcal{P}(\R^d)$ and initial data $f_0\in \mathcal{P}_2(\R^d\times \R^{d\times r})$ compactly supported in the $v$ variable, the equation \eqref{eq:MMD} has a unique global solution, for which the free energy \eqref{eq:kinetic-free-energy} decreases along the flow. It also propagates the regularity of the initial data, under suitable regularity assumptions on the kernel (\cref{prop:regularity}).
The key is to control fluctuations: we prove that $f_t$ remains compactly supported in $v$ along the flow. 

\Needspace{4\baselineskip}
{
\begin{theorem}[Conditional convergence and rate]
\label{thm:kinetic-convergence-main}
Let $\g\in C^{3,1}(\mathbb T^d)$ be even and satisfy
\[
 \widehat\g(k)\ge0\quad\text{for all }k,\qquad
 \widehat\g(k)>0\quad\text{for }k\ne0.
\]
Let $\mu\in\mathcal P(\mathbb T^d)$, and let $f_t$ solve \eqref{eq:MMD}
on $\mathbb T^d\times\R^{d\times r}$ from initial data
$f_0\in\mathcal P_2$ compactly supported in $v$.
Write $\rho_t$ for its positional marginal and set
\[
 M_2(t):=\int|v|^2\,\dd f_t.
\]
Assume that, for some $C_0>0$,
\[
 \sup_{t\ge0}\|\rho_t-\mu\|_{\dot H^{-1}(\rho_t)}\le C_0,
\]
where the weighted dual norm is defined in \eqref{eq:kinetic-weighted-dual-norm}.
\begin{enumerate}[label=\textnormal{(\roman*)},leftmargin=*,itemsep=.6em]
\item \textnormal{\textbf{Convergence.}} If $M_2(t)\to0$, then
\[
 W_2(f_t,\mu\otimes\delta_0)\longrightarrow0.
\]
\item \textnormal{\textbf{Rate.}} If, for some $A,\alpha>0$,
$M_2(t)\le A(1+t)^{-\alpha}$, then
\[
 E(\rho_t)\le C(1+t)^{-\min\{1/2,\alpha/3\}},
\]
where $C$ depends only on $\g,C_0,A,\alpha$ and $E(\rho_0)$.
\end{enumerate}
\end{theorem}
}
\begin{proof}[Sketch of proof]
    The proof is given in \cref{thm:kinetic-target-convergence,cor:kinetic-mmd-rate}. The uniform bound in $\dot H^{-1}(\rho_t)$ was already assumed in \cite{arbel2019maximum} for classical flows. 
\end{proof}

\Needspace{4\baselineskip}
{
\begin{theorem}[Kinetic limit of mean--fluctuation dynamics]
\label{thm:kinetic-mean-field-limit}
Fix $\eta>0$ and rank $r$ independently of $N$.
Let $\g\in C_b^4(\R^d)$ be even, $\mu\in\mathcal P(\R^d)$, and
$f_0\in\mathcal P_2(\R^d\times\R^{d\times r})$ with $|v|\le M$ on its support.

\smallskip\noindent\textnormal{\textbf{Initial data.}}
Initialize the full MF system \eqref{eq:particleMF} with
\[
 \Sigma^N(0)=V^N(0)V^N(0)^T,\qquad |v_i^{N,0}|\le M,
 \qquad f_N^0=\frac1N\sum_{i=1}^N\delta_{(\theta_i^{N,0},v_i^{N,0})},
\]
where $v_i^{N,0}$ are the block rows of $V^N(0)$.

\smallskip\noindent\textnormal{\textbf{Fixed-rank comparison system.}}
Let $(\bar\theta^N,V^N)$ use the same mean equation with
$\Sigma=V^N(V^N)^T$ and the factor equation
\[
 \partial_tV^N=-NH_N(\bar\theta^N)V^N,
\]
where $H_N$ is the Hessian of the $N$-neuron loss.
Use the same initial positions and factors as for the full system, and set
\[
 \bar f_N(t)=\frac1N\sum_{i=1}^N\delta_{(\bar\theta_i^N(t),v_i^N(t))},
 \qquad \|S\|_{\max}:=\max_{i,j}\|S_{ij}\|,
\]
with $d\times d$ covariance blocks $S_{ij}$.

\smallskip\noindent\textnormal{\textbf{Finite-time estimates.}}
If $f_t$ solves \eqref{eq:MMD} from $f_0$, then, for every fixed $T>0$,
\[
 \begin{aligned}
 \sup_{0\le t\le T}W_1(\bar f_N(t),f_t)
     &\le C_T W_1(f_N^0,f_0),\\[5pt]
 \sup_{0\le t\le T}\left(\max_i|\theta_i^N-\bar\theta_i^N|
       +\|\Sigma^N-V^N(V^N)^T\|_{\max}\right)
     &\le \frac{C_T}{N},
 \end{aligned}
\]
where $C_T$ is independent of $N$.
Consequently, the full positional law
$\rho_N^{\rm full}=N^{-1}\sum_i\delta_{\theta_i^N}$ and the positional
marginal $\rho_t$ of $f_t$ satisfy
\begin{equation}\label{eq:kinetic-full-marginal-limit}
    \sup_{t\le T}W_1(\rho_N^{\rm full}(t),\rho_t)
    \le C_T\bigl(W_1(f_N^0,f_0)+N^{-1}\bigr).
\end{equation}
\end{theorem}
}

\begin{proof}[Sketch of proof]
    The essential step is to reduce to a fixed-dimensional space for $V^N$, proving that the quadratic Hessian term in \eqref{eq:particleMF} is $O(N^{-1})$. This is obtained by a stability argument: it replaces a $dN\times dN$ state by factors in an $N$-independent phase space. The proof is in \cref{app:kinetic-limit-proof}.
\end{proof}

\section{Discussion}
\label{sec:discussion}

{The mean--fluctuation dynamics we introduced is rich enough to capture the Edge of Stability feedback mechanism between progressive sharpening and curvature-related drift (\cref{fig:intro-one-mode-loss,fig:mlp-eos,fig:architecture-eig12}), while still being mathematically analyzable.
We rigorously derived it from gradient descent (\cref{thm:mf-sharp-valley-validity}) 
and analyzed its stationary points and their linear stability, which can be related to the Edge of Stability threshold (\cref{prop:mf-locally-convex-stability}). 
In the overparametrized limit, we rigorously derived the first kinetic description of wide two-layer models accounting for oscillations in the kinetic variable (\cref{thm:kinetic-mean-field-limit,thm:kinetic-convergence-main}).
The experiments further identify where this covariance description follows GD and where it loses accuracy (\cref{sec:experiments,app:mlp-robustness,app:onset-failure}).}

\paragraph{Beyond sharp valleys.}
Our \cref{thm:mf-sharp-valley-validity} assumes sharp directions orthogonal to the drift, as in the formal Central-flow derivation of \cite{berthier2026timescales}. Extending MF beyond this setting remains open: there is no general reason to discard higher-order covariance terms on timescales of order $\eta^{-1}$.

\paragraph{Hydrodynamic theory at the Edge of Stability.}
Our kinetic PDE describes fluctuations below the stability threshold. Deriving such a PDE in the Edge of Stability regime remains an open question.

\paragraph{Stochastic gradient descent.} 
Extending our model to small-batch gradient descent would be valuable.  

\section*{Acknowledgments}
We thank Léna\"ic Chizat for bringing Rod flow to our attention while we were preparing the first draft of this article.
This article also benefited from discussions with Gabriel Peyré, whom we thank for his careful reading of earlier versions. 
We also thank Jeremie Klinger and Matthew Rosenzweig for discussions on related topics, and
IHES for providing access to its GPUs. 

This work has been supported by the European Research Council (ERC project Wolf).

\section*{AI use statement}
LLMs (GPT-5.5, 5.6 Sol and 6.0 Astra) were used to write and run numerical code, to assist with literature search, to check proofs, and to draft parts of proofs (\cref{thm:mf-sharp-valley-validity,prop:equivalence}, the case $\Sigma_\star\neq0$ of \cref{prop:mf-locally-convex-stability}, a first version of \cref{prop:regularity}, and \cref{thm:kinetic-convergence-main}) and of the text. 
An AI image-editing tool was used to correct the labels in \cref{fig:free-energy-schematic}.
The author verified all statements, proofs and numerical results and takes full responsibility for the content.

\bibliography{references}

@inproceedings{cohen2021eos,
  title={Gradient Descent on Neural Networks Typically Occurs at the Edge of Stability},
  author={Cohen, Jeremy M. and Kaur, Simran and Li, Yuanzhi and Kolter, J. Zico and Talwalkar, Ameet},
  booktitle={International Conference on Learning Representations},
  year={2021},
  url={https://arxiv.org/abs/2103.00065}
}

@inproceedings{arora2022eos,
  title={Understanding Gradient Descent on the Edge of Stability in Deep Learning},
  author={Arora, Sanjeev and Li, Zhiyuan and Panigrahi, Abhishek},
  booktitle={Proceedings of the 39th International Conference on Machine Learning},
  series={Proceedings of Machine Learning Research},
  volume={162},
  pages={948--1024},
  year={2022},
  url={https://proceedings.mlr.press/v162/arora22a.html}
}

@inproceedings{damian2023self,
  title={Self-Stabilization: The Implicit Bias of Gradient Descent at the Edge of Stability},
  author={Damian, Alex and Nichani, Eshaan and Lee, Jason D},
  booktitle={International Conference on Learning Representations},
  year={2023}
}

@article{zhu2022minimalist,
  title={Understanding Edge-of-Stability Training Dynamics with a Minimalist Example},
  author={Zhu, Xu and Wang, Zhaodong and Wang, Xiao and Ge, Rong},
  journal={arXiv preprint arXiv:2210.03294},
  year={2022}
}

@inproceedings{cohen2024centralflow,
  title={Understanding Optimization in Deep Learning with Central Flows},
  author={Cohen, Jeremy M and Damian, Alex and Talwalkar, Ameet and Kolter, J Zico and Lee, Jason D},
  booktitle={International Conference on Learning Representations},
  year={2025},
  url={https://arxiv.org/abs/2410.24206}
}

@inproceedings{foret2021sharpnessaware,
title={Sharpness-aware Minimization for Efficiently Improving Generalization},
author={Pierre Foret and Ariel Kleiner and Hossein Mobahi and Behnam Neyshabur},
booktitle={International Conference on Learning Representations},
year={2021},
url={https://openreview.net/forum?id=6Tm1mposlrM}
}

@inproceedings{keskar2017largeminima,
  title        = {On Large-Batch Training for Deep Learning: Generalization Gap and Sharp Minima},
  author       = {Keskar, Nitish Shirish and Mudigere, Dheevatsa and Nocedal, Jorge and Smelyanskiy, Mikhail and Tang, Ping Tak Peter},
  booktitle    = {International Conference on Learning Representations},
  year         = {2017},
  url          = {https://arxiv.org/abs/1609.04836}
}

@inproceedings{chizat2018global,
  title={On the Global Convergence of Gradient Descent for Over-parameterized Models using Optimal Transport},
  author={Chizat, L\'ena{\"i}c and Bach, Francis},
  booktitle={Advances in Neural Information Processing Systems (NeurIPS)},
  volume={31},
  year={2018}
}

@InProceedings{chen2023beyond,
  title = 	 {Beyond the Edge of Stability via Two-step Gradient Updates},
  author =       {Chen, Lei and Bruna, Joan},
  booktitle = 	 {Proceedings of the 40th International Conference on Machine Learning},
  pages = 	 {4330--4391},
  year = 	 {2023},
  editor = 	 {Krause, Andreas and Brunskill, Emma and Cho, Kyunghyun and Engelhardt, Barbara and Sabato, Sivan and Scarlett, Jonathan},
  volume = 	 {202},
  series = 	 {Proceedings of Machine Learning Research},
  month = 	 {23--29 Jul},
  publisher =    {PMLR},
  url = 	 {https://proceedings.mlr.press/v202/chen23b.html}
}

@article{marion2026edge,
  title={{Edge Flow}: A Tractable and Predictive Continuous-Time Model for Gradient Descent at the Edge of Stability},
  author={Marion, Pierre},
  journal={arXiv preprint arXiv:2606.18080},
  year={2026},
  doi={10.48550/arXiv.2606.18080},
  url={https://arxiv.org/abs/2606.18080}
}

@article{regis2026rod,
  title={Rod Flow: A Continuous-Time Model for Gradient Descent at the Edge of Stability},
  author={Regis, Eric and Chewi, Sinho},
  journal={arXiv preprint arXiv:2602.01480},
  year={2026},
  url={https://arxiv.org/abs/2602.01480}
}

@article{regis2026adam,
  title={A Rod Flow Model for Adam at the Edge of Stability},
  author={Regis, Eric and Chewi, Sinho},
  journal={arXiv preprint arXiv:2605.06821},
  year={2026},
  url={https://arxiv.org/abs/2605.06821}
}

@article{hochreiter1997flat,
  title={Flat minima},
  author={Hochreiter, Sepp and Schmidhuber, J{\"u}rgen},
  journal={Neural Computation},
  volume={9},
  number={1},
  pages={1--42},
  year={1997},
  publisher={MIT Press}
}

@article{mei2018mean,
  title={A mean field view of the landscape of two-layer neural networks},
  author={Mei, Song and Montanari, Andrea and Nguyen, Phan-Minh},
  journal={Proceedings of the National Academy of Sciences},
  volume={115},
  number={33},
  pages={E7665--E7671},
  year={2018},
  publisher={National Acad Sciences}
}

@inproceedings{rotskoff2018parameters,
  title={Parameters as interacting particles: long time convergence and asymptotic error scaling of neural networks},
  author={Rotskoff, Grant and Vanden-Eijnden, Eric},
  booktitle={Advances in Neural Information Processing Systems},
  volume={31},
  year={2018}
}

@inproceedings{arbel2019maximum,
  title={Maximum Mean Discrepancy Gradient Flow},
  author={Arbel, Michael and Korba, Anna and Salim, Adil and Gretton, Arthur},
  booktitle={Advances in Neural Information Processing Systems (NeurIPS)},
  volume={32},
  year={2019}
}

@article{chizat2026quantitative,
  title={Quantitative Convergence of {Wasserstein} Gradient Flows of Kernel Mean Discrepancies},
  author={Chizat, L{\'e}na{\"i}c and Colombo, Maria and Colombo, Roberto and Fern{\'a}ndez-Real, Xavier},
  journal={arXiv preprint arXiv:2603.01977},
  year={2026},
  url={https://arxiv.org/abs/2603.01977}
}

@article{chodron2026coulomb,
  title={{Wasserstein} gradient flows for {Coulomb} discrepancies},
  author={Chodron de Courcel, Antonin and Rosenzweig, Matthew},
  journal={arXiv preprint arXiv:2607.12579},
  year={2026},
  url={https://arxiv.org/abs/2607.12579}
}

@article{rosenzweig2026energy,
  title={{Wasserstein} gradient flows of {Maximum Mean Discrepancy} with energy kernels},
  author={Rosenzweig, Matthew and Slep{\v c}ev, Dejan and Wang, Lihan},
  journal={arXiv preprint arXiv:2608.01182},
  year={2026},
  url={https://arxiv.org/abs/2608.01182}
}

@article{lewkowycz2020catapult,
  title={The large learning rate phase of deep learning: the catapult mechanism},
  author={Lewkowycz, Aitor and Bahri, Yasaman and Dyer, Ethan and Sohl-Dickstein, Jascha and Gur-Ari, Guy},
  journal={arXiv preprint arXiv:2003.02218},
  year={2020},
  url={https://arxiv.org/abs/2003.02218}
}

@inproceedings{mei2019meanfield,
  title={Mean-field theory of two-layers neural networks: dimension-free bounds and kernel limit},
  author={Mei, Song and Misiakiewicz, Theodor and Montanari, Andrea},
  booktitle={Proceedings of the Thirty-Second Conference on Learning Theory},
  series={Proceedings of Machine Learning Research},
  volume={99},
  pages={2388--2464},
  year={2019},
  url={https://proceedings.mlr.press/v99/mei19a.html}
}

@article{jordan1998variational,
  title={The Variational Formulation of the {Fokker--Planck} Equation},
  author={Jordan, Richard and Kinderlehrer, David and Otto, Felix},
  journal={SIAM Journal on Mathematical Analysis},
  volume={29},
  number={1},
  pages={1--17},
  year={1998},
  doi={10.1137/S0036141096303359},
  url={https://doi.org/10.1137/S0036141096303359}
}

@article{otto2001geometry,
  title={The Geometry of Dissipative Evolution Equations: The Porous Medium Equation},
  author={Otto, Felix},
  journal={Communications in Partial Differential Equations},
  volume={26},
  number={1--2},
  pages={101--174},
  year={2001},
  doi={10.1081/PDE-100002243},
  url={https://doi.org/10.1081/PDE-100002243}
}

@book{ambrosio2008gradient,
  title={Gradient Flows: In Metric Spaces and in the Space of Probability Measures},
  author={Ambrosio, Luigi and Gigli, Nicola and Savar{\'e}, Giuseppe},
  edition={2nd},
  series={Lectures in Mathematics ETH Z{\"u}rich},
  publisher={Birkh{\"a}user},
  address={Basel},
  year={2008},
  doi={10.1007/978-3-7643-8722-8},
  url={https://doi.org/10.1007/978-3-7643-8722-8}
}

@incollection{lecun1998efficient,
  title={Efficient {BackProp}},
  author={LeCun, Yann and Bottou, L{\'e}on and Orr, Genevieve B. and M{\"u}ller, Klaus-Robert},
  booktitle={Neural Networks: Tricks of the Trade},
  editor={Orr, Genevieve B. and M{\"u}ller, Klaus-Robert},
  series={Lecture Notes in Computer Science},
  volume={1524},
  pages={9--50},
  publisher={Springer},
  year={1998},
  doi={10.1007/3-540-49430-8_2},
  url={https://leon.bottou.org/papers/lecun-98x}
}

@inproceedings{li2019large,
  title={Towards Explaining the Regularization Effect of Initial Large Learning Rate in Training Neural Networks},
  author={Li, Yuanzhi and Wei, Colin and Ma, Tengyu},
  booktitle={Advances in Neural Information Processing Systems},
  volume={32},
  year={2019},
  url={https://papers.nips.cc/paper_files/paper/2019/hash/bce9abf229ffd7e570818476ee5d7dde-Abstract.html}
}

@article{berthier2026timescales,
  title={The Multiple Timescales of Gradient Descent on the Edge of Stability: A Perturbative Derivation of the {Central Flow}},
  author={Berthier, Rapha{\"e}l},
  journal={arXiv preprint arXiv:2609.01034},
  year={2026},
  doi={10.48550/arXiv.2609.01034},
  url={https://arxiv.org/abs/2609.01034}
}

@article{malago2018wasserstein,
  title={Wasserstein Riemannian geometry of {Gaussian} densities},
  author={Malag{\`o}, Luigi and Montrucchio, Luigi and Pistone, Giovanni},
  journal={Information Geometry},
  volume={1},
  number={2},
  pages={137--179},
  year={2018},
  doi={10.1007/s41884-018-0014-4},
  url={https://doi.org/10.1007/s41884-018-0014-4}
}

@book{saad2003iterative,
  title={Iterative Methods for Sparse Linear Systems},
  author={Saad, Yousef},
  edition={2nd},
  publisher={Society for Industrial and Applied Mathematics},
  address={Philadelphia},
  year={2003},
  doi={10.1137/1.9780898718003},
  url={https://doi.org/10.1137/1.9780898718003}
}

@article{strang1968construction,
  title={On the Construction and Comparison of Difference Schemes},
  author={Strang, Gilbert},
  journal={SIAM Journal on Numerical Analysis},
  volume={5},
  number={3},
  pages={506--517},
  year={1968},
  doi={10.1137/0705041},
  url={https://doi.org/10.1137/0705041}
}

@incollection{gutknecht2007block,
  title={Block {Krylov} Space Methods for Linear Systems with Multiple Right-Hand Sides: An Introduction},
  author={Gutknecht, Martin H.},
  booktitle={Modern Mathematical Models, Methods and Algorithms for Real World Systems},
  editor={Siddiqi, A. H. and Duff, I. S. and Christensen, O.},
  publisher={Anamaya Publishers},
  address={New Delhi},
  pages={420--447},
  year={2007},
  url={https://people.math.ethz.ch/~mhg/pub/delhipap.pdf}
}

@inproceedings{orvieto2022anticorrelated,
  title={Anticorrelated Noise Injection for Improved Generalization},
  author={Orvieto, Antonio and Kersting, Hans and Proske, Frank and Bach, Francis and Lucchi, Aurelien},
  booktitle={Proceedings of the 39th International Conference on Machine Learning},
  series={Proceedings of Machine Learning Research},
  volume={162},
  pages={17094--17116},
  year={2022},
  url={https://proceedings.mlr.press/v162/orvieto22a.html}
}

@inproceedings{barrett2021implicit,
  title={Implicit Gradient Regularization},
  author={Barrett, David G. T. and Dherin, Benoit},
  booktitle={International Conference on Learning Representations},
  year={2021},
  url={https://openreview.net/forum?id=3q5IqUrkcF}
}

@inproceedings{smith2021origin,
  title={On the Origin of Implicit Regularization in Stochastic Gradient Descent},
  author={Smith, Samuel L. and Dherin, Benoit and Barrett, David G. T. and De, Soham},
  booktitle={International Conference on Learning Representations},
  year={2021},
  url={https://arxiv.org/abs/2101.12176}
}

@inproceedings{ziyin2025neural,
  title={Neural Thermodynamics: Entropic Forces in Deep and Universal Representation Learning},
  author={Liu, Ziyin and Xu, Yizhou and Chuang, Isaac},
  booktitle={Advances in Neural Information Processing Systems},
  volume={38},
  year={2025},
  url={https://proceedings.neurips.cc/paper_files/paper/2025/hash/51173cf34c5faac9796a47dc2fdd3a71-Abstract-Conference.html}
}

@article{xu2026flatness,
  title={Does {SGD} Seek Flatness or Sharpness? An Exactly Solvable Model},
  author={Xu, Yizhou and Beneventano, Pierfrancesco and Chuang, Isaac and Liu, Ziyin},
  journal={arXiv preprint arXiv:2602.05065},
  year={2026},
  url={https://arxiv.org/abs/2602.05065}
}

@article{braun1977vlasov,
  title={The {Vlasov} dynamics and its fluctuations in the {$1/N$} limit of interacting classical particles},
  author={Braun, W. and Hepp, K.},
  journal={Communications in Mathematical Physics},
  volume={56},
  number={2},
  pages={101--113},
  year={1977},
  doi={10.1007/BF01611497},
  url={https://doi.org/10.1007/BF01611497}
}

@article{dobrushin1979vlasov,
  title={{Vlasov} equations},
  author={Dobrushin, R. L.},
  journal={Functional Analysis and Its Applications},
  volume={13},
  number={2},
  pages={115--123},
  year={1979},
  doi={10.1007/BF01077243},
  url={https://doi.org/10.1007/BF01077243}
}

@incollection{golse2016dynamics,
  title={On the Dynamics of Large Particle Systems in the Mean Field Limit},
  author={Golse, Fran{\c{c}}ois},
  booktitle={Macroscopic and Large Scale Phenomena: Coarse Graining, Mean Field Limits and Ergodicity},
  editor={Muntean, Adrian and Rademacher, Jens and Zagaris, Antonios},
  series={Lecture Notes in Applied Mathematics and Mechanics},
  volume={3},
  publisher={Springer},
  address={Cham},
  pages={1--144},
  year={2016},
  doi={10.1007/978-3-319-26883-5_1},
  url={https://doi.org/10.1007/978-3-319-26883-5_1}
}

@article{chaintron2022propagation,
  title={Propagation of chaos: A review of models, methods and applications. {I}. Models and methods},
  author={Chaintron, Louis-Pierre and Diez, Antoine},
  journal={Kinetic and Related Models},
  volume={15},
  number={6},
  pages={895--1015},
  year={2022},
  doi={10.3934/krm.2022017},
  url={https://doi.org/10.3934/krm.2022017}
}

@article{jabin2014review,
  title={A review of the mean field limits for {Vlasov} equations},
  author={Jabin, Pierre-Emmanuel},
  journal={Kinetic and Related Models},
  volume={7},
  number={4},
  pages={661--711},
  year={2014},
  doi={10.3934/krm.2014.7.661},
  url={https://doi.org/10.3934/krm.2014.7.661}
}
\bibliographystyle{plainnat}

\tableofcontents

\appendix
\clearpage

\printendnotes

\section{Proofs and conditional derivation of the mean--fluctuation dynamics}
\label[appendix]{app:mf-analysis}

\subsection{Derivation of the mean--fluctuation dynamics}
\label{app:mf-sharp-valley-derivation}

In this section, we rigorously derive the mean--fluctuation dynamics from gradient descent, identifying its validity regime for the specific case of a sharp valley (\cref{thm:mf-sharp-valley-validity}).
In a nutshell, the sharp-valley picture describes a landscape on orthogonal coordinates $\theta=(\theta_1,\theta_2)\in\R^d\times\R^r$ whose valley bottom
$\theta_2=0$ satisfies $\nabla_2E_\eta(\theta_1,0)=0$, while a small
tangential gradient along $\theta_1$ may remain. The curvature along the sharp directions $\theta_2$ is $2/\eta + O(1)$, hence the name ``sharp'' valley as $\eta\to 0$.

\begin{center}
\begin{minipage}{\linewidth}
    \centering
    
    \captionof{figure}{Local schematic of a sharp valley for $d=r=1$.
    The highlighted bottom $\theta_2=0$ is stationary transversely but
    may have a tangential slope along $\theta_1$. The transverse curvature is
    $2/\eta+O(1)$.
    {The schematic of \cref{fig:mf-sharp-valley-landscape-main}
    is repeated here for reference.}}
    \label{fig:mf-sharp-valley-landscape}
\end{minipage}
\end{center}

\paragraph{Structure of the section.}
We first perform a scaling analysis which then motivates the assumptions needed to prove \cref{thm:mf-sharp-valley-validity}. 
We then decompose the loss landscape into sharp $H_\eta$ and tangential $V_\eta$ components (see \cref{lem:mf-normal-taylor-decomposition}).
We then prove the main result in the restricted case where $m=2$ and $V_\eta=0$ (\cref{prop:mf-pair-direct-validity}), 
then in the full sharp-valley picture (\cref{thm:mf-pair-potential}) for $m=2$ and finally extend this to larger even windows $m\ge 2$, thereby proving \cref{thm:mf-sharp-valley-validity}. 

\paragraph{Scaling analysis.}

Before stating the assumptions and commenting on them, we perform an informal scaling analysis to clarify the picture.

On the algorithmic timescale $t$, the mean--fluctuation equations read:
\begin{equation}
 \begin{aligned}
 \frac{d \theta}{d t}
 &=-\eta \nabla E_\eta(\theta) 
   -{\frac12\eta D^3E_\eta(\theta):\Sigma},\\
 \frac{d\Sigma}{d t}
 &=-{(\eta H\Sigma+\eta \Sigma H)} 
   +{\eta^2 H\Sigma  H},
 \qquad  H=D^2E_\eta(\theta).
 \end{aligned}
\end{equation}
First, we can factor out $\eta$ in both the mean and covariance equations, so that it is natural to consider the slow timescale $\tau$ defined by $\tau = \eta t$. 
On this timescale, the mean--fluctuation dynamics become
\begin{equation}\label{eq:mf-sv-full-mf}
 \begin{aligned}
 \frac{d \theta}{d \tau}
 &=- \nabla E_\eta(\theta) 
   -{\frac12 D^3E_\eta(\theta):\Sigma},\\
 \frac{d\Sigma}{d \tau}
 &=-{( H\Sigma+ \Sigma H)} 
   +{\eta H\Sigma  H},
 \qquad  H=D^2E_\eta(\theta).
 \end{aligned}
\end{equation}
Consider the mean equation. We shall consider a regime where both mean drifts are of comparable size:
\begin{equation*}
     \nabla E_\eta \sim  D^3 E_\eta : \Sigma,
\end{equation*}
which yields, if $\eta |D^3 E_\eta| =O(1)$ and writing $\Sigma \sim a^2$,
\begin{equation*}
    |\nabla E| \sim \frac{a^2}{\eta} =: \delta.
\end{equation*}
Therefore, $\delta$ shall be thought of as the tangential displacement scale over slow times $\tau$ of order one. 
Consider the covariance equation. At the Edge of Stability, and if $\eta\to 0$, one must have $H\sim 2/\eta$, so that we have
\begin{align*}
    &-{( H\Sigma+ \Sigma H)} \sim \frac{a^2}{\eta}= \delta, \\ 
   &{\eta H\Sigma  H} \sim \frac{a^2}{\eta} \sim \delta.
\end{align*}
However, their leading terms cancel. In the sharp subspace, writing
$H=2\Id/\eta+B$ with $\|B\|=O(1)$ gives
\begin{equation*}
 -(H\Sigma+\Sigma H)+\eta H\Sigma H
 =B\Sigma+\Sigma B+\eta B\Sigma B \sim a^2.
\end{equation*}
Thus, over $\tau=O(1)$, or $t=O(\eta^{-1})$ iterations, the mean $\theta$ can move by $O(\delta)$ while the covariance
can change by $O(a^2)$, an amount comparable to its own size.

Normalizing these quantities as $\bar\theta=\theta/\delta$
and $\bar\Sigma =\Sigma /a^2$ therefore gives order-one
bounds on the retained velocities in $\tau$.

The scaling $H\sim 2/\eta$ is also the reason why one only assumes $\eta |D^j E_\eta|\le C$ for $j\ge 2$: uniform $\eta$-independent regularity of higher derivatives of $E_\eta$ would more or less restrict the analysis to quadratic landscapes (derivatives of higher order being negligible with respect to the Hessian). 

{The two-point scaling is illustrated in
\cref{fig:mf-pair-scales-appendix}.
We now state the assumptions used to justify it.}

\begin{figure}[!ht]
    \centering
\begingroup
\definecolor{mfscalegd}{RGB}{53,103,155}
\definecolor{mfscalemean}{RGB}{190,92,28}
\begin{tikzpicture}[
    x=1cm,y=1cm,font=\small,
    axis/.style={black!65,line width=.55pt,-{Stealth[length=1.6mm]}},
    guide/.style={black!40,densely dashed,line width=.4pt},
    measure/.style={black!80,line width=.55pt,
        {Stealth[length=1.15mm]}-{Stealth[length=1.15mm]}}
]
    \foreach \j in {0,...,30} {
        \coordinate (p\j) at
            ({.25+.27*\j},{.36+.002*\j+(-1)^\j*sqrt(1.44-.0048*(\j-15)^2)});
    }
    \foreach \j in {0,...,29} {
        \pgfmathtruncatemacro{\next}{\j+1}
        \coordinate (m\j) at ($(p\j)!.5!(p\next)$);
    }

    \draw[axis] (.1,0) -- (11.05,0) node[right] {$\theta_1$};
    \draw[axis] (.385,-1.1) -- (.385,1.93) node[above] {$\theta_2$};
    \node[below left,inner sep=2pt] at (.385,0) {$0$};

    \draw[mfscalegd!70,densely dashed,line width=.8pt]
        plot[domain=0:30,samples=61,variable=\j]
        ({.25+.27*\j},{.36+.002*\j+sqrt(1.44-.0048*(\j-15)^2)});
    \draw[mfscalegd!70,densely dashed,line width=.8pt]
        plot[domain=0:30,samples=61,variable=\j]
        ({.25+.27*\j},{.36+.002*\j-sqrt(1.44-.0048*(\j-15)^2)});

    \foreach \j in {0,...,29} {
        \pgfmathtruncatemacro{\next}{\j+1}
        \draw[mfscalegd!75,line width=.6pt,
            -{Stealth[length=1.1mm,width=.8mm]}] (p\j) -- (p\next);
    }
    \foreach \j in {0,...,30} {
        \fill[mfscalegd] (p\j) circle[radius=1.2pt];
    }
    \draw[mfscalemean,line width=1pt,
        preaction={draw=white,line width=2.6pt}]
        (m0) \foreach \j in {1,...,29} {-- (m\j)};
    \foreach \j in {0,...,29} {
        \fill[mfscalemean] (m\j) circle[radius=1pt];
    }

    \draw[mfscalegd,line width=1.15pt,
        -{Stealth[length=1.7mm]}] (p29) -- (p30);
    \fill[mfscalemean] (m29) circle[radius=2pt];
    \node[above left,inner sep=3pt,text=mfscalemean] at (m29) {$\theta_t$};

    \draw[guide] (m0) -- (m0 |- 0,-1.25);
    \draw[guide] (m29) -- (m29 |- 0,-1.25);
    \draw[measure] (m0 |- 0,-1.25) -- (m29 |- 0,-1.25)
        node[midway,below=3pt] {$\theta_1\sim a^2/\eta$};

    \draw[guide] (m29) -- (m29 |- 0,1.92);
    \draw[guide] (p30) -- (p30 |- 0,1.92);
    \draw[black!80,line width=.55pt] (7.8,1.84) -- (8.78,1.84);
    \draw[black!80,line width=.55pt,-{Stealth[length=1.15mm]}]
        (7.8,1.84) -- (m29 |- 0,1.84);
    \draw[black!80,line width=.55pt,-{Stealth[length=1.15mm]}]
        (8.78,1.84) -- (p30 |- 0,1.84);
    \node[above=3pt] at (8.2825,1.84) {$|\xi_1|\sim a^2$};

    \draw[guide] (p30) -- (9.2,0 |- p30);
    \draw[guide] (m29) -- (9.2,0 |- m29);
    \draw[measure] (9.2,0 |- m29) -- (9.2,0 |- p30)
        node[midway,right=5pt] {$|\xi_2|\sim a$};
    \draw[measure] (9.2,0) -- (9.2,0 |- m29)
        node[midway,right=5pt] {$\theta_2\sim a^2$};

    \node[text=mfscalegd,inner sep=2pt] (growthlabel) at (2.05,1.78)
        {growth};
    \draw[mfscalegd!80,line width=.55pt,-{Stealth[length=1.2mm]}]
        (growthlabel.south) -- (p6);
    \node[text=mfscalegd,inner sep=2pt] (dampinglabel) at (6.7,1.78)
        {damping};
    \draw[mfscalegd!80,line width=.55pt,-{Stealth[length=1.2mm]}]
        (dampinglabel.south) -- (p24);

    \draw[mfscalegd,line width=.8pt] (1.2,2.18) -- (1.65,2.18);
    \fill[mfscalegd] (1.425,2.18) circle[radius=1.2pt];
    \node[right] at (1.7,2.18) {GD iterates};
    \draw[mfscalemean,line width=1pt] (4,2.18) -- (4.45,2.18);
    \fill[mfscalemean] (4.225,2.18) circle[radius=1.2pt];
    \node[right] at (4.5,2.18) {pair means};
\end{tikzpicture}
\endgroup
    \caption{{Two-point scaling, repeated from
    \cref{fig:mf-pair-scales} for reference in the derivation.
    Here $m=2$, $V_\eta=0$, and the sharp pair mean is prepared as in
    \eqref{eq:mf-energy-initial-preparation}.
    The orange points are consecutive-pair means; the highlighted pair has
    midpoint $\theta_t$ and half-separation $\xi_t$.
    Labels indicate upper scales over $O(\eta^{-1})$ iterations.
    The dashed envelope illustrates a possible $O(a)$ change in the sharp
    fluctuation amplitude.}}
    \label{fig:mf-pair-scales-appendix}
\end{figure}

{
\paragraph{Assumptions.}
Let $\eta\to0$ and $\delta:=a^2/\eta\to0$, with no further
restriction on $a/\eta$. Fix orthogonal coordinates
$\theta=(\theta_1,\theta_2)\in\R^d\times\R^r$ and constants $C,C_0\ge1$
independent of $\eta,a,m$. Here $C_0$ fixes the tangential neighborhood,
while $C$ bounds the derivatives and initial data below.

\noindent\textbf{Landscape.}
We assume $E_\eta\in C^4(\R^{d+r})$.

For $|\theta_1|\le\delta/C_0$,
we assume
\begin{equation}\label{eq:mf-energy-valley-assumptions}
 \nabla_2E_\eta(\theta_1,0)=0,\qquad
 \left\|D_2^2E_\eta(\theta_1,0)-\frac2\eta\Id_r\right\|_{\rm op}
 \le C,
\end{equation}
\begin{equation}\label{eq:mf-energy-tangential-assumptions}
 |\nabla_1E_\eta(\theta_1,0)|\le C\delta,\qquad
 \|D_1^2E_\eta(\theta_1,0)\|_{\rm op}\le C,
\end{equation}
and, on the whole space, the derivative bound
\begin{equation}\label{eq:mf-energy-higher-derivatives}
 \|D^3E_\eta\|_{\infty,\rm op}
 +\|D^4E_\eta\|_{\infty,\rm op}\le C/\eta.
\end{equation}
Here $\|\cdot\|_{\infty,\rm op}$ is the supremum of the operator
norm over $\R^{d+r}$. Dimensions and constants are
independent of $\eta,a$ and the window length $m$.

\noindent\textbf{Prepared first pair.}
Start GD with $|\tilde\theta_{1,0}|\le\delta/C_0$ and
$\tilde\theta_{k+1}=\tilde\theta_k-\eta\nabla E_\eta(\tilde\theta_k)$.
After translating the tangential origin, assume the first pair satisfies
\begin{equation}\label{eq:mf-energy-initial-preparation}
 \frac{\tilde\theta_{1,0}+\tilde\theta_{1,1}}2=0,\qquad
 \left|\frac{\tilde\theta_{2,0}+\tilde\theta_{2,1}}2\right|\le Ca^2,
\qquad |\tilde\theta_{2,0}|\le Ca.
\end{equation}

\noindent\textbf{Window moments and common initialization.}
Let $m\ge2$ be even, with $m(\eta+\delta)$ sufficiently small;
for convergence we require $m(\eta+\delta)\to0$.
Define the actual GD window moments by
\[
 \theta_k=\frac1m\sum_{j=0}^{m-1}\tilde\theta_{k+j},\qquad
 \Sigma_k=\frac1m\sum_{j=0}^{m-1}
 (\tilde\theta_{k+j}-\theta_k)(\tilde\theta_{k+j}-\theta_k)^T,
\]
and initialize MF at $(\theta_0,\Sigma_0)$.
For $m>2$, this window mean need not equal the prepared first-pair
mean in \eqref{eq:mf-energy-initial-preparation}.
Forward windows are the moving averages of \cref{subsec:mf-model},
with the time origin shifted to the first complete window.

\noindent\textbf{Time and block notation.}
We use slow time $\tau=\eta t$, where $t$ is iteration time.
We write $(\theta(\tau),\Sigma(\tau))$ for MF in this time variable
and $\Sigma_{ij}$ for the covariance blocks in the coordinate split.
The proof constructs a slow-time horizon $T_*>0$ independent of
$\eta,a,m$ on which GD and MF asymptotically coincide.

}

\paragraph{Decomposition of the loss landscape.}
The following lemma proves a useful decomposition of the loss landscape of the form 
$$
    E_\eta(\theta_1,\theta_2) = V_\eta(\theta_1) + \frac12 \theta_2^T H_\eta (\theta_1) \theta_2 + R_\eta (\theta_1,\theta_2).
$$
Morally, $V_\eta$ is responsible for the classical gradient-flow drift along the valley $\theta_2=0$, while $H_\eta$ accounts for the orthogonal sharp directions that are driven at the Edge of Stability. 
    
Therefore, the overall drift will be a combination of these two effects, while $R_\eta$ only contributes at higher orders.

\begin{lemma}[Decomposition of the loss into orthogonal pieces]
\label{lem:mf-normal-taylor-decomposition}
Under the landscape assumptions
\eqref{eq:mf-energy-valley-assumptions}--\eqref{eq:mf-energy-higher-derivatives}
and the regularity stated above, define
\[
 V_\eta(\theta_1)=E_\eta(\theta_1,0),\qquad
 H_\eta(\theta_1)=D_2^2E_\eta(\theta_1,0).
\]
Then $V_\eta\in C^4$, $H_\eta\in C^2$, and, for $|\theta_1|\le\delta/C_0$,
\begin{equation}\label{eq:mf-sv-landscape}
 E_\eta(\theta_1,\theta_2)
 =V_\eta(\theta_1)+\frac12\theta_2^TH_\eta(\theta_1)\theta_2
   +R_\eta(\theta_1,\theta_2),
\end{equation}
where
\begin{equation}\label{eq:mf-normal-taylor-remainder}
 R_\eta(\theta_1,\theta_2)
 =\frac12\int_0^1(1-s)^2
 D_2^3E_\eta(\theta_1,s\theta_2)
       [\theta_2,\theta_2,\theta_2]\,\dd s.
\end{equation}
On this region, the remainder is jointly $C^2$ and satisfies
\begin{equation}\label{eq:mf-sv-remainder-properties}
 \begin{gathered}
 R_\eta(\theta_1,0)=\nabla_2R_\eta(\theta_1,0)
                 =\nabla_2^2R_\eta(\theta_1,0)=0,\\
 |R_\eta(\theta_1,\theta_2)|\le C|\theta_2|^3/(6\eta).
 \end{gathered}
\end{equation}
We furthermore have
$|\nabla_1V_\eta|\le C\delta$, $\|D_1^2V_\eta\|_{\rm op}\le C$,
$\|D_1^jV_\eta\|_{\rm op}\le C/\eta$ for $j=3,4$,
$\|H_\eta-2\Id_r/\eta\|_{\rm op}\le C$, and
$\|D_1^jH_\eta\|_{\rm op}\le C/\eta$ for $j=1,2$.
\end{lemma}

\begin{proof}
Apply Taylor's formula with integral remainder to
$s\mapsto E_\eta(\theta_1,s\theta_2)$ between $0$ and $1$.
The constant and quadratic terms are $V_\eta$ and
$\theta_2^TH_\eta\theta_2/2$, while the linear term vanishes by
assumption. This gives \eqref{eq:mf-normal-taylor-remainder}.
The bound on $D^3E_\eta$ and
$\frac12\int_0^1(1-s)^2\,\dd s=1/6$ give the bound on $R_\eta$.
By their definitions, $V_\eta\in C^4$ and $H_\eta\in C^2$;
their stated derivative bounds are restrictions of derivatives of
$E_\eta$ of order at most four.
Hence $R_\eta=E_\eta-V_\eta-\theta_2^TH_\eta\theta_2/2$ is $C^2$;
differentiating this identity with respect to $\theta_2$ at $\theta_2=0$ gives
the vanishing derivatives in \eqref{eq:mf-sv-remainder-properties}.
\end{proof}

\paragraph{Two-point windows without tangential potential.}
We begin with $m=2$ and $V_\eta=0$. The aim is therefore to show the following proposition:

\begin{proposition}[Two-point window without tangential potential]
\label{prop:mf-pair-direct-validity}
Fix $C_0\ge1$ and $C\ge1$, independently of $\eta,a$.
Let $\eta, a\in (0,1)$ be such that $\delta:=a^2/\eta\in (0,1)$.
The constant $C_0$ fixes the tangential neighborhood
$|\theta_1|\le\delta/C_0$; $C$ is a common upper-bound constant
for the loss derivatives and the initial data specified below.
Assume $E_\eta\in C^4(\R^{d+r})$ satisfies
\eqref{eq:mf-energy-valley-assumptions} and
\eqref{eq:mf-energy-higher-derivatives} with this constant $C$.
Let $H_\eta=D_2^2E_\eta(\cdot,0)$ as in
\cref{lem:mf-normal-taylor-decomposition} and further assume that $V_\eta \equiv 0$. 

Starting from $|\tilde\theta_{1,0}|\le\delta/C_0$, run
$\tilde\theta_{t+1}=\tilde\theta_t-\eta\nabla E_\eta(\tilde\theta_t)$
and define
\[
 \theta_t=\frac{\tilde\theta_t+\tilde\theta_{t+1}}2,
 \qquad \xi_t=\frac{\tilde\theta_t-\tilde\theta_{t+1}}2,
 \qquad \Sigma_t=\xi_t\xi_t^T.
\]
Assume $\theta_{1,0}=0$, $|\theta_{2,0}|\le Ca^2$, and
$|\tilde\theta_{2,0}|\le Ca$.
Let $(\theta(\tau),\Sigma(\tau))$ solve the full MF equations
\eqref{eq:mf-sv-full-mf} in slow time $\tau=\eta t$, initialized at
$(\theta_0,\Sigma_0)$, and set
$\bar\theta=\theta/\delta$, $\bar\Sigma=\Sigma/a^2$ for both dynamics.

There exist $T_*>0$, $C_*<\infty$ and a positive threshold,
all depending only on $C_0,C,d,r$, such that whenever $\eta+\delta$
is below this threshold, both dynamics
exist on the corresponding time intervals and
\begin{equation}\label{eq:mf-pair-direct-comparison}
 \sup_{0\le\tau\le T_*}\left(
 |\bar\theta_{\lfloor\tau/\eta\rfloor}-\bar\theta(\tau)|
 +\|\bar\Sigma_{\lfloor\tau/\eta\rfloor}-\bar\Sigma(\tau)\|_F
 \right)\le C_*(\eta+\delta).
\end{equation}
\end{proposition}

We decompose the proof of this proposition into three lemmas. 
We first bound the Taylor remainders between two iterates of $(\theta_t, \Sigma_t)$ in \cref{lem:mf-pair-taylor}.
Then, we control the components of $(\theta_t,\Sigma_t)$ that remain small on the long timescale $\eta^{-1}$ (i.e. the mean along the sharp directions $\theta_{2,t}$ and the tangential blocks $\Sigma_{11,t}, \Sigma_{12,t}, \Sigma_{21,t}$) in \cref{lem:mf-pair-small-components}.
This part crucially uses an established stabilization mechanism to propagate the smallness over long timescales of order $\eta^{-1}$.
Finally, we rescale the retained quantities $(\theta_{1,t}, \Sigma_{22,t})$ and combine the remainder estimates to obtain a bound for the accumulated error on timescales of order $\eta^{-1}$ (\cref{lem:mf-pair-rescaled-defects}).

We use a bootstrap argument. Increase $C_0$, if necessary, so that it
exceeds the uniform bound on $\tr\bar\Sigma_{22,0}$ by a fixed positive
margin. This enlargement depends only on the original $C_0,C$ and
only shrinks the neighborhood required in the assumptions.
The prepared pair identity and Taylor expansion give
$|\tilde\theta_{1,0}|=|\tilde\theta_{1,1}|=O(a^2)=o(\delta)$,
so the initial iterates remain in this smaller neighborhood for small
enough parameters. Keep $C_0$ fixed thereafter, independently of
$\eta,a$. For $t=\lfloor\tau/\eta\rfloor$, consider the temporary bounds
\begin{equation}\label{eq:mf-pair-bootstrap-bounds}
 \begin{gathered}
 \max\{|\bar\theta_1(\tau)|,|\bar\theta_{1,t}|\}<\frac1{2C_0},\qquad
 \max\{\tr\bar\Sigma_{22}(\tau),\tr\bar\Sigma_{22,t}\}<C_0,\\
 \max_{j=0,1}|\tilde\theta_{1,t+j}|<\frac\delta{C_0},\qquad
 \max_{j=0,1}|\tilde\theta_{2,t+j}|<2\sqrt{C_0}\,a,\\
 |\theta_2(\tau)|<a,\qquad
 \max\{\|\Sigma_{11}(\tau)\|_F,\|\Sigma_{12}(\tau)\|_F\}<a^2.
 \end{gathered}
\end{equation}
Define their common stopping time in slow time by
\begin{equation}\label{eq:mf-pair-stopping-time}
 T:=\sup\left\{s\in(0,1]:
 \begin{array}{l}
 \text{MF exists on }[0,s)\text{ and }\eqref{eq:mf-pair-bootstrap-bounds}
 \text{ holds}\\
 \text{for every }0\le\tau<s
 \end{array}\right\}.
\end{equation}
Thus $T$ is the first exit time, capped at $1$, and may depend on
$\eta,a$. The prepared initial data satisfy these bounds strictly
for small enough parameters, so $T>0$.
The next three lemmas apply before $T$; since $T\le1$, their constants
$C_T$ can be chosen uniformly in $\eta,a$.
After the lemmas, we prove $T>T_*>0$ with $T_*$ independent of
the parameters, thereby closing the bootstrap.

\begin{lemma}[Taylor remainders for a GD pair]
\label{lem:mf-pair-taylor}
Under the assumptions of \cref{prop:mf-pair-direct-validity}, set
$H_t=D^2E_\eta(\theta_t)$ and $A_t=\Id-\eta H_t$.
For every integer $t\ge0$ with $\eta t<T$, where $T$ is defined in
\eqref{eq:mf-pair-stopping-time}, the pair moments satisfy
\begin{align}
 \theta_{t+1}
 &=\theta_t-\eta\nabla E_\eta(\theta_t)
       -\frac\eta2D^3E_\eta(\theta_t):\Sigma_t+R_t^\theta,
       \label{eq:mf-pair-direct-mean}\\
 \Sigma_{t+1}
 &=A_t\Sigma_t A_t^T+R_t^\Sigma,
       \label{eq:mf-pair-direct-covariance}
\end{align}
where
\begin{equation}\label{eq:mf-pair-direct-mean-remainders}
 |R_{1,t}^\theta|+|R_{2,t}^\theta|\le Ca^3,
\end{equation}
and
\begin{equation}\label{eq:mf-pair-direct-covariance-remainders}
 \|R_{11,t}^\Sigma\|_F\le Ca^5,\qquad
 \|R_{12,t}^\Sigma\|_F+\|R_{22,t}^\Sigma\|_F\le Ca^4.
\end{equation}
\end{lemma}

\begin{proof}
The exact pair identities give
\begin{equation}\label{eq:mf-pair-direct-mean-remainder-exact}
 R_t^\theta=-\frac\eta2\left[
 \nabla E_\eta(\theta_t+\xi_t)+\nabla E_\eta(\theta_t-\xi_t)
 -2\nabla E_\eta(\theta_t)-D^3E_\eta(\theta_t):\Sigma_t
 \right].
\end{equation}
Writing the odd gradient defect as
$q_t=\tfrac12[\nabla E_\eta(\theta_t+\xi_t)
-\nabla E_\eta(\theta_t-\xi_t)]-H_t\xi_t$,
we have $\xi_{t+1}=A_t\xi_t-\eta q_t$ and hence
\begin{equation}\label{eq:mf-pair-direct-covariance-remainder-exact}
 R_t^\Sigma=-\eta\left[(A_t\xi_t)q_t^T+q_t(A_t\xi_t)^T\right]
                 +\eta^2q_tq_t^T.
\end{equation}
The identity $\nabla_2E_\eta(\theta_1,0)=0$ and Taylor's formula,
applied directly to the gradient of $E_\eta$ with respect to $\theta_2$, give
\[
 \begin{aligned}
 \nabla_1E_\eta(\theta_1,\theta_2)
 &=\frac12\nabla_1H_\eta(\theta_1):\theta_2\theta_2^T
     +O(|\theta_2|^3/\eta),\\
 \nabla_2E_\eta(\theta_1,\theta_2)
 &=H_\eta(\theta_1)\theta_2+O(|\theta_2|^2/\eta).
 \end{aligned}
\]
Since $\xi_t=\eta\nabla E_\eta(\tilde\theta_t)/2$, these estimates imply
\begin{equation}\label{eq:mf-pair-direct-sizes}
 |\xi_{1,t}|\le Ca^2,\qquad |\xi_{2,t}|\le Ca,\qquad
 |\theta_{2,t}|\le Ca(\eta+a).
\end{equation}
For the last estimate, the leading sharp term cancels in
$\theta_{2,t}=\tilde\theta_{2,t}-\eta\nabla_2E_\eta(\tilde\theta_t)/2$.
Taylor expansion of the even and odd gradient combinations, using
\eqref{eq:mf-energy-higher-derivatives},
yields $|R_t^\theta|\le Ca^3$ and $|q_t|\le Ca^3/\eta$.
At the pair mean, $H_{12,t}=O(a+\delta)$; together with
\eqref{eq:mf-pair-direct-sizes}, this gives
$|(A_t\xi_t)_1|\le Ca^2$ and $|(A_t\xi_t)_2|\le Ca$.
Substitution into \eqref{eq:mf-pair-direct-covariance-remainder-exact}
proves the blockwise covariance bounds.
\end{proof}

\begin{lemma}[Small sharp mean and tangential covariance blocks]
\label{lem:mf-pair-small-components}
For MF at $0\le\tau<T$, and for GD pairs at $0\le\eta t<T$,
with $T$ defined in \eqref{eq:mf-pair-stopping-time}, we have
\begin{equation}\label{eq:mf-pair-fast-propagation}
 |\theta_2|\le C_Ta^2,\qquad
 \|\Sigma_{11}\|_F\le C_Ta^4,\qquad
 \|\Sigma_{12}\|_F\le C_Ta^3.
\end{equation}
In particular, after normalization,
$|\bar\theta_2|\le C_T\eta$,
$\|\bar\Sigma_{11}\|_F\le C_Ta^2$ and
$\|\bar\Sigma_{12}\|_F\le C_Ta$.
\end{lemma}

\begin{proof}
\underline{Gradient descent long-time estimates.}
The covariance bounds follow directly from $\Sigma_t=\xi_t\xi_t^T$
and \eqref{eq:mf-pair-direct-sizes}. The same estimate gives
$|\theta_{2,t}|\le Ca(\eta+a)$, which proves the sharp-mean bound
when $a\ge\eta$.
Suppose now that $a\le\eta$. We use the cancellation between two
successive steps, rather than taking absolute values in each step.
We first justify the expansion along the sharp directions used to identify the forcing.
For fixed $\theta_1$, $\nabla_2E_\eta(\theta_1,0)=0$ and Taylor's formula give
the exact identity
\[
 \nabla_2E_\eta(\theta_1,\theta_2)
 =H_\eta(\theta_1)\theta_2
 +\int_0^1(1-s)D_2^3E_\eta(\theta_1,s\theta_2):
                  (\theta_2\theta_2^T)\,\dd s.
\]
The segment $(\theta_1,s\theta_2)$, $s\in[0,1]$, keeps $\theta_1$ fixed in the allowed neighborhood,
and the derivative bounds hold along the entire segment. Hence
\[
 \begin{aligned}
 |\nabla_2E_\eta(\theta_1,\theta_2)-H_\eta(\theta_1)\theta_2|
 &\le\frac{C}{2\eta}|\theta_2|^2,\\
 \|D_2^3E_\eta(\theta_1,\theta_2)-D_2^3E_\eta(\theta_1,0)\|_{\rm op}
 &\le\frac C\eta|\theta_2|.
 \end{aligned}
\]
The first bound uses $D^3E_\eta$ and the second uses $D^4E_\eta$.
Thus the expansion requires a small displacement $|\theta_2|$ along the sharp directions.
Set
\[
 \begin{aligned}
 M_t&:=\eta H_\eta(\theta_{1,t})-\Id_r,
       &\|M_t-\Id_r\|_{\rm op}&\le C\eta,\\
 F_t&:=-\frac\eta2D_2^3E_\eta(\theta_{1,t},0):\Sigma_{22,t},
       &|F_t|&\le Ca^2.
 \end{aligned}
\]
Substituting these estimates into the sharp mean update
\eqref{eq:mf-pair-direct-mean} gives
\[
 \theta_{2,t+1}=-M_t\theta_{2,t}+F_t
                  +O(a^3+\eta^2a^2).
\]
Indeed, after multiplication by $\eta$, the nonlinear gradient
contributes $O(|\theta_{2,t}|^2)$, and replacing the sharp cubic
coefficient by its value at $(\theta_{1,t},0)$ costs
$O(|\theta_{2,t}|a^2)$ because $\|\Sigma_{22,t}\|_F\le Ca^2$.
The other covariance contractions are $O(a^3)$, using
$\|\Sigma_{12,t}\|_F\le Ca^3$ and $\|\Sigma_{11,t}\|_F\le Ca^4$;
the pair Taylor defect is also $O(a^3)$.
We use only the preliminary bound $|\theta_{2,t}|\le C\eta a$
from \eqref{eq:mf-pair-direct-sizes}, not the desired $O(a^2)$ bound.
The total error is therefore $O(\eta^2a^2+a^3)=O(\eta a^2)$.

The forcing varies slowly. Using only the preliminary bounds,
\cref{lem:mf-pair-taylor} gives
\[
 |\theta_{1,t+1}-\theta_{1,t}|\le Ca^2,\qquad
 \|\Sigma_{22,t+1}-\Sigma_{22,t}\|_F\le C\eta a^2.
\]
For the second estimate, the leading sharp multiplier in $A_t$
is $-\Id_r+O(\eta)$, so its two minus signs cancel in
$A_t\Sigma_tA_t^T$; the covariance defect $O(a^4)$ is also
$O(\eta a^2)$. Together with
$\eta\|D_1D_2^3E_\eta\|_{\rm op}\le C$, these bounds imply
$|F_{t+1}-F_t|\le C(a^4+\eta a^2)\le C\eta a^2$.
Composing two mean updates now makes the cancellation explicit:
\[
 \theta_{2,t+2}=M_{t+1}M_t\theta_{2,t}
       +(F_{t+1}-F_t)-(M_{t+1}-\Id_r)F_t+O(\eta a^2).
\]
Consequently,
\[
 |\theta_{2,t+2}|\le(1+C\eta)|\theta_{2,t}|+C\eta a^2.
\]
The initial preparation controls $\theta_{2,0}$, and the one-step
update then gives $|\theta_{2,1}|\le Ca^2$. For either parity
$j\in\{0,1\}$ and $\eta(2k+j)<T$, iteration yields
\[
 \begin{aligned}
 |\theta_{2,2k+j}|
 &\le(1+C\eta)^k|\theta_{2,j}|
       +C\eta a^2\sum_{\ell=0}^{k-1}(1+C\eta)^\ell\\
 &\le Ca^2(1+T)e^{CT}.
 \end{aligned}
\]
The Gr\"onwall factor is uniform in $\eta$, since
$(1+C\eta)^k\le e^{Ck\eta}\le e^{CT}$.

\underline{Mean--fluctuation mean in sharp directions.}
All derivatives are now in slow time $\tau$. Covariance positivity
and the stopping bounds give $\tr\Sigma\le Ca^2$, hence
$|D_2D^2E_\eta(\theta):\Sigma|\le Ca^2/\eta$.
Taylor expansion of the gradient gives
\[
 \frac{\dd\theta_2}{\dd\tau}
 =-H_\eta(\theta_1)\theta_2
   +O(|\theta_2|^2/\eta)
   -\frac12D_2D^2E_\eta(\theta):\Sigma.
\]
Since $H_\eta\ge(2/\eta-C)\Id_r$, taking the inner product with
$\theta_2/|\theta_2|$ gives
\[
 \begin{aligned}
 D_\tau^+|\theta_2|
 &\le-\left(\frac2\eta-C\right)|\theta_2|
          +\frac C\eta|\theta_2|^2+\frac C\eta a^2\\
 &\le-\frac1\eta|\theta_2|+\frac C\eta a^2.
 \end{aligned}
\]
Here $D_\tau^+$ is the upper right Dini derivative, so the inequality
also holds at $\theta_2=0$. The last line uses $|\theta_2|<a$ and
small $a,\eta$ to absorb the quadratic term into the damping.
The integrating-factor estimate is therefore
\[
 |\theta_2(\tau)|\le e^{-\tau/\eta}|\theta_2(0)|
                    +Ca^2(1-e^{-\tau/\eta})\le Ca^2.
\]

\underline{Mean--fluctuation covariance blocks.}
Write $H=D^2E_\eta(\theta)$ and $H_{ij}$ for its blocks.
Since $D_1^2E_\eta(\theta_1,0)=0$ and differentiating
$\nabla_2E_\eta(\theta_1,0)=0$ twice in $\theta_1$ gives
$D_1^2D_2E_\eta(\theta_1,0)=0$, the fourth-derivative bound yields
$D_1^2E_\eta(\theta_1,\theta_2)=O(|\theta_2|^2/\eta)$.
Expanding the other blocks and using the sharp-mean bound just proved,
we obtain
\[
 H_{11}=O(a^2\delta),\qquad H_{12}=O(\delta),\qquad
 H_{22}=2\Id_r/\eta+O(1).
\]
The stabilizing term in the mixed covariance equation is visible
by writing that block explicitly:
\[
 \begin{aligned}
 \frac{\dd\Sigma_{12}}{\dd\tau}
 &=-\Sigma_{12}H_{22}
   +H_{11}\Sigma_{12}(\eta H_{22}-\Id_r)\\
 &\quad+H_{12}\Sigma_{22}(\eta H_{22}-\Id_r)
   +(\eta H_{11}-\Id_d)\Sigma_{11}H_{12}
   +\eta H_{12}\Sigma_{21}H_{12}.
 \end{aligned}
\]
Only the first term has a coefficient of order $1/\eta$; it
contributes at most $-(2/\eta-C)\|\Sigma_{12}\|_F^2$ to the
Frobenius inner product with $\Sigma_{12}$.
The remaining terms, and the $11$ block of the same covariance
equation, give
\[
 \begin{aligned}
 D_\tau^+\|\Sigma_{12}\|_F
 &\le-(2/\eta-C)\|\Sigma_{12}\|_F
             +C\delta(a^2+\|\Sigma_{11}\|_F),\\
 D_\tau^+\|\Sigma_{11}\|_F
 &\le C\|\Sigma_{11}\|_F+C\delta\|\Sigma_{12}\|_F
             +C\eta\delta^2a^2.
 \end{aligned}
\]
In the second inequality, the last forcing term comes from
$\eta H_{12}\Sigma_{22}H_{21}$; all other couplings to the sharp
block contain $\Sigma_{12}$ or $\Sigma_{21}$.
For small $\eta$, $2/\eta-C\ge1/\eta$. The stopping bound
$\|\Sigma_{11}\|_F<a^2$ and the initial estimate
$\|\Sigma_{12}(0)\|_F\le Ca^3$ thus give
\[
 \begin{aligned}
 \|\Sigma_{12}(\tau)\|_F
 &\le Ca^3e^{-\tau/\eta}
       +C\delta a^2\int_0^\tau e^{-(\tau-s)/\eta}\,\dd s\\
 &\le Ca^3e^{-\tau/\eta}+Ca^4,
 \end{aligned}
\]
where $\eta\delta=a^2$. This proves the mixed-block bound.
Substitution into the $11$ inequality, followed by Gr\"onwall, yields
\[
 \begin{aligned}
 \|\Sigma_{11}(\tau)\|_F
 &\le e^{C\tau}\left[Ca^4
     +C\delta a^3\int_0^\tau e^{-s/\eta}\,\dd s
     +C\tau\delta a^4+C\tau\eta\delta^2a^2\right]\\
 &\le Ce^{CT}\left[a^4+a^5+T\delta a^4\right]
 \le C_Ta^4.
 \end{aligned}
\]
Together with the sharp-mean bound, this proves
\eqref{eq:mf-pair-fast-propagation}.
\end{proof}

Now that we have proved the propagation, over times $t=O(\eta^{-1})$, of the smallness of $\theta_2$ and the covariance blocks $\Sigma_{12}, \Sigma_{21}, \Sigma_{11}$,
we combine the remainders to keep only the relevant, non-vanishing contributions:

\begin{lemma}[Rescaled equations and error bounds]
\label{lem:mf-pair-rescaled-defects}
Let $T$ be the stopping time \eqref{eq:mf-pair-stopping-time}.
For $0\le\eta t<T$, the GD pair moments satisfy
\begin{equation}\label{eq:mf-pair-retained-updates}
 \begin{aligned}
 \bar\theta_{1,t+1}
 &=\bar\theta_{1,t}
       -\frac{\eta^2}{2}\nabla_1H_\eta:\bar\Sigma_{22,t}
       +\bar R_t^\theta,\\
 \bar\Sigma_{22,t+1}
 &=\bar\Sigma_{22,t}
       -\eta\left[H_\eta\bar\Sigma_{22,t}
         +\bar\Sigma_{22,t}H_\eta\right]
       +\eta^2H_\eta\bar\Sigma_{22,t}H_\eta
       +\bar R_t^\Sigma,
 \end{aligned}
\end{equation}
where $H_\eta$ and its derivatives are evaluated at
$\delta\, \bar\theta_{1,t}$, and
\begin{equation}\label{eq:mf-pair-retained-remainders}
 \begin{gathered}
 |\bar R_t^\theta|\le C_T\eta a,\qquad
 \|\bar R_t^\Sigma\|_F\le C_T\eta\delta,\\
 \sum_{0\le t<T/\eta}
 \left(|\bar R_t^\theta|+\|\bar R_t^\Sigma\|_F\right)
 \le C_T(a+\delta).
 \end{gathered}
\end{equation}
For MF, in slow time $0\le\tau<T$,
\begin{equation}\label{eq:mf-pair-retained-mf}
 \begin{aligned}
 \frac{\dd\bar\theta_1}{\dd\tau}
 &=-\frac\eta2\nabla_1H_\eta:\bar\Sigma_{22}+O_T(a^2),\\
 \frac{\dd\bar\Sigma_{22}}{\dd\tau}
 &=-H_\eta\bar\Sigma_{22}-\bar\Sigma_{22}H_\eta
       +\eta H_\eta\bar\Sigma_{22}H_\eta+O_T(\delta),
 \end{aligned}
\end{equation}
where $H_\eta$ and its derivatives are evaluated at $\delta\, \bar\theta_1(\tau)$.
The leading drifts are uniformly bounded on
$|\bar\theta_1|\le1/(2C_0)$, $\|\bar\Sigma_{22}\|_F\le C_0$.
On this region, we also have
\begin{equation}\label{eq:mf-pair-curvature-derivative}
 \|D_1H_\eta(\delta\bar\theta_1)\|_{\rm op}
 \le\frac C{\eta+a}.
\end{equation}
\end{lemma}

\begin{proof}
By \cref{lem:mf-pair-small-components} and Taylor expansion along the
sharp directions $\theta_2$,
\[
 \begin{aligned}
 \eta\nabla_1E_\eta(\theta_t)&=O_T(a^4),\\
 \eta D_1D^2E_\eta(\theta_t):\Sigma_t
   &=\eta\nabla_1H_\eta(\theta_{1,t}):\Sigma_{22,t}+O_T(a^4).
 \end{aligned}
\]
Here $\nabla_2E_\eta(\theta_1,0)=0$ and the fourth-derivative bound give
$D_1^2D_2E_\eta=O(|\theta_2|/\eta)$ and
$D_1D_2^2E_\eta-\nabla_1H_\eta=O(|\theta_2|/\eta)$.
For the covariance, the Hessian bounds
$H_{22,t}-H_\eta(\theta_{1,t})=O_T(\delta)$ and
$H_{12,t}=O_T(\delta)$ show that replacing the sharp block of
$A_t\Sigma_tA_t^T$ by
$(\Id_r-\eta H_\eta)\Sigma_{22,t}(\Id_r-\eta H_\eta)$
also costs $O_T(\eta\delta a^2)=O_T(a^4)$.
Adding the Taylor defects from \cref{lem:mf-pair-taylor} thus gives
physical errors $O_T(a^3)$ and $O_T(a^4)$.
Division by $\delta$ and $a^2$ yields the pointwise bounds in
\eqref{eq:mf-pair-retained-remainders}; summation over fewer than
$T/\eta+1$ steps gives the accumulated bound $C_T(a+\delta)$.

For MF, the same discarded terms have size $O_T(a^4/\eta)$ per
unit slow time, with no discrete Taylor defects. Normalization
therefore gives the stated errors $O_T(a^2)$ and $O_T(\delta)$.

The mean drift is bounded by $C\eta\|D_1H_\eta\|_{\rm op}\le C$.
For the covariance drift, substituting
$H_\eta=2\Id_r/\eta+O(1)$ cancels the singular terms and gives the same bound.
To prove \eqref{eq:mf-pair-curvature-derivative}, the assumed bound
$\|D_1H_\eta\|_{\rm op}\le C/\eta$ suffices when $a\le\eta$.
When $a\ge\eta$, Taylor's theorem on tangential segments of length proportional to
$a\le\delta$, inside $|\theta_1|\le\delta/C_0$, gives
\[
 a\|D_1H_\eta\|_{\rm op}\le C(1+a^2/\eta)=C(1+\delta),
\]
using $\|H_\eta-2\Id_r/\eta\|_{\rm op}\le C$ and
$\|D_1^2H_\eta\|_{\rm op}\le C/\eta$.
Since $\eta+a\le2a$ in this case, this proves
\eqref{eq:mf-pair-curvature-derivative}.
\end{proof}

\begin{proof}[Proof of \cref{prop:mf-pair-direct-validity}]
Local MF existence and covariance positivity follow as in
\cref{prop:mf-well-posedness}. We compare the two retained components
at their respective error scales $\eta+a$ and $\eta+\delta$, both tending to zero:
\[
 e_t=\frac{|\bar\theta_{1,t}-\bar\theta_1(\eta t)|}{\eta+a}
          +\frac{\|\bar\Sigma_{22,t}-\bar\Sigma_{22}(\eta t)\|_F}{\eta+\delta}.
\]
The leading vector field in \eqref{eq:mf-pair-retained-mf} is
uniformly Lipschitz after this normalization. Indeed, its two
cross-derivative bounds follow from \eqref{eq:mf-pair-curvature-derivative} and
\[
 \frac{\eta(\eta+\delta)}{(\eta+a)^2}
 =\frac{\eta^2+a^2}{(\eta+a)^2}\le1,
 \qquad \frac\delta{\eta+\delta}\le1;
\]
the derivatives in each component's own variable are bounded by
$Ca^2\|D_1^2H_\eta\|_{\rm op}\le C\delta$ and $C$, respectively.
The GD remainders and integrated MF remainders contribute at most
$C_T\eta$ per step in this error norm. Since the retained MF
velocities are bounded, the quadrature error is at most
$C_T\eta^2[(\eta+a)^{-1}+(\eta+\delta)^{-1}]\le C_T\eta$.
Thus, for $\eta(t+1)<T$,
\[
 e_{t+1}\le(1+C\eta)e_t+C_T\eta,\qquad e_0=0,
 \qquad \sup_{\eta t<T}e_t\le C_T,
\]
where the last bound is discrete Gr\"onwall, with constants independent
of $a/\eta$. The two component errors are therefore
$O_T(\eta+a)$ and $O_T(\eta+\delta)$.
Adding the small components from \cref{lem:mf-pair-small-components},
the $O(\eta)$ variation between samples, and using
$a=\sqrt{\eta\delta}\le(\eta+\delta)/2$ proves
\eqref{eq:mf-pair-direct-comparison} on $[0,T)$.

\underline{Final bootstrap argument.}
To close the bootstrap, \cref{lem:mf-pair-small-components} gives
\[
 |\theta_2|\le C_Ta^2<a,\qquad
 \|\Sigma_{11}\|_F\le C_Ta^4<a^2,\qquad
 \|\Sigma_{12}\|_F\le C_Ta^3<a^2,
\]
strictly improving the MF bounds for small parameters.
By \cref{lem:mf-pair-rescaled-defects}, the retained normalized
variables move by at most $C\tau$.
Choose $T_*\in(0,1)$ so that this motion preserves the fixed initial
margins in the bounds on $\bar\theta_1$ and $\tr\bar\Sigma_{22}$.
The constants are uniform since $T\le1$, so $T_*$ is independent
of $\eta,a$. For the individual GD iterates,
$\tilde\theta_t=\theta_t+\xi_t$ and
$\tilde\theta_{t+1}=\theta_t-\xi_t$ give, for $j=0,1$,
\[
 |\tilde\theta_{1,t+j}|\le\frac\delta{2C_0}+Ca^2<\frac\delta{C_0},
 \qquad
 |\tilde\theta_{2,t+j}|\le\sqrt{C_0}\,a+C_Ta^2<2\sqrt{C_0}\,a,
\]
using $|\xi_{1,t}|\le Ca^2$ and
$|\xi_{2,t}|^2=\tr\Sigma_{22,t}$.
MF continuity and the GD update from the last admissible pair
therefore exclude a first exit at or before $T_*$, including the
extra iterate needed for the final pair. Boundedness permits MF
continuation, so $T>T_*$ and the comparison holds on $[0,T_*]$.
\end{proof}

\paragraph{Adding a tangential potential.}
A tangential potential adds the drift $-\eta\nabla_1V_\eta$.
We require it to be $O(a^2)$ per GD step, like the curvature-induced
drift, and assume a bounded tangential Hessian to preserve stability.

\begin{theorem}[GD--MF comparison with a tangential potential]
\label{thm:mf-pair-potential}
Keep the assumptions of \cref{prop:mf-pair-direct-validity}, except
that $V_\eta:=E_\eta(\cdot,0)$ need not vanish and satisfies
\begin{equation}\label{eq:mf-pair-potential-assumptions}
 |\nabla_1V_\eta|\le C\delta,\qquad
 \|D_1^2V_\eta\|_{\rm op}\le C
 \qquad\text{for }|\theta_1|\le\delta/C_0.
\end{equation}
With the pair moments and common MF initial data computed from this
energy, the conclusions of \cref{prop:mf-pair-direct-validity} remain valid.
\end{theorem}

\begin{proof}
Only the tangential drift changes at leading order.
Indeed, $\eta|\nabla_1V_\eta|\le Ca^2$ preserves the fluctuation
and Taylor bounds of \cref{lem:mf-pair-taylor}. The sharp-mean
estimate is unchanged because $V_\eta$ is independent of $\theta_2$.
For MF, the additional block $D_1^2V_\eta=O(1)$ changes only the
constants in the covariance inequalities of
\cref{lem:mf-pair-small-components}.
The rescaled GD mean therefore satisfies
\begin{equation}\label{eq:mf-pair-potential-update}
 \bar\theta_{1,t+1}
 =\bar\theta_{1,t}
   -\frac\eta\delta\nabla_1V_\eta(\delta\bar\theta_{1,t})
   -\frac{\eta^2}{2}\nabla_1H_\eta(\delta\bar\theta_{1,t})
                         :\bar\Sigma_{22,t}
   +\bar R_t^\theta.
\end{equation}
The covariance update and remainder bounds are unchanged.
In the normalized MF mean equation, the additional covariance
correction is $-D_1^3V_\eta:\Sigma_{11}/(2\delta)=O_T(a^2)$,
so it fits the remainder in \cref{lem:mf-pair-rescaled-defects}.
The leading drift $-\nabla_1V_\eta(\delta\bar\theta_1)/\delta$
is bounded and uniformly Lipschitz, since its derivative is
$-D_1^2V_\eta$. The Gr\"onwall estimate and bootstrap of
\cref{prop:mf-pair-direct-validity} therefore apply with the same
error bound and a uniform $T_*>0$.
\end{proof}

\paragraph{Larger even windows.}
An even window is a union of adjacent GD pairs. Its covariance
is the average of their covariances plus the covariance of their means.
This identity extends the pair estimate without a new Taylor expansion.

\begin{corollary}[GD--MF comparison for even windows]
\label{cor:mf-even-window-direct-validity}
Under the hypotheses of \cref{thm:mf-pair-potential}, let $m\ge2$
be even, possibly depending on $\eta,a$, and define
\[
 \theta_t^{(m)}=\frac1m\sum_{j=0}^{m-1}\tilde\theta_{t+j},\qquad
 \Sigma_t^{(m)}=\frac1m\sum_{j=0}^{m-1}
 (\tilde\theta_{t+j}-\theta_t^{(m)})
 (\tilde\theta_{t+j}-\theta_t^{(m)})^T.
\]
Initialize full MF at these window moments, and normalize both
processes by $\bar\theta=\theta/\delta$, $\bar\Sigma=\Sigma/a^2$.
There exist $T_*,\varepsilon_*,C_*>0$, independent of $m,\eta,a$,
such that, whenever $m(\eta+\delta)\le\varepsilon_*$, both dynamics
exist on the required intervals and
\begin{equation}\label{eq:mf-even-window-direct-comparison}
 \sup_{0\le\tau\le T_*}\left(
 |\bar\theta_{\lfloor\tau/\eta\rfloor}^{(m)}-\bar\theta(\tau)|
 +\|\bar\Sigma_{\lfloor\tau/\eta\rfloor}^{(m)}-\bar\Sigma(\tau)\|_F
 \right)\le C_*m(\eta+\delta).
\end{equation}
The preparation condition remains imposed on the first GD pair;
the window length may grow provided $m(\eta+\delta)\to0$.
\end{corollary}

\begin{proof}
Use superscript $(2)$ for the pair moments. Partitioning the window
into $m/2$ pairs gives the exact identities
\[
 \begin{aligned}
 \theta_t^{(m)}
 &=\frac2m\sum_{j=0}^{m/2-1}\theta_{t+2j}^{(2)},\\
 \Sigma_t^{(m)}
 &=\frac2m\sum_{j=0}^{m/2-1}\left[\Sigma_{t+2j}^{(2)}
 +(\theta_{t+2j}^{(2)}-\theta_t^{(m)})
  (\theta_{t+2j}^{(2)}-\theta_t^{(m)})^T\right].
 \end{aligned}
\]
The tangential pair mean changes by $O(a^2)$ per step, the sharp
mean stays $O_T(a^2)$, and the sharp covariance changes by
$O_T(\eta a^2)$ per step. Substitution into these identities gives
\[
 \begin{gathered}
 |\theta_{2,t}^{(m)}|\le C_Ta^2,\qquad
 \|\Sigma_{11,t}^{(m)}\|_F\le C_Tm^2a^4,\qquad
 \|\Sigma_{12,t}^{(m)}\|_F\le C_T(a^3+ma^4),\\
 |\bar\theta_{1,t}^{(m)}-\bar\theta_{1,t}^{(2)}|\le C_Tm\eta,\qquad
 \|\bar\Sigma_{22,t}^{(m)}-\bar\Sigma_{22,t}^{(2)}\|_F
 \le C_T(m\eta+a^2),
 \end{gathered}
\]
with constants independent of $m$. The last $a^2$ term comes from
the covariance of the sharp pair means, divided by $a^2$.

Since $2ma\le m(\eta+\delta)$, the larger initial tangential
covariance is still small compared with $a^2$.
For MF initialized at the window moments, the differential
inequalities of \cref{lem:mf-pair-small-components} give
\[
 |\theta_2(\tau)|\le C_Ta^2,\qquad
 \|\Sigma_{12}(\tau)\|_F\le C_T(a^3e^{-\tau/\eta}+a^4),\qquad
 \|\Sigma_{11}(\tau)\|_F\le C_Tm^2a^4.
\]
The mixed block retains its damping on slow times of order $\eta$;
the factor $m^2$ affects only the tangential covariance bound.
Consequently, the normalized mean remainder becomes $O_T(m^2a^2)$
(including $D_1^3V_\eta:\Sigma_{11}/(2\delta)$), while the
sharp-covariance remainder remains $O_T(\delta)$.

Compare this MF solution with the GD pair moments, using the
normalized error $e_t$ from the proof of
\cref{prop:mf-pair-direct-validity}. The initial window--pair
bounds give $e_0\le Cm$. The leading drift has the same Lipschitz
bound, and $m^2a^2\ll1$ preserves bounded velocities and the
quadrature estimate. Gr\"onwall therefore yields
\[
 \sup_t e_t\le C_T\left[m+\frac{m^2a^2}{\eta+a}\right]
 \le C_Tm,
\]
since $ma^2/(\eta+a)\le m\delta\le\varepsilon_*$.
Returning to the window moments and including the small blocks gives
\eqref{eq:mf-even-window-direct-comparison}, using $2a\le\eta+\delta$ and
$m^2a^2=(m\eta)(m\delta)\le\varepsilon_*m\eta$.

To close the bootstrap, note that
$\|\Sigma_{11}\|_F/a^2\le C_T(ma)^2\ll1$; the other small-block
bounds also strictly improve the stopping bounds.
Choose $T_*$ below half the pair-comparison horizon, then
$\varepsilon_*$ small enough to accommodate the initial window's
tangential shift and the extra slow-time interval $m\eta$ needed
for the final window. This gives a common interval $[0,T_*]$.
Returning to iteration time proves \cref{thm:mf-sharp-valley-validity}.
\end{proof}

\paragraph{Why the window must be even.}
An odd window leaves an uncancelled alternating component.
For $E_\eta(\theta_2)=\theta_2^2/\eta$ and
$\tilde\theta_{2,0}=a$, an odd window has mean $(-1)^ta/m$,
whereas its window-initialized MF mean is $(a/m)e^{-2\tau/\eta}$.
Thus at $\tau=\eta$ the normalized mean error equals
$(1+e^{-2})\eta/(ma)$.
With $a=\eta$, this error is $(1+e^{-2})/m$ and does not vanish
for fixed odd $m$, although $\delta\to0$.

\paragraph{Rank of established oscillations.}
The pair covariance $\Sigma^{(2)}_t=\xi_t\xi_t^\top$ has rank at most
one. The window estimates above give
\begin{equation}\label{eq:rank-one-first-order}
    \big\|\bar\Sigma^{(m)}_t-\operatorname{diag}\big(0,\bar\Sigma^{(2)}_{22,t}\big)\big\|_F\le C_T\big(m\eta+a+m^2a^2+ma^2\big)\le C\,m(\eta+\delta),
\end{equation}
and \cref{thm:mf-sharp-valley-validity} gives the same bound for MF.
Thus, for $0\le\tau\le T_*$, both
$\bar\Sigma^{(m)}_{\lfloor\tau/\eta\rfloor}$ and $\bar\Sigma(\tau)$
stay within $C_*m(\eta+\delta)$ of a matrix of rank at most one.
For established oscillations, $|\xi_0|\asymp a$ and
$\|\bar\Sigma_0\|_F\asymp1$, so the additional MF covariance
directions are of lower order (\cref{rem:mf-rank}).

In the scaling limit of \cref{cor:mf-sharp-valley-limit}, this
factorization is exact. Write
$\bar\Sigma_{22}^0=\zeta_0\zeta_0^\top$ and solve
$\dot\zeta=\bar H(\bar\theta_1)\zeta$ with $\zeta(0)=\zeta_0$.
Differentiating $\zeta\zeta^\top$ gives the covariance equation in
\eqref{eq:mf-sharp-valley-limit}; hence
$\bar\Sigma_{22}(\tau)=\zeta(\tau)\zeta(\tau)^\top$.
Its rank remains one if $\zeta_0\ne0$, and zero otherwise.

\paragraph{Numerical check.}
In our MLP runs (seed~0, $\eta\in\{0.0025,0.004\}$, one- and
two-mode initializations), let $\widehat X_t=U_t^\top\widehat\Sigma_tU_t$
be the 16-iterate GD covariance projected on the three tracked leading
Hessian directions. The fraction outside its principal direction,
$(\lambda_2(\widehat X_t)+\lambda_3(\widehat X_t))/\tr\widehat X_t$,
stays below $0.4\%$ at every sample (\cref{fig:rank}), even for
two-mode initializations. This check does not measure fluctuations
outside the tracked subspace.

\paragraph{Why MF still retains several directions.}
The rank-one limit concerns established oscillations, not their onset.
In the onset experiment, drift dominates the initial window covariance
and the critical Hessian direction carries little variance
(\cref{fig:rod-onset-geometry}(b)). Truncating to the principal
covariance direction can discard this small component. MF retains
it, allowing it to grow and become dominant as oscillations develop
(\cref{app:onset-failure}); at fixed Hessian, its variance grows at
rate $\eta\lambda_1(\eta\lambda_1-2)$. 

Last but not least, this rank-one reduction only applies to the sharp valley picture we consider here, so retaining several directions might still be an advantage in general situations. 

\begin{figure}[ht!]
  \centering
  \includegraphics[width=\linewidth]{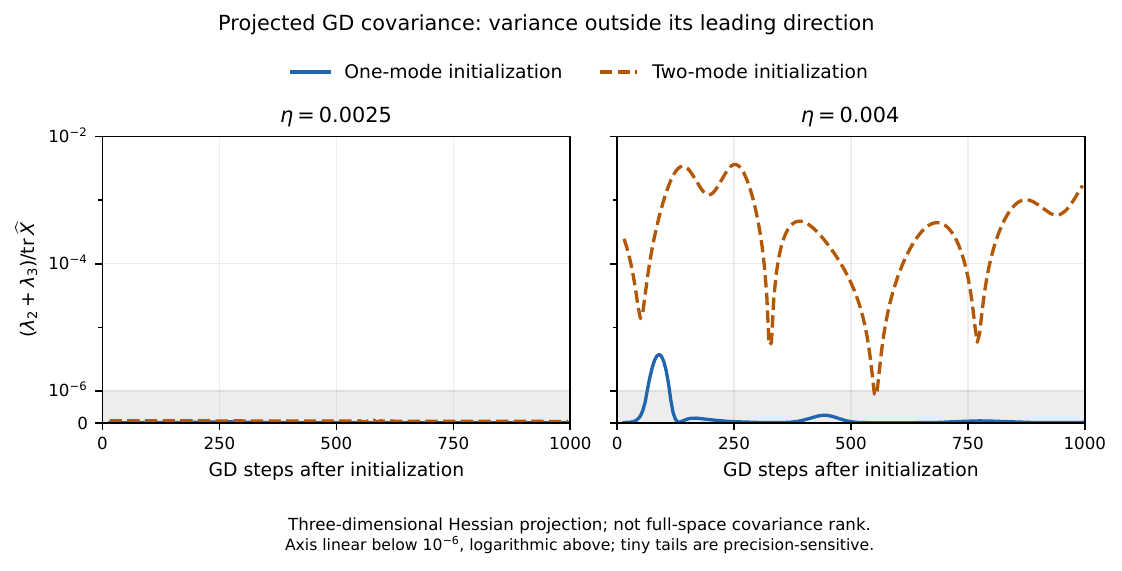}
  \caption{Fraction of the projected GD window covariance outside its principal direction (MLP, windows of 16 iterates, sampled every five steps). The vertical axis is linear below $10^{-6}$ (shaded) and logarithmic above; values in this range are at float32 precision.}
  \label{fig:rank}
\end{figure}
\FloatBarrier

\paragraph{A parameter-free scaling limit.}
The preceding comparison allows any relative scaling of $a$ and $\eta$.
The additional choice $a=\eta$ gives a limiting system in which the
mean and covariance can remain coupled, provided the rescaled landscape
has a limit.

\begin{corollary}[Coupled sharp-valley limit]
\label{cor:mf-sharp-valley-limit}
Under the hypotheses of \cref{cor:mf-even-window-direct-validity},
assume $a=\eta\to0$, so that $\delta=\eta$, and let the even window
length satisfy $m\eta\to0$. For $|y|<1/C_0$, define
\begin{equation}\label{eq:mf-sharp-valley-profiles}
 \bar V_\eta(y)
 =\frac{V_\eta(\eta y)-V_\eta(0)}{\eta^2},\qquad
 \bar H_\eta(y)=H_\eta(\eta y)-\frac2\eta\Id_r.
\end{equation}
Suppose $\bar V_\eta\to\bar V$ and $\bar H_\eta\to\bar H$
locally in $C^1$, and that the initial GD pair covariance satisfies
$\Sigma_{22,0}^{(2)}/\eta^2\to\bar\Sigma_{22}^0$.
Then, on a common interval $[0,T_*]$ independent of $\eta,m$, the
normalized GD window moments and their associated MF solution
converge uniformly to
\[
 (\bar\theta_1,0),\qquad
 \begin{pmatrix}0&0\\0&\bar\Sigma_{22}\end{pmatrix},
\]
where the GD moments are evaluated at $\lfloor\tau/\eta\rfloor$ and
the normalization is $\bar\theta=\theta/\eta$,
$\bar\Sigma=\Sigma/\eta^2$. The limiting components solve
\begin{equation}\label{eq:mf-sharp-valley-limit}
 \begin{aligned}
 \frac{\dd\bar\theta_1}{\dd\tau}
 &=-\nabla\bar V(\bar\theta_1)
   -\frac12\nabla\bar H(\bar\theta_1):\bar\Sigma_{22},\\
 \frac{\dd\bar\Sigma_{22}}{\dd\tau}
 &=\bar H(\bar\theta_1)\bar\Sigma_{22}
   +\bar\Sigma_{22}\bar H(\bar\theta_1),\\
 \bar\theta_1(0)&=0,\qquad
 \bar\Sigma_{22}(0)=\bar\Sigma_{22}^0.
 \end{aligned}
\end{equation}
All derivatives of $\bar V,\bar H$ are with respect to the
rescaled tangential coordinate $y$; in particular,
$(\nabla\bar H:S)_j=\tr((\partial_j\bar H)S)$.
\end{corollary}

\begin{proof}
First consider MF initialized at the GD pair moments.
By \cref{lem:mf-pair-rescaled-defects} and the potential correction
in the proof of \cref{thm:mf-pair-potential}, its retained equations are
\[
 \begin{aligned}
 \frac{\dd\bar\theta_1}{\dd\tau}
 &=-\nabla\bar V_\eta(\bar\theta_1)
   -\frac12\nabla\bar H_\eta(\bar\theta_1):\bar\Sigma_{22}
   +O_T(\eta),\\
 \frac{\dd\bar\Sigma_{22}}{\dd\tau}
 &=\bar H_\eta\bar\Sigma_{22}+\bar\Sigma_{22}\bar H_\eta
   +\eta\bar H_\eta\bar\Sigma_{22}\bar H_\eta+O_T(\eta),
 \end{aligned}
\]
where $\bar H_\eta$ is evaluated at $\bar\theta_1$.
The first equation uses
$\eta\nabla_1H_\eta(\eta y)=\nabla\bar H_\eta(y)$.
In the second, substituting $H_\eta=2\Id_r/\eta+\bar H_\eta$
into $-H_\eta\bar\Sigma_{22}-\bar\Sigma_{22}H_\eta
+\eta H_\eta\bar\Sigma_{22}H_\eta$ cancels the singular terms.
Thus the Hessian-quadratic contribution is essential to the limit.

The retained vector fields have a uniform Lipschitz bound on the
bounded region constructed by the bootstrap. Their coefficients
converge there by assumption, and the remaining drift terms tend
uniformly to zero. Writing the equations in integral form and
applying Gr\"onwall's lemma therefore gives convergence to
\eqref{eq:mf-sharp-valley-limit}, on a possibly smaller uniform
interval. The small-component bounds in
\cref{lem:mf-pair-small-components} show that all other normalized
blocks vanish. The pair GD moments have the same limit by
\cref{thm:mf-pair-potential}.

Finally, the window estimates in the proof of
\cref{cor:mf-even-window-direct-validity} bound the difference
between the normalized pair and even-window GD moments by
$C_Tm\eta$. That corollary gives the same order of error between
the window moments and their associated MF solution.
Since $m\eta\to0$, both have the asserted limit and the same
limiting initial data.
\end{proof}

Under the present derivative assumptions, the limiting curvature
is necessarily affine. Indeed,
\[
 \|D_y^2\bar H_\eta\|_{\rm op}
 =\eta^2\|D_1^2H_\eta(\eta y)\|_{\rm op}\le C\eta,
 \qquad
 \|D_y^4\bar V_\eta\|_{\rm op}\le C\eta.
\]
Taylor's formula therefore shows that $\bar H$ is affine and
$\bar V$ is a polynomial of degree at most three locally.
An affine $\bar H$ still allows two-way coupling: the covariance
drives the mean, which changes the curvature governing covariance
growth when $\bar H$ is nonconstant. A nonlinear limiting curvature would require
weaker bounds on the physical derivatives and a new remainder analysis.

\subsection{Formal derivation of the mean--fluctuation dynamics}
\label{app:mf-derivation}

Let $L_3,L_4$ be the sup norms of the third and fourth derivatives of $E$, respectively. 
Taylor expansion in \cref{eq:exact-mf-mean} gives
\begin{align}
    \theta_{t+1}-\theta_t
    &=-\eta\left\langle
      \nabla E(\theta_t)+H(\theta_t)\xi_t
      +\frac12\nabla H(\theta_t):\xi_t\xi_t^T
      \right\rangle_m+R_t^\theta \nonumber\\
    &=-\eta\nabla E(\theta_t)-\frac\eta2\nabla H(\theta_t):\Sigma_t
      +R_t^\theta,
    \label{eq:mf-discrete-mean-remainder}
\end{align}
where we used $\langle\xi_t\rangle_m=0$ and
$\langle\xi_t\xi_t^T\rangle_m=\Sigma_t$, and
\[
    |R_t^\theta|\le\frac{\eta L_4}{6}
       \left\langle|\xi_t|^3\right\rangle_m.
\]

We now turn to the fluctuation equation \eqref{eq:exact-mf-oscillation}.
Set $A_t=\Id-\eta H_t$ and
\[
    r_{t,j}=\nabla E(\theta_t+\xi_{t,j})
       -\nabla E(\theta_t)-H(\theta_t)\xi_{t,j},
    \qquad \bar r_t=\langle r_t\rangle_m.
\]
Then $|r_{t,j}|\le L_3|\xi_{t,j}|^2/2$ and
\[
    \xi_{t+1,j}=A_t\xi_{t,j}-\eta(r_{t,j}-\bar r_t).
\]
Taking the outer product and averaging over $j$, with
$\langle\xi_t\rangle_m=0$, yields exactly
\begin{align}
    \Sigma_{t+1}
    &=A_t\Sigma_t A_t^T
      -\eta\left\langle A_t\xi_t r_t^T+r_t\xi_t^T A_t^T\right\rangle_m
      +\eta^2\left(\langle r_t r_t^T\rangle_m-\bar r_t\bar r_t^T\right).
    \label{eq:mf-discrete-covariance-remainder}
\end{align}
Expanding $A_t\Sigma_t A_t^T$, we obtain
\begin{equation}
    \Sigma_{t+1}-\Sigma_t
    =-\eta(H_t\Sigma_t+\Sigma_tH_t)+\eta^2H_t\Sigma_tH_t+R_t^\Sigma,
    \label{eq:mf-closed-covariance-increment}
\end{equation}
with the Frobenius-norm bound
\begin{equation}
    \|R_t^\Sigma\|
    \le\eta L_3\|A_t\|_{\mathrm{op}}\left\langle|\xi_t|^3\right\rangle_m
       +\frac{\eta^2L_3^2}{4}\left\langle|\xi_t|^4\right\rangle_m.
    \label{eq:mf-window-closure-error}
\end{equation}

Neglecting higher-order terms, we have arrived at the coupled updates
\begin{align}
    \theta_{t+1} &= \theta_t -\eta\nabla E(\theta_t)-\frac\eta2\nabla H(\theta_t):\Sigma_t, \\
    \Sigma_{t+1} &= (\Id - \eta H(\theta_t))\Sigma_t(\Id - \eta H(\theta_t)).
\end{align}

\paragraph{Role of the window length.}
Notice that no $m\to\infty$ limit has been taken: the window should only average the fast phase so that the drift is slow between $t$ and $t+1$. 
The value of $m$ remains implicit in $\theta$ and $\Sigma$ as well as the remainders. 

\paragraph{Continuous-time flows.} 
Neglecting higher-order terms, and assuming that the window has been chosen so that the evolution of $\theta$ is slow between $t$ and $t+1$, one obtains the raw mean--fluctuation updates:
\begin{align}
    \frac{d\theta}{dt} &= -\eta \nabla E(\theta) - \frac{\eta}{2} \nab H(\theta) : \Sigma_t, \\
    \Sigma_{t+1} &= (\Id - \eta H(\theta_t))\Sigma_t(\Id - \eta H(\theta_t)).
\end{align}
The diagnostics in \cref{app:experiments} test these approximations on individual trajectories, while \Cref{prop:mf-validity} controls the raw-to-coarse covariance error along a prescribed Hessian path with a fixed invariant subspace.

\subsection{Closure validity and a two-cycle counterexample}
\label{app:mf-closure-validity}

Small oscillations alone do not guarantee a long-time approximation by MF.
The following example has a constant windowed GD covariance, whereas both MF and raw MF predict growing covariance.

Let $\eta=1$ and $E(x)=(2+\eps)x^2/2-x^4/4$, with $0<\eps<1$.
The GD update is
\[
    x_{n+1}=T(x_n),
    \qquad T(x)=x-E'(x)=-(1+\eps)x+x^3.
\]
At $x=\pm\sqrt{\eps}$, the identity $x^3=\eps x$ gives $T(x)=-x$.
Thus, starting from $x_0=\sqrt{\eps}$, the exact trajectory is
\[
    x_n=(-1)^n\sqrt{\eps},
    \qquad x_{2k}=\sqrt{\eps},\quad x_{2k+1}=-\sqrt{\eps}.
\]
The two values are not fixed points: GD exchanges them at each step, while each parity subsequence remains constant.
Since $T'(\pm\sqrt{\eps})=-1+2\eps$, the derivative of the two-step map at either point is $(-1+2\eps)^2<1$, so this cycle is locally attracting.
For any even window length $m$, each window contains equally many positive and negative iterates; hence its mean and covariance are exactly
\[
    \theta_n=0,\qquad \Sigma_n=\eps.
\]

Initialize MF and raw MF with these same values, $(\theta_0,\Sigma_0)=(0,\eps)$.
Their mean remains zero because $E'(0)=E'''(0)=0$.
The Hessian at this mean is $H(0)=2+\eps$, so their covariance equations reduce to
\[
    \frac{d\Sigma_{\rm MF}}{dt}=(2+\eps)\eps\Sigma_{\rm MF},
    \qquad
    \Sigma_{{\rm raw},n+1}=(1+\eps)^2\Sigma_{{\rm raw},n}.
\]
Consequently,
\[
    \Sigma_{\rm MF}(t)=\eps e^{(2+\eps)\eps t},
    \qquad
    \Sigma_{{\rm raw},n}=\eps(1+\eps)^{2n}.
\]

The discrepancy comes from the cubic term discarded by the fluctuation closure.
On the exact orbit, this term cancels the excess amplification in the linearized update:
\[
    -(1+\eps)x_n+x_n^3=-(1+\eps)x_n+\eps x_n=-x_n.
\]
Without this cancellation, raw MF multiplies the variance by $(1+\eps)^2$ at each step, and MF predicts the corresponding exponential growth.
Both relative covariance errors $|\Sigma-\eps|/\eps$ vanish as $\eps\to0$ over a fixed time interval.
However, at $t=c/\eps$ and $n=\lfloor c/\eps\rfloor$, respectively, they tend to $e^{2c}-1$ for every fixed $c>0$.
Thus the relative error becomes order one on the timescale $\eps^{-1}$, despite the vanishing oscillation amplitude $\sqrt{\eps}$.
Nevertheless, their mutual gap satisfies $|\Sigma_{\rm MF}(n)-\Sigma_{{\rm raw},n}|/\eps\to0$ at $n=\lfloor c/\eps\rfloor$: raw/coarse agreement does not imply relative covariance accuracy against GD.

\subsection{Well-posedness and positivity}

\begin{proposition}[Well-posedness, positivity, and dissipation]
\label{prop:mf-well-posedness}
    Let $E\in C_b^4(\R^d)$, $\eta>0$, and
    $(\theta_0,\Sigma_0)\in\R^d\times\R^{d\times d}$. Then \eqref{eq:MF}
    has a unique global solution. If $\Sigma_0$ is symmetric positive
    semidefinite, so is $\Sigma_t$ for every $t\ge0$. Moreover, we have for all $t\ge 0$,
    \begin{equation}
        \frac{d}{d t}F(\theta_t,\Sigma_t)
        =-\eta\|\nabla_\theta F(\theta_t,\Sigma_t)\|^2
        -\frac\eta2
        \big[H(\theta_t)(2\Id-\eta H(\theta_t))H(\theta_t)\big]:\Sigma_t.
        \label{eq:mf-dissipation}
    \end{equation}
    In particular, the free energy is nonincreasing in the stable regime where $\lambda_{\max}(H(\theta_t))\le 2/\eta$.
\end{proposition}
\begin{proof}[Proof of \cref{prop:mf-well-posedness}]
    The vector field is locally Lipschitz. Since the derivatives of $E$ are
    bounded,
    \[
        \|\dot\Sigma\|\le C\|\Sigma\|,
        \qquad
        |\dot\theta|\le C(1+\|\Sigma\|).
    \]
    Gronwall's inequality prevents finite-time blow-up and gives global
    well-posedness. Symmetry is preserved. If $\Sigma\ge0$ and
    $v\in\ker\Sigma$, then
    \[
        v^T\dot\Sigma v
        =\eta^2(Hv)^T\Sigma(Hv)\ge0.
    \]
    The invariance criterion for closed convex sets therefore preserves the
    positive semidefinite cone.

    It remains to differentiate the free energy. Since
    $\dot\theta=-\eta\nabla_\theta F$,
    \begin{align*}
        \frac{\dd F}{\dd t}
        &=\nabla_\theta F\cdot\dot\theta+\frac12H:\dot\Sigma\\
        &=-\eta\|\nabla_\theta F\|^2
          -\frac\eta2\big[H(2\Id-\eta H)H\big]:\Sigma.
    \end{align*}
    Under the sharpness bound, $H(2\Id-\eta H)H$ is positive semidefinite;
    its Frobenius product with $\Sigma\ge0$ is therefore nonnegative. This
    proves \eqref{eq:mf-dissipation}.

\end{proof}

\subsection{Rank and quantitative dissipation}
\label{app:mf-rank-proof}

\paragraph{Rank of the fluctuation covariance.} The rank of $\Sigma$ is generally not preserved along the MF flow: the following example starts from rank one and becomes full rank immediately. In contrast, the flat-landscape limit preserves rank throughout its existence interval (\cref{prop:mf-limits}).

Take $d=2$, $\eta=1$, and a smooth bounded
function which agrees near the origin with
\[
    E(x,y)=xy+\frac{x^3}{6}.
\]
With $\theta_0=0$ and $\Sigma_0=\operatorname{diag}(1,0)$, one has
\[
    H(0)=\begin{pmatrix}0&1\\1&0\end{pmatrix},
    \qquad
    \nabla H(0):\Sigma_0=(1,0).
\]
Consequently $\dot\theta_0=(-1/2,0)$ and
\[
    \dot\Sigma_0
    =-(H(0)\Sigma_0+\Sigma_0H(0))+H(0)\Sigma_0H(0)
    =\begin{pmatrix}0&-1\\-1&1\end{pmatrix}.
\]
Thus
\[
    \Sigma_t=\begin{pmatrix}1&-t\\-t&t\end{pmatrix}+O(t^2),
    \qquad
    \det\Sigma_t=t+O(t^2)>0,
\]
and the covariance becomes full rank immediately.

\paragraph{Modified Polyak-\L ojasiewicz inequality.}

Suppose that $\eta \lambda_{\max}(H)\le 2$, that $F\ge F_\star$ along the flow, and
that we have the following \emph{modified Polyak-\L ojasiewicz inequality}: 
\[
    \exists \mu>0,\forall t\ge0, \quad \|\nabla_\theta F(\theta_t,\Sigma_t)\|^2
    \ge \frac\mu\eta\big(F(\theta_t,\Sigma_t)-F_\star\big).
\]
Then, for all $t\ge 0$, we have the exponential convergence
$F(\theta_t,\Sigma_t)-F_\star
    \le e^{-\mu t}(F(\theta_0,\Sigma_0)-F_\star)$.

The claim follows directly from
\eqref{eq:mf-dissipation}:
\[
    \frac{\dd}{\dd t}(F-F_\star)
    \le-\eta\|\nabla_\theta F\|^2
    \le-\mu(F-F_\star).
\]
Gronwall's inequality gives the stated rate.

\subsection{Stability of stationary states}
\label{app:mf-stationary-stability}

In this section, we analyse the stability (or instability) of stationary states of the MF dynamics. 

Roughly speaking, the stationary states can be classified into two categories: non-oscillating states (i.e. $\Sigma_\star=0$), which can be stable or unstable depending on the sharpness of the Hessian, and oscillating states (i.e. $\Sigma_\star\neq 0$). To each oscillating stationary state, we can associate a physically meaningful state for which the covariance is purely in the EoS subspace. Their stability is analysed below in terms of the Hessian outside the EoS subspace and the effective Hessian of the free energy.

\paragraph{On the reduction to purely EoS stationary states.}

There are two distinct statements behind the reduction to a pure EoS
covariance.  The first is an exact property of the MF stationary equations:
\cref{lem:mf-stationary-decomposition} shows that
\[
    (\theta_\star,\Sigma_{0,\star}+\Sigma_{2,\star})
    \ \text{stationary}
    \quad\Longrightarrow\quad
    (\theta_\star,\Sigma_{2,\star})\ \text{stationary}.
\]
Thus every stationary MF state has a pure EoS representative with the same
mean and the same EoS covariance block. However, this does not
assert that either covariance is generated by a temporal window of GD iterates.

To fix ideas, consider the \emph{frozen homogeneous
linear model}
\begin{equation}
    \xi_{k+1}=A_\star\xi_k,
    \qquad A_\star=\Id-\eta H_\star.
    \label{eq:mf-frozen-gd-model}
\end{equation}

For $z_k=P_0\xi_k$ and $w_k=P_2\xi_k$, the spectral identities
$P_0A_\star=P_0$ and $P_2A_\star=-P_2$ give
\[
    z_{k+1}=z_k,
    \qquad
    w_{k+1}=-w_k.
\]
The constant flat component has zero centered temporal covariance, whereas
the EoS component is period two.  This motivates the modeling restriction
\begin{equation}
    \Sigma_{0,\star}=0,
    \qquad\text{equivalently}\qquad
    \operatorname{Ran}\Sigma_\star\subset V_2,
    \label{eq:mf-physical-stationary-branch}
\end{equation}
but does not prove it for temporal windows of nonlinear GD.  In particular,
it is a selection principle, not a realizability theorem.  

We prove \cref{prop:mf-locally-convex-stability}, stated in the main text. We use $\operatorname{Sym}(V)$ for the self-adjoint maps on a subspace $V$.

The proof is split into three lemmas. \Cref{lem:mf-stationary-decomposition} addresses the covariance decomposition alluded to above, while \cref{lem:mf-zero-covariance-stability,lem:mf-pure-eos-stability} address zero-covariance and pure EoS states, respectively.

We first restate the main theorem of this section:

\begin{theorem}[Stability in the locally convex regime]
    Let $(\theta_\star,\Sigma_\star)$ be stationary for \eqref{eq:MF}, with
    $\Sigma_\star\ge0$, and suppose that $E$ is locally convex near
    $\theta_\star$. Set
    \[
        H_\star=H(\theta_\star),
        \qquad
        V_0=\ker H_\star,
        \qquad
        V_2=\ker(2\Id-\eta H_\star),
    \]
    and let $P_0,P_2$ be the corresponding orthogonal projections. We write
    $\operatorname{Sym}(V)$ for the self-adjoint linear maps from $V$ to
    itself.

    If $\Sigma_\star=0$, then $\nabla E(\theta_\star)=0$. The state is
    stable if and only if $0\le\eta H_\star\le2\Id$, and it is strictly
    stable if and only if $0<\eta H_\star<2\Id$. If
    $\lambda_{\max}(\eta H_\star)>2$, it has an exponentially growing
    covariance perturbation.

    Suppose now that $\Sigma_\star\ne0$. Then
    \begin{equation}
        \operatorname{Ran}\Sigma_\star\subset
        \ker H_\star\oplus\ker(2\Id-\eta H_\star),
    \end{equation}
    and the fluctuations in the flat and EoS subspaces are decorrelated, in
    the sense that
    \[
        u^T\Sigma_\star w=0
        \qquad\text{for every }u\in\ker H_\star
        \text{ and }w\in\ker(2\Id-\eta H_\star).
    \]
    Consequently,
    \[
        \Sigma_\star=\Sigma_{0,\star}+\Sigma_{2,\star},
        \qquad
        \Sigma_{0,\star}=P_0\Sigma_\star P_0,
        \qquad
        \Sigma_{2,\star}=P_2\Sigma_\star P_2.
    \]
    Moreover, $(\theta_\star,\Sigma_{2,\star})$ is also stationary. Thus every
    stationary state can be assigned a pure EoS representative with the same
    $\theta_\star$, and the stability analysis below focuses on these
    representatives.

    Finally, consider a nonzero pure EoS representative and assume that
    \[
        0\ne\Sigma_\star\ge0,
        \qquad
        \operatorname{Ran}\Sigma_\star\subset V_2,
        \qquad
        H_\star^{\mathrm{eff}}
        :=\nabla_\theta^2F(\theta_\star,\Sigma_\star)
        =H_\star+\frac12\nabla^2H(\theta_\star):\Sigma_\star\ge0.
    \]
    If every eigenvalue of $\eta H_\star$ on $V_2^\perp$ lies in $[0,2)$,
    then the state is stable.
    Conversely, an eigenvalue of $\eta H_\star|_{V_2^\perp}$ larger than $2$
    gives an exponentially growing covariance perturbation and hence instability.
\end{theorem}

\begin{lemma}[Stationary covariance decomposition]
\label{lem:mf-stationary-decomposition}
Let $(\theta_\star,\Sigma_\star)$ be stationary for \eqref{eq:MF}, with
$\Sigma_\star\ge0$, and suppose that $E$ is locally convex near
$\theta_\star$.  Set
\[
    H_\star=H(\theta_\star),
    \qquad
    V_0=\ker H_\star,
    \qquad
    V_2=\ker(2\Id-\eta H_\star),
\]
and let $P_0,P_2$ be the corresponding orthogonal projections.  Then
\[
    \operatorname{Ran}\Sigma_\star\subset V_0\oplus V_2,
    \qquad
    u^T\Sigma_\star w=0
    \quad (u\in V_0,\ w\in V_2).
\]
Consequently,
\[
    \Sigma_\star=\Sigma_{0,\star}+\Sigma_{2,\star},
    \qquad
    \Sigma_{0,\star}=P_0\Sigma_\star P_0,
    \qquad
    \Sigma_{2,\star}=P_2\Sigma_\star P_2.
\]
Moreover, $(\theta_\star,\Sigma_{2,\star})$ is stationary.
\end{lemma}

\begin{proof}[Proof of \cref{lem:mf-stationary-decomposition}]
Local convexity gives $H_\star=H(\theta_\star)\ge0$. Set
$A_\star=\Id-\eta H_\star$. The covariance equation factors as
\begin{equation}
    \dot\Sigma=(\Id-\eta H)\Sigma(\Id-\eta H)-\Sigma,
    \label{eq:mf-covariance-factorization}
\end{equation}
so covariance stationarity is equivalent to
\begin{equation}
    A_\star\Sigma_\star A_\star=\Sigma_\star.
    \label{eq:mf-stationary-covariance}
\end{equation}

We first prove the support inclusion. Let $U$ be an
orthogonal matrix such that
\[
    U^TH_\star U=D=\operatorname{diag}(\lambda_1,\ldots,\lambda_d).
\]
We distinguish the covariance from its matrix in this eigenbasis by writing
\[
    \widetilde\Sigma_\star=U^T\Sigma_\star U,
    \qquad
    \widetilde A_\star=U^TA_\star U
    =\operatorname{diag}(a_1,\ldots,a_d),
    \qquad a_i=1-\eta\lambda_i.
\]
Multiplying \eqref{eq:mf-stationary-covariance} on the left by $U^T$ and on
the right by $U$ gives
\[
    \widetilde A_\star\widetilde\Sigma_\star\widetilde A_\star
    =\widetilde\Sigma_\star.
\]
Therefore, entry by entry,
\begin{equation}
    (a_i a_j-1)(\widetilde\Sigma_\star)_{ij}=0,
    \qquad 1\le i,j\le d.
    \label{eq:mf-stationary-spectral}
\end{equation}
Taking $i=j$ yields
\[
    (a_i^2-1)(\widetilde\Sigma_\star)_{ii}=0.
\]
Since $\lambda_i\ge0$, the coefficient vanishes precisely when
$\lambda_i\in\{0,2/\eta\}$. Otherwise
$(\widetilde\Sigma_\star)_{ii}=0$. The matrix
$\widetilde\Sigma_\star$ is positive semidefinite. Hence each of its
$2\times2$ principal submatrices is positive semidefinite and has
nonnegative determinant:
\[
    \det\begin{pmatrix}
        (\widetilde\Sigma_\star)_{ii}
        & (\widetilde\Sigma_\star)_{ij}\\
        (\widetilde\Sigma_\star)_{ij}
        & (\widetilde\Sigma_\star)_{jj}
    \end{pmatrix}\ge0.
\]
Equivalently,
\[
    |(\widetilde\Sigma_\star)_{ij}|^2
    \le (\widetilde\Sigma_\star)_{ii}
       (\widetilde\Sigma_\star)_{jj}.
\]
Thus $(\widetilde\Sigma_\star)_{ii}=0$ implies
$(\widetilde\Sigma_\star)_{ij}=0$ for every $j$; by symmetry, the entire
$i$-th row and column vanish. Consequently, if
\[
    V=\operatorname{span}\{e_i:\lambda_i\in\{0,2/\eta\}\},
\]
then $\operatorname{Ran}\widetilde\Sigma_\star\subset V$.

We now return explicitly to the original coordinates. Since
$\Sigma_\star=U\widetilde\Sigma_\star U^T$, for every $x\in\R^d$,
\[
    \Sigma_\star x
    =U\widetilde\Sigma_\star(U^Tx)\in UV.
\]
Moreover, the columns $Ue_i$ are the eigenvectors of $H_\star$ with
eigenvalues $\lambda_i$. Therefore
\[
    UV
    =\operatorname{span}\{Ue_i:\lambda_i=0\text{ or }\lambda_i=2/\eta\}
    =\ker H_\star\oplus\ker(2\Id-\eta H_\star).
\]
Combining the last two displays proves the support inclusion. Finally, if
$\lambda_i=0$ and $\lambda_j=2/\eta$, then $a_i a_j=-1$, so
\eqref{eq:mf-stationary-spectral} gives
$(\widetilde\Sigma_\star)_{ij}=0$. The corresponding transposed entry also
vanishes. Hence, for every $u\in\ker H_\star$ and
$w\in\ker(2\Id-\eta H_\star)$, expansion in the eigenbasis $U$ gives
\[
    u^T\Sigma_\star w=0.
\]
When $\Sigma_\star=\E[\xi\xi^T]$, this identity reads
\[
    \E\big[(u^T\xi)(w^T\xi)\big]=0.
\]
Thus the fluctuations in the flat and EoS subspaces are decorrelated.

We now show that the flat block can be removed without changing the
stationary mean. Let $P_0$ and $P_2$ be the orthogonal projections onto
\[
    V_0=\ker H_\star,
    \qquad
    V_2=\ker(2\Id-\eta H_\star),
\]
respectively. The support and decorrelation properties proved above yield
\begin{equation}
    \Sigma_\star=\Sigma_{0,\star}+\Sigma_{2,\star},
    \qquad
    \Sigma_{0,\star}=P_0\Sigma_\star P_0,
    \qquad
    \Sigma_{2,\star}=P_2\Sigma_\star P_2.
    \label{eq:mf-flat-eos-decomposition}
\end{equation}
Fix $u\in V_0$. Local convexity implies
$g_u(\theta):=u^TH(\theta)u\ge0$ near $\theta_\star$, while
$g_u(\theta_\star)=0$. Therefore $\theta_\star$ is a local minimizer of
$g_u$ and
\begin{equation}
    \nabla g_u(\theta_\star)
    =\nabla^3E(\theta_\star)[u,u]=0.
    \label{eq:mf-flat-third-derivative}
\end{equation}
Take a spectral decomposition
$\Sigma_{0,\star}=\sum_r s_r u_r u_r^T$, with $s_r\ge0$ and $u_r\in V_0$.
Then
\begin{equation}
    H_\star\Sigma_{0,\star}
    =\Sigma_{0,\star}H_\star=0,
    \qquad
    \nabla H(\theta_\star):\Sigma_{0,\star}
    =\sum_r s_r\nabla g_{u_r}(\theta_\star)=0.
    \label{eq:mf-flat-component-identities}
\end{equation}
Since $A_\star=\Id$ on $V_0$, the covariance identity also gives
\[
    A_\star\Sigma_{0,\star}A_\star=\Sigma_{0,\star}.
\]
Subtracting these flat contributions from the stationary equations for
$(\theta_\star,\Sigma_\star)$ gives
\begin{align*}
    \nabla E(\theta_\star)
      +\frac12\nabla H(\theta_\star):\Sigma_{2,\star}
    &=0,\\
    A_\star\Sigma_{2,\star}A_\star
    &=A_\star\Sigma_\star A_\star
      -A_\star\Sigma_{0,\star}A_\star
      =\Sigma_\star-\Sigma_{0,\star}
      =\Sigma_{2,\star}.
\end{align*}
Thus $(\theta_\star,\Sigma_{2,\star})$ is stationary: removing the flat
block changes neither the stationary mean nor the EoS block.
\end{proof}

For the two stability lemmas below, we record the linearization at an
arbitrary stationary state $(\theta_\star,\Sigma_\star)$.  Set
\[
    H_\star=H(\theta_\star),
    \qquad
    A_\star=\Id-\eta H_\star,
    \qquad
    H_\star^{\mathrm{eff}}
    =H_\star+\frac12D^2H(\theta_\star):\Sigma_\star
    =\nabla_\theta^2F(\theta_\star,\Sigma_\star).
\]
If $x$ and $S=S^T$ are perturbations of $\theta_\star$ and $\Sigma_\star$,
respectively, differentiation of \eqref{eq:MF} gives
\begin{align}
    \dot x
    &=-\eta\left(H_\star^{\mathrm{eff}}x
        +\frac12\nabla H(\theta_\star):S\right),
        \label{eq:mf-convex-linearization-theta}\\
    \dot S
    &=A_\star S A_\star-S
      -\eta\big(DH(\theta_\star)[x]\Sigma_\star A_\star
                 +A_\star\Sigma_\star DH(\theta_\star)[x]\big).
        \label{eq:mf-convex-linearization-sigma}
\end{align}

\begin{lemma}[Zero-covariance states]
\label{lem:mf-zero-covariance-stability}
Let $(\theta_\star,0)$ be stationary for \eqref{eq:MF}, and suppose that $E$
is locally convex near $\theta_\star$.  Set
$H_\star=H(\theta_\star)$.  Then
$\nabla E(\theta_\star)=0$.  The state is stable if and only if
$0\le\eta H_\star\le2\Id$, and it is strictly stable if and only if
$0<\eta H_\star<2\Id$.  If
$\lambda_{\max}(\eta H_\star)>2$, its linearization has an eigenvalue with positive real
part.
\end{lemma}

\begin{proof}[Proof of \cref{lem:mf-zero-covariance-stability}]
The stationary mean equation gives $\nabla E(\theta_\star)=0$, while local
convexity gives $H_\star\ge0$.  Since $\Sigma_\star=0$, one has
$H_\star^{\mathrm{eff}}=H_\star$.  Choose an orthogonal matrix $U$ whose
columns $u_i=Ue_i$ form an eigenbasis of $H_\star$:
\[
    H_\star u_i=\lambda_i u_i,
    \qquad
    A_\star u_i=(1-\eta\lambda_i)u_i.
\]
Set
\[
    \widetilde x=U^Tx,
    \qquad
    \widetilde S=U^TSU,
    \qquad
    D=U^TH_\star U=\operatorname{diag}(\lambda_i),
    \qquad
    \widetilde A_\star=U^TA_\star U
      =\operatorname{diag}(1-\eta\lambda_i).
\]
Thus the vector equation is written in the eigenbasis by multiplying it on
the left by $U^T$, whereas the covariance equation is conjugated by
multiplying it on the left by $U^T$ and on the right by $U$.  Since
$S=U\widetilde S U^T$, this gives
\begin{align*}
    \dot{\widetilde x}
    &=-\eta D\widetilde x
      -\frac\eta2U^T\big(\nabla H(\theta_\star):
                          (U\widetilde S U^T)\big),\\
    \dot{\widetilde S}
    &=\widetilde A_\star\widetilde S\widetilde A_\star
      -\widetilde S.
\end{align*}
The covariance equation is diagonal entry by entry:
\[
    \dot{\widetilde S}_{ij}
    =\big((1-\eta\lambda_i)(1-\eta\lambda_j)-1\big)
       \widetilde S_{ij}.
\]
Consequently, the mean direction $u_i$ has rate $-\eta\lambda_i$, while the
symmetric covariance direction
$u_i u_j^T+u_j u_i^T
 =U(e_i e_j^T+e_j e_i^T)U^T$ has rate
\begin{equation}
    (1-\eta\lambda_i)(1-\eta\lambda_j)-1.
    \label{eq:mf-zero-covariance-rates}
\end{equation}
These are continuous-time MF rates, not the discrete GD multipliers
$1-\eta\lambda_i$.  The covariance equation does not depend on $x$, whereas
the covariance perturbation may force the mean through
$\nabla H(\theta_\star):S$.  The full linearized system is therefore block
triangular, so this forcing does not change its eigenvalues.
Since $\lambda_i\ge0$, all these spectral values are nonpositive exactly
when $0\le\eta H_\star\le2\Id$, and they are all strictly negative exactly
when $0<\eta H_\star<2\Id$. If some $\eta\lambda_i>2$, the diagonal
covariance perturbation $u_i u_i^T$ has the positive rate
\[
    (1-\eta\lambda_i)^2-1
    =\eta\lambda_i(\eta\lambda_i-2)>0.
\]
This gives the complete linear classification of zero-covariance states.
\end{proof}

\begin{lemma}[Pure EoS states]
\label{lem:mf-pure-eos-stability}
Let $(\theta_\star,\Sigma_\star)$ be stationary for \eqref{eq:MF}, and suppose
that $E$ is locally convex near $\theta_\star$.  Set
\[
    H_\star=H(\theta_\star),
    \qquad
    V_2=\ker(2\Id-\eta H_\star),
    \qquad
    H_\star^{\mathrm{eff}}
      =\nabla_\theta^2F(\theta_\star,\Sigma_\star).
\]
Assume that
\[
    0\ne\Sigma_\star\ge0,
    \qquad
    \operatorname{Ran}\Sigma_\star\subset V_2,
    \qquad
    H_\star^{\mathrm{eff}}\ge0.
\]
If every eigenvalue of $\eta H_\star|_{V_2^\perp}$ belongs to $[0,2)$,
then the state is stable.  Conversely, if this restriction has an eigenvalue
larger than $2$, the state is unstable.
\end{lemma}

\begin{proof}[Proof of \cref{lem:mf-pure-eos-stability}]
The proof uses coordinates adapted simultaneously to the Hessian and the
stationary covariance.  This does not diagonalize the whole linearization:
the third derivative of $E$ still couples its mean and covariance parts.
It does, however, make the sign of every spectral value accessible through
one weighted eigenmode identity.

Let $P$ and $Q$ be the orthogonal projections onto $V_2$ and $V_2^\perp$.
Since $A_\star:=\Id-\eta H_\star$ equals $-\Id$ on $V_2$, every covariance
perturbation $S=S^T$ decomposes as
\[
    S=B+Z+Z^T+C,
    \qquad
    B=PSP,\quad Z=PSQ,\quad C=QSQ.
\]
Projecting \eqref{eq:mf-convex-linearization-sigma} onto these blocks gives
\begin{align}
    \dot x
    &=-\eta\left(H_\star^{\mathrm{eff}}x
       +\frac12\nabla H(\theta_\star):(B+Z+Z^T+C)\right),
       \label{eq:mf-eos-coordinate-x}\\
    \dot B
    &=\eta\big(PDH(\theta_\star)[x]P\Sigma_\star
              +\Sigma_\star PDH(\theta_\star)[x]P\big),
       \label{eq:mf-eos-coordinate-B}\\
    \dot Z
    &=-Z(2\Id-\eta H_\star|_{V_2^\perp})
      +\eta\Sigma_\star PDH(\theta_\star)[x]Q,
       \label{eq:mf-eos-coordinate-Z}\\
    \dot C
    &=A_\perp C A_\perp-C,
       \qquad A_\perp=A_\star|_{V_2^\perp}.
       \label{eq:mf-eos-coordinate-C}
\end{align}

Assume first that the eigenvalues of
$\eta H_\star|_{V_2^\perp}$ belong to $[0,2)$.  Since
$H_\star=(2/\eta)\Id$ on $V_2$, we may choose an orthonormal basis $(e_i)$
of $V_2$ that diagonalizes $\Sigma_\star$:
\[
    \Sigma_\star e_i=\sigma_i e_i,
    \qquad \sigma_i\ge0.
\]
Choose also an orthonormal eigenbasis $(u_\alpha)$ of $V_2^\perp$ such that
\[
    H_\star u_\alpha=\lambda_\alpha u_\alpha,
    \qquad
    \delta_\alpha:=2-\eta\lambda_\alpha>0.
\]
For $G=DH(\theta_\star)[x]$, write
\[
    B_{ij}=e_i^TBe_j,
    \qquad Z_{i\alpha}=e_i^TZu_\alpha,
    \qquad G_{ij}=e_i^TGe_j,
    \qquad G_{i\alpha}=e_i^TGu_\alpha.
\]
Equations \eqref{eq:mf-eos-coordinate-B}--
\eqref{eq:mf-eos-coordinate-C} then read
\begin{align}
    \dot B_{ij}
      &=\eta(\sigma_i+\sigma_j)G_{ij},
       \label{eq:mf-eos-coordinate-Bij}\\
    \dot Z_{i\alpha}
      &=-\delta_\alpha Z_{i\alpha}+\eta\sigma_iG_{i\alpha},
       \label{eq:mf-eos-coordinate-Zia}\\
    \dot C_{\alpha\beta}
      &=\big((1-\eta\lambda_\alpha)(1-\eta\lambda_\beta)-1\big)
          C_{\alpha\beta}.
       \label{eq:mf-eos-coordinate-Cab}
\end{align}
In particular, the $C$-variables form an autonomous subsystem.  Their
rates are nonpositive because
$|1-\eta\lambda_\alpha|\le1$.

\paragraph{Full rank on the EoS subspace.}
Suppose first that $\operatorname{Ran}\Sigma_\star=V_2$, so every
$\sigma_i$ is positive.  The full linearized system is block triangular,
with the autonomous $C$-subsystem on the diagonal.  It therefore remains
to study the subsystem obtained by setting $C=0$.

Let $r\in\mathbb C$ be an eigenvalue of this subsystem, with nonzero
complex eigenvector $(x,B,Z)$.  Replace the time derivatives in
\eqref{eq:mf-eos-coordinate-x}, \eqref{eq:mf-eos-coordinate-Bij}, and
\eqref{eq:mf-eos-coordinate-Zia} by multiplication by $r$.  Take the real
part after multiplying these equations, respectively, by
\[
    x^*,
    \qquad \frac{\overline{B_{ij}}}{2(\sigma_i+\sigma_j)},
    \qquad \frac{\overline{Z_{i\alpha}}}{\sigma_i},
\]
and summing over the matrix indices.  The identity
\[
    DH(\theta_\star)[x]:Y
    =x\cdot\big(\nabla H(\theta_\star):Y\big)
\]
and the symmetry of $DH(\theta_\star)[x]$ give
\begin{align*}
    \Re\sum_{i,j}G_{ij}\overline{B_{ij}}
    &=\Re\big(x^*(\nabla H(\theta_\star):B)\big),\\
    2\Re\sum_{i,\alpha}G_{i\alpha}\overline{Z_{i\alpha}}
    &=\Re\big(x^*(\nabla H(\theta_\star):(Z+Z^T))\big).
\end{align*}
Thus the $B$- and $Z$-couplings cancel those in the $x$-equation, and we
obtain
\begin{align}
    \Re(r)\bigg(
       |x|^2
       &+\frac12\sum_{i,j}
          \frac{|B_{ij}|^2}{\sigma_i+\sigma_j}
       +\sum_{i,\alpha}\frac{|Z_{i\alpha}|^2}{\sigma_i}
       \bigg) \notag\\
    &=-\eta x^*H_\star^{\mathrm{eff}}x
      -\sum_{i,\alpha}
          \frac{\delta_\alpha}{\sigma_i}|Z_{i\alpha}|^2
      \le0.
    \label{eq:mf-eos-spectral-identity}
\end{align}
The factor multiplying $\Re(r)$ is positive because $(x,B,Z)\ne0$.
The right-hand side is nonpositive because
$H_\star^{\mathrm{eff}}\ge0$ and $\delta_\alpha>0$.  Hence
$\Re(r)\le0$, which proves stability in the full-rank case.

\paragraph{Rank-deficient covariance.}
Suppose now that some $\sigma_i$ vanish, and set
\[
    I_+=\{i:\sigma_i>0\},
    \qquad I_0=\{i:\sigma_i=0\}.
\]
The coordinates excluded by the denominators in
\eqref{eq:mf-eos-spectral-identity} already have explicit nonpositive
rates.  Indeed, \eqref{eq:mf-eos-coordinate-Bij}--
\eqref{eq:mf-eos-coordinate-Zia} give
\[
    \dot B_{ij}=0\quad(i,j\in I_0),
    \qquad
    \dot Z_{i\alpha}=-\delta_\alpha Z_{i\alpha}
       \quad(i\in I_0).
\]
Together with the $C$-variables, these coordinates form an autonomous
block.  They may force the remaining variables through the $x$-equation,
but the remaining variables do not force them; hence the full system is
block triangular and their rates remain part of its spectrum.

Set this autonomous block to zero.  The remaining coordinates satisfy
$\sigma_i+\sigma_j>0$ for $B_{ij}$ and $i\in I_+$ for $Z_{i\alpha}$.
Repeating the preceding eigenmode calculation, now with the sums restricted
to these indices, gives
\begin{align*}
    \Re(r)\bigg(
       |x|^2
       &+\frac12\sum_{\sigma_i+\sigma_j>0}
          \frac{|B_{ij}|^2}{\sigma_i+\sigma_j}
       +\sum_{i\in I_+,\alpha}
          \frac{|Z_{i\alpha}|^2}{\sigma_i}
       \bigg)\\
    &=-\eta x^*H_\star^{\mathrm{eff}}x
      -\sum_{i\in I_+,\alpha}
          \frac{\delta_\alpha}{\sigma_i}|Z_{i\alpha}|^2
      \le0.
\end{align*}
Thus every rate of the remaining block also has nonpositive real part.
This proves stability in the rank-deficient case.

Finally, suppose that $H_\star u=\lambda u$ for some $u\in V_2^\perp$
with $\eta\lambda>2$.  Equation \eqref{eq:mf-eos-coordinate-C} shows that
$C=uu^T$ has rate
\[
    (1-\eta\lambda)^2-1
    =\eta\lambda(\eta\lambda-2)>0.
\]
This rate belongs to the autonomous $C$-block and hence to the spectrum of
the full block-triangular linearization.  The stationary state is therefore
unstable.
\end{proof}

\begin{proof}[Proof of \cref{prop:mf-locally-convex-stability}]
The covariance decomposition is
\cref{lem:mf-stationary-decomposition}.  The zero-covariance case is
\cref{lem:mf-zero-covariance-stability}.  The nonzero pure EoS case is
\cref{lem:mf-pure-eos-stability}.
\end{proof}

\subsection{Small-step and flat-landscape limits}

\begin{proposition}[Gradient-flow and flat-landscape limits]
\label{prop:mf-limits}
    Let $E\in C^4(\R^d)$.
    \begin{enumerate}
        \item Fix $(\theta_0,\bar\Sigma_0)$ and, for each $\eta>0$, let
        $(\theta^\eta,\Sigma^\eta)$ denote the MF solution in the rescaled
        time $\tau=\eta t$, with
        $\theta^\eta(0)=\theta_0$ and
        $\Sigma^\eta(0)=\eta\bar\Sigma_0$.  Let $\theta$ solve
        \[
            \frac{\dd\theta}{\dd\tau}=-\nabla E(\theta),
            \qquad \theta(0)=\theta_0.
        \]
        For every $T$ smaller than the maximal existence time of $\theta$,
        the rescaled MF solution exists on $[0,T]$ for all sufficiently
        small $\eta$, and there is a constant $C_T$ such that
        \[
            \sup_{0\le\tau\le T}
              |\theta^\eta(\tau)-\theta(\tau)|\le C_T\eta.
        \]

        \item Fix $\eta>0$ and $(\theta_0,\Sigma_0)$.  For each $\eps>0$,
        let $(\theta^\eps,\Sigma^\eps)$ denote the MF solution associated
        with the energy $\eps E$, in the rescaled time
        $\tau=\eta\eps t$, and with initial data
        $(\theta_0,\Sigma_0)$.  Let $(\theta,\Sigma)$ solve
        \begin{equation}
            \begin{aligned}
                \frac{\dd\theta}{\dd\tau}
                &=-\nabla E(\theta)
                  -\frac12\nabla H(\theta):\Sigma,\\
                \frac{\dd\Sigma}{\dd\tau}
                &=-(H(\theta)\Sigma+\Sigma H(\theta)),\\
                (\theta,\Sigma)|_{\tau=0}&=(\theta_0,\Sigma_0).
            \end{aligned}
            \label{eq:flat}
        \end{equation}
        For every $T$ smaller than the maximal existence time of
        $(\theta,\Sigma)$, the rescaled MF solution exists on $[0,T]$ for
        all sufficiently small $\eps$, and there is a constant $C_T$ such
        that
        \[
            \sup_{0\le\tau\le T}
              \left(|\theta^\eps(\tau)-\theta(\tau)|
                    +\|\Sigma^\eps(\tau)-\Sigma(\tau)\|\right)
            \le C_T\eps.
        \]
        Moreover, $\rank\Sigma(\tau)=\rank\Sigma_0$ throughout the maximal
        existence interval of the limiting solution.
    \end{enumerate}
\end{proposition}

\begin{proof}[Proof of \cref{prop:mf-limits}]
We treat the two limits separately.

\paragraph{Gradient-flow limit.}
Fix $T>0$, let $\theta$ solve
\[
    \frac{\dd\theta}{\dd\tau}=-\nabla E(\theta),
    \qquad \theta(0)=\theta_0,
\]
on $[0,T]$, and set
\[
    \bar\Sigma^\eta=\eta^{-1}\Sigma^\eta.
\]
Since $\Sigma_0^\eta=\eta\bar\Sigma_0$, the rescaled equations are
\begin{align}
    \frac{\dd\theta^\eta}{\dd\tau}
    &=-\nabla E(\theta^\eta)
      -\frac\eta2\nabla H(\theta^\eta):\bar\Sigma^\eta,
      \label{eq:mf-small-step-theta}\\
    \frac{\dd\bar\Sigma^\eta}{\dd\tau}
    &=-\big(\bar\Sigma^\eta H(\theta^\eta)
             +H(\theta^\eta)\bar\Sigma^\eta\big)
      +\eta H(\theta^\eta)\bar\Sigma^\eta H(\theta^\eta),
      \qquad \bar\Sigma^\eta(0)=\bar\Sigma_0.
      \label{eq:mf-small-step-sigma}
\end{align}

Only local bounds are needed.  Let
\[
    R=1+\max_{0\le\tau\le T}|\theta(\tau)|,
    \qquad K=\{y\in\R^d:|y|\le R\}.
\]
The limiting trajectory lies at distance at least one from the boundary of
the compact set $K$.  Let $L$ be a Lipschitz constant for $\nabla E$ on
$K$, and choose finite constants $M_0,M_1$ such that
\[
    \|H(y)\|_{\mathrm{op}}\le M_0,
    \qquad
    |\nabla H(y):S|\le M_1\|S\|
    \quad (y\in K).
\]
Let $T_\eta$ be the first time at which $\theta^\eta$ leaves $K$, capped at
$T$ and at the maximal existence time of the rescaled MF solution.  On
$[0,T_\eta]$, \eqref{eq:mf-small-step-sigma} and $0<\eta\le1$
give
\[
    \frac{\dd}{\dd\tau}\|\bar\Sigma^\eta\|
    \le (2M_0+M_0^2)\|\bar\Sigma^\eta\|.
\]
Hence, by Gronwall's inequality,
\begin{equation}
    \sup_{0\le\tau\le T_\eta}\|\bar\Sigma^\eta(\tau)\|
    \le \|\bar\Sigma_0\|e^{(2M_0+M_0^2)T}
    =:M_T.
    \label{eq:mf-small-step-covariance-bound}
\end{equation}
Subtracting the gradient-flow equation from
\eqref{eq:mf-small-step-theta} and using the local bounds on $K$ yields,
for $0\le\tau\le T_\eta$,
\[
    |\theta^\eta(\tau)-\theta(\tau)|
    \le L\int_0^\tau
        |\theta^\eta(s)-\theta(s)|\,\dd s
      +\frac\eta2M_1M_T\tau.
\]
A second application of Gronwall's inequality gives
\begin{equation}
    \sup_{0\le\tau\le T_\eta}
       |\theta^\eta(\tau)-\theta(\tau)|
    \le \frac\eta2M_1M_TT e^{LT}=:C_T\eta.
    \label{eq:mf-small-step-error}
\end{equation}
For $\eta$ small enough that $C_T\eta<1$, this estimate prevents
$\theta^\eta$ from reaching the boundary of $K$ before time $T$.
The local continuation theorem for ODEs then gives $T_\eta=T$, and
\eqref{eq:mf-small-step-error} proves that $\theta^\eta\to\theta$ uniformly
on $[0,T]$.  All constants used above depend only on the compact
neighborhood $K$ of the limiting trajectory.

\paragraph{Flat-landscape limit and rank preservation.}
Let $(\theta^\eps,\Sigma^\eps)$ denote the rescaled MF solution associated
with the energy $\eps E$.  Its equations are
\begin{align}
    \frac{\dd\theta^\eps}{\dd\tau}
    &=-\nabla E(\theta^\eps)
      -\frac12\nabla H(\theta^\eps):\Sigma^\eps,
      \label{eq:mf-flat-epsilon-theta}\\
    \frac{\dd\Sigma^\eps}{\dd\tau}
    &=-\big(\Sigma^\eps H(\theta^\eps)
             +H(\theta^\eps)\Sigma^\eps\big)
      +\eta\eps H(\theta^\eps)\Sigma^\eps H(\theta^\eps).
      \label{eq:mf-flat-epsilon-sigma}
\end{align}
Fix $T>0$ and let $z=(\theta,\Sigma)$ solve \eqref{eq:flat} on $[0,T]$
with the same initial data.  On the product space, use the norm
\[
    |(y,S)|_*=|y|+\|S\|,
\]
and define the compact ball
\[
    R=1+\max_{0\le\tau\le T}|z(\tau)|_*,
    \qquad
    \mathcal K=\{(y,S):|(y,S)|_*\le R\}.
\]
The limiting vector field
\[
    \mathcal W_0(y,S)
    =\left(-\nabla E(y)-\frac12\nabla H(y):S,
           -\big(SH(y)+H(y)S\big)\right)
\]
is Lipschitz on $\mathcal K$; denote a Lipschitz constant by $L_T$.
This uses only local regularity: on the projection of $\mathcal K$ onto
$\R^d$, the derivatives of $E$ involved in $\mathcal W_0$ are bounded.
In particular, local boundedness of $\nabla H$ makes $H$ locally
Lipschitz, while local boundedness of $\nabla^2H$ makes $\nabla H$ locally
Lipschitz.  On the same compact set, there is a constant $M_T^{\mathrm{flat}}$
such that
\[
    \eta\|H(y)S H(y)\|\le M_T^{\mathrm{flat}}
    \qquad ((y,S)\in\mathcal K).
\]

Let $T_\eps$ be the first time at which
$z^\eps=(\theta^\eps,\Sigma^\eps)$ leaves $\mathcal K$, capped at $T$ and
at the maximal existence time of the rescaled MF solution.
Equations \eqref{eq:mf-flat-epsilon-theta}--
\eqref{eq:mf-flat-epsilon-sigma} and \eqref{eq:flat} give, on
$[0,T_\eps]$,
\[
    |z^\eps(\tau)-z(\tau)|_*
    \le L_T\int_0^\tau|z^\eps(s)-z(s)|_*\,\dd s
       +\eps M_T^{\mathrm{flat}}\tau.
\]
Therefore
\begin{equation}
    \sup_{0\le\tau\le T_\eps}|z^\eps(\tau)-z(\tau)|_*
    \le \eps M_T^{\mathrm{flat}}T e^{L_TT}.
    \label{eq:mf-flat-error}
\end{equation}
For $\eps$ small enough, the right-hand side is smaller than one, so
$z^\eps$ cannot reach the boundary of $\mathcal K$ before time $T$.
The local continuation theorem then gives $T_\eps=T$, and
\eqref{eq:mf-flat-error} proves uniform convergence on $[0,T]$ using only
constants on a compact neighborhood of the limiting trajectory.

It remains to prove rank preservation for the limiting covariance.  If
$P$ and $Q$ solve
\begin{align*}
    \frac{\dd P}{\dd\tau}
      &=-H(\theta(\tau))P,
      &P(0)&=\Id,\\
    \frac{\dd Q}{\dd\tau}
      &=QH(\theta(\tau)),
      &Q(0)&=\Id,
\end{align*}
then
\[
    \frac{\dd}{\dd\tau}(QP)
    =QH(\theta)P-QH(\theta)P=0.
\]
Thus $Q(\tau)P(\tau)=\Id$, so $P(\tau)$ is invertible.  Define
\[
    \widetilde\Sigma(\tau)=P(\tau)\Sigma_0P(\tau)^T.
\]
Since $H(\theta)$ is symmetric, differentiation gives
\begin{align*}
    \frac{\dd\widetilde\Sigma}{\dd\tau}
    &=-H(\theta)P\Sigma_0P^T
      -P\Sigma_0P^TH(\theta)\\
    &=-H(\theta)\widetilde\Sigma
      -\widetilde\Sigma H(\theta),
    \qquad \widetilde\Sigma(0)=\Sigma_0.
\end{align*}
Hence $\widetilde\Sigma$ and $\Sigma$ solve the same linear matrix ODE with
the same initial condition, and uniqueness yields
$\Sigma(\tau)=P(\tau)\Sigma_0P(\tau)^T$.  Multiplication on the left and
right by invertible matrices preserves rank, so
$\rank\Sigma(\tau)=\rank\Sigma_0$ throughout the maximal existence
interval.

\end{proof}

\subsection{Prescribed-Hessian raw-to-coarse approximation}
\label{app:mf-raw-to-coarse}

In this section, we provide a minimal argument for the validity of the coarse covariance equation, assuming that the Hessian is prescribed and varies at a bounded rate. 
This compares the two covariance equations, not either equation with GD. 

\begin{proposition}[Prescribed-Hessian validity of the coarse covariance equation]
\label{prop:mf-validity}
    Fix $T,L>0$, $\eta>0$, and $0<\varepsilon\le1/4$.  Let
    $V\subset\R^d$ be a linear subspace and let
    $H\in C^1([0,T/\varepsilon];\mathbb S_+^d)$, where $\mathbb S_+^d$ denotes the
    symmetric positive semidefinite matrices.  Assume that $H(t)V\subset V$
    for every $t$, that the active Hessian varies at bounded speed,
    \[
        \sup_{0\le t\le T/\varepsilon}
        \|\left.\dot H(t)\right|_V\|_{\mathrm{op}}\le L,
    \]
    and that
    \begin{equation}
        \sup_{0\le t\le T/\varepsilon}
          \|\left.(\eta H(t)-2\Id)\right|_V\|_{\mathrm{op}}\le\varepsilon.
        \label{eq:mf-prescribed-hessian-tube}
    \end{equation}
    Consider the raw covariance recursion
    \begin{equation}
    \label{eq:prescribed-rawMF}
    \Sigma_{k+1}^{\rm raw}
       =(\Id-\eta H(k))\Sigma_k^{\rm raw}(\Id-\eta H(k)),
    \end{equation}
    and its coarse counterpart
    \begin{equation}
    \label{eq:prescribed-coarseMF}
    \frac{\dd\Sigma}{\dd t}
       =-\eta(\Sigma H(t)+H(t)\Sigma)+\eta^2H(t)\Sigma H(t),
    \end{equation}
    from the same $\Sigma_0\ge0$, with $\operatorname{Ran}\Sigma_0\subset V$.
    Set
    $\Sigma^{\rm raw}(t)=\Sigma_k^{\rm raw}$ for $k\le t<k+1$.  There is a
    constant $C_{T,L}$, independent of $\eta$, $\varepsilon$, $H$, and $\Sigma_0$,
    such that
    \begin{equation}
        \sup_{0\le t\le T/\varepsilon}
        \|\Sigma^{\rm raw}(t)-\Sigma(t)\|
        \le C_{T,L}\left(\varepsilon+\frac{\eta}{\varepsilon}\right)\|\Sigma_0\|.
        \label{eq:mf-raw-coarse-error}
    \end{equation}
    In particular, if $\eta\le\varepsilon^2$, the error is at most
    $C_{T,L}\varepsilon\|\Sigma_0\|$ on $O(\varepsilon^{-1})$ algorithmic times.
\end{proposition}

\begin{proof}[Proof of \cref{prop:mf-validity}]
Let $P$ be the orthogonal projection onto $V$.  Since $H(t)$ is symmetric and
$H(t)V\subset V$, the orthogonal complement $V^\perp$ is also invariant under
$H(t)$: for $v\in V$ and $w\in V^\perp$,
\[
    v\cdot H(t)w=H(t)v\cdot w=0.
\]
Consequently, $PH(t)=H(t)P$.  Moreover,
$\operatorname{Ran}\Sigma_0\subset V$ and the symmetry of $\Sigma_0$ give
$\Sigma_0=P\Sigma_0P$.  The raw update then shows by induction that
\[
    \Sigma_k^{\rm raw}=P\Sigma_k^{\rm raw}P
    \quad\text{for every }k.
\]
The coarse vector field has the same property: if $S=PSP$, then every term in
\eqref{eq:prescribed-coarseMF} is supported in $V$.  Uniqueness of the linear
matrix ODE therefore gives $\Sigma(t)=P\Sigma(t)P$.  It is thus enough to work
on $V$; the behavior of $H(t)$ on $V^\perp$ never enters the two dynamics.

On $V$, set $D(t)=\left.(\eta H(t)-2\Id)\right|_V$.  For a matrix
$S:V\to V$, define
\[
    \mathcal L_t(S)=D(t)S+SD(t)+D(t)SD(t).
\]
Expanding the raw map and the coarse vector field gives, respectively,
\begin{align}
    (\Id-\eta H(k))S(\Id-\eta H(k))
      &=S+\mathcal L_k(S),
      \label{eq:mf-prescribed-raw-euler}\\
    -\eta(SH(t)+H(t)S)+\eta^2H(t)SH(t)
      &=\mathcal L_t(S).
      \label{eq:mf-prescribed-coarse-field}
\end{align}
Thus the raw recursion is the explicit Euler discretization, with unit step,
of the non-autonomous coarse equation.

The tube assumption \eqref{eq:mf-prescribed-hessian-tube} implies
\begin{align}
    \|\mathcal L_t(S)\|&\le3\varepsilon\|S\|,
      \label{eq:mf-prescribed-field-bound}\\
    \|(\mathcal L_t-\mathcal L_s)(S)\|
      &\le3L\eta|t-s|\|S\|.
      \label{eq:mf-prescribed-time-bound}
\end{align}
Indeed, the first estimate follows by bounding the three terms in
$\mathcal L_t$.  For the second, expand
\[
    D(t)SD(t)-D(s)SD(s)
    =(D(t)-D(s))SD(t)+D(s)S(D(t)-D(s))
\]
and use
$\|D(t)-D(s)\|_{\mathrm{op}}\le L\eta|t-s|$.

Before comparing the two dynamics, we bound their size.  From
\eqref{eq:mf-prescribed-field-bound} and Gronwall's inequality,
\[
    \|\Sigma(t)\|\le e^{3\varepsilon t}\|\Sigma_0\|.
\]
Likewise, \eqref{eq:mf-prescribed-raw-euler} gives
\[
    \|\Sigma_{k+1}^{\rm raw}\|
       \le(1+3\varepsilon)\|\Sigma_k^{\rm raw}\|.
\]
Both covariances are therefore bounded by $C_T\|\Sigma_0\|$ for
$t,k\le T/\varepsilon$.  In particular, if $k\le t<k+1$, then
\begin{equation}
    \|\Sigma(t)-\Sigma(k)\|
       \le C_T\varepsilon(t-k)\|\Sigma_0\|.
    \label{eq:mf-prescribed-within-step}
\end{equation}

The exact coarse evolution over one step is
\[
    \Sigma(k+1)=\Sigma(k)+\mathcal L_k(\Sigma(k))+r_k,
\]
where
\[
    r_k=\int_k^{k+1}
       \big(\mathcal L_t(\Sigma(t))-\mathcal L_k(\Sigma(k))\big)\,\dd t.
\]
Using \eqref{eq:mf-prescribed-field-bound} for the change of $\Sigma$,
\eqref{eq:mf-prescribed-time-bound} for the change of $H$, and
\eqref{eq:mf-prescribed-within-step}, we obtain
\begin{equation}
    \|r_k\|\le C_{T,L}(\varepsilon^2+\eta)\|\Sigma_0\|.
    \label{eq:mf-prescribed-local-defect}
\end{equation}
Subtracting this identity from the raw update yields
\[
    \|\Sigma_{k+1}^{\rm raw}-\Sigma(k+1)\|
    \le(1+3\varepsilon)\|\Sigma_k^{\rm raw}-\Sigma(k)\|
       +C_{T,L}(\varepsilon^2+\eta)\|\Sigma_0\|.
\]
The initial error vanishes.  Discrete Gronwall, for $k\le T/\varepsilon$, gives
\[
    \|\Sigma_k^{\rm raw}-\Sigma(k)\|
       \le C_{T,L}\left(\varepsilon+\frac{\eta}{\varepsilon}\right)\|\Sigma_0\|.
\]
Together with \eqref{eq:mf-prescribed-within-step}, this proves
\eqref{eq:mf-raw-coarse-error} for the piecewise-constant interpolation of the
raw covariance.  Since the embedding of matrices supported in $V$ into
$\R^{d\times d}$ preserves the Frobenius norm, the same estimate holds in the
ambient space.
\end{proof}

\subsection{Raw-to-coarse Rod reduction and comparison with MF}
\label{app:mf-rod-comparaison}

We prove the two parts of \cref{prop:equivalence} in turn.
Throughout this subsection, $\|\cdot\|$ denotes the Frobenius norm
and $\|\cdot\|_{\rm op}$ the operator norm.

\subsubsection{From raw Rod flow to coarse Rod flow}

\paragraph{Detailed statement of part~(i).}
Fix $T>0$, $C\ge1$, and $\eta,a\in(0,1)$ with
$\delta=a^2/\eta\in(0,1)$. Let $(\theta_t,\Sigma_t)$ solve
\eqref{eq:rawRod} on $[0,T/\eta]$, and use the same principal extent
$\xi_t$ when evaluating the coarse field $\mathcal V_{\rm R}$.
It is enough that $E_\eta$ be $C^5$ on a neighborhood of the segments
$\{\theta_t+s\xi_t:|s|\le1\}$ and that, uniformly in $t$,
\begin{equation}\label{eq:rod-balanced-scaling}
    |\nabla E_\eta(\theta_t)|\le C\delta,\qquad
    C^{-1}a^2\le|\xi_t|^2\le Ca^2,\qquad
    \|D^jE_\eta\|_{\rm op}\le\frac C\eta\quad(2\le j\le5).
\end{equation}
Here $C$ is independent of $\eta,a,t$, and each derivative norm is
the supremum of the multilinear operator norm on the corresponding
segment. The two conclusions of \cref{prop:equivalence}(i) are
\begin{align}
    \int_0^{T/\eta}
      \left|\frac{d\theta_t}{dt}
        -\mathcal V_{\rm R}^\theta(\theta_t,\Sigma_t)\right|\,dt
      &\le C_T\delta,
        \label{eq:rod-integrated-mean-remainder}\\
    \frac1{a^2}\int_0^{T/\eta}
      \left\|\frac{d\Sigma_t}{dt}
        -\mathcal V_{\rm R}^\Sigma(\theta_t,\Sigma_t)\right\|_F\,dt
      &\le C_T\delta,
        \label{eq:rod-integrated-covariance-remainder}
\end{align}
where $C_T$ depends only on $C,T$.
The first estimate is an absolute mean defect; the second divides
the covariance defect by its natural scale $a^2$.

\begin{proof}[Proof of \cref{prop:equivalence}, part~(i)]
In the Taylor expansion of \eqref{eq:rawRod} in $\xi$, the odd
powers cancel between the two endpoints, and we obtain
\begin{align}
    \frac{d\theta}{dt}
      &=-\eta\nabla E(\theta)
        -\frac\eta2\nabla H(\theta):\xi\xi^T+R^\theta,
        \label{eq:rod-taylor-mean}\\
    \frac{d\Sigma}{dt}
      &=\frac{\eta^2}{2}\left(
          \nabla E(\theta)\nabla E(\theta)^T
          +H(\theta)\xi\xi^TH(\theta)\right)-2\Sigma+R^\Sigma.
        \label{eq:rod-taylor-covariance}
\end{align}
We include all terms proportional to $\eta^2$ in the mean equation in
$R^\theta$. In the covariance equation, $R^\Sigma$ includes the
products of $\nabla E(\theta)$ with the quadratic Taylor correction
$\frac12\nabla H(\theta):\xi\xi^T$. Explicitly,
\begin{align}
    |R^\theta|
      &\le\frac\eta{24}\|D^5E\|_{\rm op}|\xi|^4
        +\frac{\eta^2}{2}\|D^2E\|_{\rm op}
          \left(|\nabla E(\theta)|
                 +\frac12\|D^3E\|_{\rm op}|\xi|^2\right),
        \label{eq:rod-taylor-mean-remainder}\\
    \|R^\Sigma\|_F
      &\le\frac{\eta^2}{2}\|D^3E\|_{\rm op}
               |\nabla E(\theta)|\,|\xi|^2\nonumber\\
      &\quad+\eta^2\left(
          \frac18\|D^3E\|_{\rm op}^2
          +\frac16\|D^2E\|_{\rm op}\|D^4E\|_{\rm op}
        \right)|\xi|^4.
        \label{eq:rod-taylor-covariance-remainder}
\end{align}
For the first bound, the symmetric gradient average has Taylor
remainder at most $\|D^5E\|_{\rm op}|\xi|^4/24$, since its linear and
cubic terms cancel. The averaged Hessian is bounded by
$\|D^2E\|_{\rm op}$, and the averaged gradient by
$|\nabla E(\theta)|+\|D^3E\|_{\rm op}|\xi|^2/2$.
For the second bound, decompose the endpoint gradients into their even
and odd parts. Their outer-product average is the sum of the outer
products of these two parts. The even part differs from
$\nabla E(\theta)$ by at most $\|D^3E\|_{\rm op}|\xi|^2/2$.
The odd part differs from $H(\theta)\xi$ by at most
$\|D^4E\|_{\rm op}|\xi|^3/6$, and both have norm at most
$\|D^2E\|_{\rm op}|\xi|$. Expanding these outer products gives the
stated covariance bound.

By \eqref{eq:rod-taylor-covariance}, the integrand in
\eqref{eq:rod-integrated-covariance-remainder} is the norm of
$\frac{\eta^2}{2}\nabla E(\theta_t)\nabla E(\theta_t)^T+R_t^\Sigma$.
The gradient contribution is bounded by
$\frac{\eta^2}{2}|\nabla E(\theta_t)|^2\lesssim\eta^2\delta^2=a^4$.
Substituting \eqref{eq:rod-balanced-scaling} into
\eqref{eq:rod-taylor-covariance-remainder} gives
\[
    \|R_t^\Sigma\|_F
       \lesssim\eta|\nabla E(\theta_t)|\,|\xi_t|^2+|\xi_t|^4
       \lesssim\eta\delta a^2+a^4=2a^4,
\]
with constants depending only on $C$.
Integration over an interval of length $T/\eta$ gives a bound
$C_Ta^4/\eta=C_Ta^2\delta$; dividing by the covariance scale $a^2$
proves \eqref{eq:rod-integrated-covariance-remainder}.
Similarly, \eqref{eq:rod-taylor-mean-remainder} yields
\[
    |R_t^\theta|
       \lesssim |\xi_t|^4+\eta|\nabla E(\theta_t)|+|\xi_t|^2
       \lesssim a^4+\eta\delta+a^2\lesssim a^2.
\]
Since $\frac{d\theta_t}{dt}
-\mathcal V_{\rm R}^\theta(\theta_t,\Sigma_t)=R_t^\theta$,
integration over $[0,T/\eta]$ gives $C_Ta^2/\eta=C_T\delta$ and proves
\eqref{eq:rod-integrated-mean-remainder}.
\end{proof}

Part~(i) of \cref{prop:equivalence} controls the accumulated defect
along the raw trajectory. Controlling the distance between independently
evolved raw and coarse solutions would require a stability
estimate on $[0,T/\eta]$. In particular, the absolute mean bound is
not smaller than the available $O(\delta)$ bound on the mean displacement.
The effect of retaining only $\xi\xi^T$, rather than $\Sigma$, is
quantified in the comparison with MF below.

\subsubsection{{Comparison with MF}}

\begin{proof}[Proof of \cref{prop:equivalence}, part~(ii)]
    Write $D:=\eta H(\theta)-2\Id$. The covariance fields are
    \begin{align}
        \mathcal V_{\rm MF}^\Sigma
        &=-\eta(\Sigma H+H\Sigma)+\eta^2H\Sigma H,\nonumber\\
        \mathcal V_{\rm R}^\Sigma
        &=\frac12\eta^2 H{\xi\xi^T}H-2\Sigma.
        \label{eq:mf-rod-fields}
    \end{align}
    Subtracting the two identities, we obtain
    \begin{align}
        \mathcal V_{\rm MF}^\Sigma - \mathcal{V}_{\rm R}^\Sigma
        &= - \eta (\Sigma H + H \Sigma) + 2\Sigma + \frac12 \eta^2 H\Sigma H
           \nonumber\\
        &\quad{+\frac{\eta^2}{2}H(\Sigma-\xi\xi^T)H}\nonumber\\
        &= \frac12 D\Sigma D
           {+\frac{\eta^2}{2}H(\Sigma-\xi\xi^T)H}.
    \end{align}
    For $\Sigma\ge0$, the factorization
    \[
        D\Sigma D=(D\Sigma^{1/2})(D\Sigma^{1/2})^T
    \]
    and the inequality $\|AA^T\|\le\|A\|^2$ also give
    \begin{equation}\label{eq:mf-rod-active-defect}
        \|\mathcal V_{\rm MF}^\Sigma-\mathcal V_{\rm R}^\Sigma\|
        \le\frac12\|D\Sigma^{1/2}\|^2
           {+\frac{\eta^2}{2}\|H(\Sigma-\xi\xi^T)H\|}.
    \end{equation}

    Let $P=P_\Sigma$ be the orthogonal projection onto
    $\operatorname{Ran}\Sigma$, and assume $\|DP\|_{\rm op}\le\eps$.
    Since $\Sigma=P\Sigma P$ {and $\xi\in\operatorname{Ran}\Sigma$},
    \[
        D\Sigma=(DP)\Sigma,
        \qquad
        \Sigma D=\Sigma(PD),
        \qquad
        D\Sigma D=(DP)\Sigma(PD).
    \]
    {The matrix $\Sigma-\xi\xi^T$ is also supported in
    $\operatorname{Ran}\Sigma$, so
    \[
        H(\Sigma-\xi\xi^T)H
        =(HP)(\Sigma-\xi\xi^T)(PH),\qquad
        \|\eta HP\|_{\rm op}=\|(D+2\Id)P\|_{\rm op}\le2+\eps.
    \]}
    Symmetry of $D$ gives $\|PD\|_{\rm op}=\|DP\|_{\rm op}$.
    Applying $\|A\Sigma C\|\le
    \|A\|_{\rm op}\|\Sigma\|\|C\|_{\rm op}$ to these products therefore yields
    \[
        \|\mathcal V_{\rm MF}^\Sigma-\mathcal V_{\rm R}^\Sigma\|
        \le\frac12\eps^2\|\Sigma\|
          {+\frac{(2+\eps)^2}{2}\|\Sigma-\xi\xi^T\|}.
    \]
    {For the individual fields, expanding $\eta H=2\Id+D$
    gives
    \begin{align*}
        \mathcal V_{\rm MF}^\Sigma
          &=D\Sigma+\Sigma D+D\Sigma D,\\
        \mathcal V_{\rm R}^\Sigma
          &=D\xi\xi^T+\xi\xi^TD+\tfrac12D\xi\xi^TD
             -2(\Sigma-\xi\xi^T).
    \end{align*}
    Since $\|\xi\xi^T\|\le\|\Sigma\|$, we obtain}
    \begin{align}
        \|\mathcal V_{\rm MF}^\Sigma\|
        &\le(2\eps+\eps^2)\|\Sigma\|,\nonumber\\
        \|\mathcal V_{\rm R}^\Sigma\|
        &\le(2\eps+\tfrac12\eps^2)\|\Sigma\|
          {+2\|\Sigma-\xi\xi^T\|}.
        \label{eq:mf-rod-drift}
    \end{align}
\end{proof}

\section{Analysis of the kinetic equation}
\label[appendix]{app:kinetic-analysis}

\subsection{Velocity moments and variational structure}
\label{app:kinetic-variational}

We work on $\mathbb T^d\times\R^{d\times r}$ and use the Fourier convention
$\widehat{\nabla h}(k)=ik\widehat h(k)$.  Throughout this subsection,
$\g\in C^2(\mathbb T^d)$ is even and $\widehat\g(k)\ge0$.  For a probability
measure $f$ with positional marginal $\rho$, set
\[
    M_p(f)=\int |v|^p\,\dd f,
    \qquad
    \mathcal E(\rho)=\frac12\int\g*(\rho-\mu)\,\dd(\rho-\mu).
\]

The free energy is unbounded below when velocities are unrestricted, but a bound on
their second moment restores a uniform lower bound.

\begin{proposition}[Free energy with and without a velocity bound]
\label{prop:kinetic-free-energy-unbounded}
    Suppose that $\g$ is nonconstant.  Then
    \begin{equation}
        \inf_{f\in\mathcal P_2(\mathbb T^d\times\R^{d\times r})}
        \mathcal F(f)=-\infty.
        \label{eq:kinetic-unrestricted-infimum}
    \end{equation}
    Moreover, every $f\in\mathcal P_2$ satisfies
    \begin{equation}
        \mathcal F(f)
        \ge \mathcal E(\rho)
        -\|\nabla^2\g\|_{L^\infty}M_2(f).
        \label{eq:kinetic-velocity-tube-bound}
    \end{equation}
    Consequently, for every $\varepsilon>0$,
    \begin{equation}
        -\|\nabla^2\g\|_{L^\infty}\varepsilon^2
        \le
        \inf_{\substack{f\in\mathcal P_2\\M_2(f)\le\varepsilon^2}}
        \mathcal F(f)
        \le0.
        \label{eq:kinetic-velocity-tube-infimum}
    \end{equation}
\end{proposition}

\begin{proof}
    We first prove unboundedness.  If $\g*\mu$ is nonconstant,
    there are $\theta\in\mathbb T^d$ and $w\in\R^{d\times r}$ such that
    $\nabla^2(\g*\mu)(\theta):ww^T>0$.  Indeed, otherwise
    $\nabla^2(\g*\mu)\le0$ everywhere, so the periodic function $\g*\mu$
    would be concave and hence constant.  Direct substitution of
    $f_R=\delta_{(\theta,Rw)}$ in \eqref{eq:kinetic-free-energy} gives
    \[
        \mathcal F(f_R)
        =\mathcal E(\delta_\theta)
        -\frac{R^2}{2}\nabla^2(\g*\mu)(\theta):ww^T
        \longrightarrow-\infty.
    \]
    It remains to consider the case where $\g*\mu$ is constant.  Fix
    $\theta$ and take
    \[
        f_R=\frac12\delta_{(\theta,Rw)}
             +\frac12\delta_{(\theta,-Rw)}.
    \]
    Direct substitution now gives
    \[
        \mathcal F(f_R)
        =\mathcal E(\delta_\theta)
         +\frac{R^2}{2}\nabla^2\g(0):ww^T.
    \]
    Because
    $\nabla^2\g(0)=-\sum_{k\ne0}\widehat\g(k)k\otimes k$ and $\g$ is
    nonconstant, $w$ can be chosen so that the last term is negative.
    This proves \eqref{eq:kinetic-unrestricted-infimum} in both cases.

    We next prove the lower bound.  Expanding $(v-v')(v-v')^T$ in
    \eqref{eq:kinetic-free-energy} yields
    \[
        \mathcal F(f)-\mathcal E(\rho)
        =\frac12\int
          [\nabla^2\g*(\rho-\mu)](\theta):vv^T\,\dd f
        -\frac12\iint
          v:\nabla^2\g(\theta-\theta')v'\,\dd f\,\dd f'.
    \]
    Since $\widehat{\nabla^2\g}(k)=-k\otimes k\,\widehat\g(k)$, Fourier
    expansion applied columnwise gives
    \[
        -\iint v:\nabla^2\g(\theta-\theta')v'\,\dd f\,\dd f'
        =\sum_{k\ne0}\widehat\g(k)
          \left|k^T\int e^{-ik\cdot\theta}v\,\dd f(\theta,v)\right|^2
        \ge0.
    \]
    Moreover,
    \[
        \|\nabla^2\g*(\rho-\mu)\|_{L^\infty}
        \le \|\nabla^2\g\|_{L^\infty}\|\rho-\mu\|_{\rm TV}
        \le 2\|\nabla^2\g\|_{L^\infty}.
    \]
    Hence
    \[
        \frac12\int
        [\nabla^2\g*(\rho-\mu)](\theta):vv^T\,\dd f
        \ge-\|\nabla^2\g\|_{L^\infty}M_2(f),
    \]
    which proves \eqref{eq:kinetic-velocity-tube-bound}.  If
    $M_2(f)\le\varepsilon^2$, this gives the lower bound in
    \eqref{eq:kinetic-velocity-tube-infimum}; the upper bound follows from taking $f=\mu\otimes\delta_0$, whose free energy vanishes.
\end{proof}

A natural question is whether the free energy attains its minimum under velocity constraints, which we now address.

\begin{proposition}[Existence under uniform velocity integrability]
\label{prop:kinetic-restricted-minimizers}
    Fix $p>2$ and $L>0$, and let
    \[
        \mathcal B_L^{p}
        :=\left\{
            f\in\mathcal P_p(\mathbb T^d\times\R^{d\times r}):
            M_p(f)\le L
        \right\}.
    \]
    Then $\mathcal F$ attains its minimum on
    $\mathcal B_L^{p}$.
\end{proposition}

\begin{proof}
    For every $f\in\mathcal B_L^p$, H\"older's inequality gives
    $M_2(f)\le M_p(f)^{2/p}\le L^{2/p}$, so
    \eqref{eq:kinetic-velocity-tube-bound} bounds $\mathcal F$ from below on
    $\mathcal B_L^p$.
    Let $(f_j)$ be a minimizing sequence. Since the $\theta$-space is compact and
    $M_p(f_j)\le L$, Markov's inequality makes $(f_j)$ tight.  Prokhorov's
    theorem therefore gives, up to a subsequence, $f_j\rightharpoonup f$ narrowly as $j\to\infty$.
    The moment constraint passes to the limit by lower semicontinuity.

    It remains to pass to the limit in the energy, which is quadratic in
    velocity.  For $z=(\theta,v)$ and $z'=(\theta',v')$, set
    \[
        K(z,z')
        =\g(\theta-\theta')
        +\frac12\nabla^2\g(\theta-\theta'):
          (v-v')(v-v')^T.
    \]
    With the target understood as $\mu\otimes\delta_0$,
    \eqref{eq:kinetic-free-energy} reads
    \[
        \mathcal F(f_j)
        =\frac12\iint K(z,z')\,
          \dd(f_j-\mu)(z)\,\dd(f_j-\mu)(z').
    \]

    Let $\chi\in C_c(\R^{d\times r};[0,1])$ satisfy
    $\chi=1$ on $B_1$ and $\chi=0$ outside $B_2$, and set
    \[
        \chi_R(v)=\chi(v/R),
        \qquad
        K_R(z,z')=\chi_R(v)\chi_R(v')K(z,z').
    \]
    The kernel $K_R$ is bounded and continuous.  Moreover,
    $f_j\rightharpoonup f$ implies convergence of
    $f_j\otimes f_j$, $f_j\otimes\mu$, and $\mu\otimes f_j$ to the
    corresponding products with $f$.  Consequently,
    \begin{equation}
        \frac12\iint K_R\,
          \dd(f_j-\mu)\,\dd(f_j-\mu)
        \longrightarrow
        \frac12\iint K_R\,
          \dd(f-\mu)\,\dd(f-\mu)
        \qquad (j\to\infty).
        \label{eq:kinetic-truncated-energy-limit}
    \end{equation}

    It remains to remove the cutoff uniformly in $j$.  For $R\ge1$,
    \[
        \int_{\{|v|>R\}}\dd f_j
        \le L R^{-p},
        \qquad
        \int_{\{|v|>R\}}|v|^2\,\dd f_j
        \le L R^{2-p},
        \qquad
        M_2(f_j)\le L^{2/p}.
    \]
    The middle estimate is precisely the uniform integrability of $|v|^2$.
    Since
    \[
        |K(z,z')|\le C(1+|v|^2+|v'|^2),
        \qquad
        |f_j-\mu|\le f_j+\mu,
    \]
    and $1-\chi_R(v)\chi_R(v')$ vanishes unless
    $|v|>R$ or $|v'|>R$, it is enough to estimate either of these two
    symmetric regions.  Since the lifted target has velocity zero,
    \begin{align*}
        &\iint_{\{|v|>R\}}
          (1+|v|^2+|v'|^2)\,
          \dd(f_j+\mu)(z)\,\dd(f_j+\mu)(z')\\
        &\quad=2\int_{\{|v|>R\}}(1+|v|^2)\,\dd f_j
          +M_2(f_j)\int_{\{|v|>R\}}\dd f_j\\
        &\quad\le 2L\bigl(R^{-p}+R^{2-p}\bigr)
          +L^{1+2/p}R^{-p}.
    \end{align*}
    The same estimate holds with $v$ and $v'$ exchanged.  We conclude that
    \begin{align}
        &\left|\mathcal F(f_j)
        -\frac12\iint K_R\,
          \dd(f_j-\mu)\,\dd(f_j-\mu)\right| \\
        &\hspace{8em}\le
          C L\bigl(1+L^{2/p}\bigr)R^{2-p}.
        \label{eq:kinetic-energy-cutoff-tail}
    \end{align}
    Here $C$ depends only on
    $\|\g\|_{L^\infty}+\|\nabla^2\g\|_{L^\infty}$ and not on $j$ or $R$.
    Since $p>2$, the right-hand side tends to zero as $R\to\infty$.  The
    limit measure also satisfies $M_p(f)\le L$, so the same estimate holds
    with $f$ in place of $f_j$.  We may therefore first let $j\to\infty$ in
    \eqref{eq:kinetic-truncated-energy-limit} and then let $R\to\infty$ in
    \eqref{eq:kinetic-energy-cutoff-tail}.  This proves
    $\mathcal F(f_j)\to\mathcal F(f)$, and $f$ is a minimizer.
\end{proof}

Finally, we prove the convergence of minimizers to the zero-velocity limit. We actually prove a slightly stronger statement, where we need not consider only minimizers, but any sequence of measures with vanishing second moment and nonpositive free energy.

\begin{proposition}[Small-variance limit]
\label{prop:kinetic-small-variance-limit}
    Assume that $\widehat\g(k)>0$ for every $k\ne0$.  Let
    $\varepsilon_n\to0$ and let $f_n\in\mathcal P_2$ satisfy
    \[
        M_2(f_n)\le\varepsilon_n^2,
        \qquad
        \mathcal F(f_n)\le0.
    \]
    Then
    \begin{equation}
        W_2(f_n,\mu\otimes\delta_0)\longrightarrow0.
        \label{eq:kinetic-minimizer-convergence}
    \end{equation}
    In particular, $f_n$ converges narrowly to $\mu\otimes\delta_0$.

    Moreover, fix any $p>2$ and arbitrary $L_n>0$.  The conclusion applies
    when $f_n$ minimizes $\mathcal F$ on
    \[
        \{f\in\mathcal B_{L_n}^{p}:M_2(f)\le\varepsilon_n^2\}.
    \]
    Such minimizers exist by \cref{prop:kinetic-restricted-minimizers}, and
    $\mu\otimes\delta_0$ is an admissible competitor.
\end{proposition}

\begin{proof}
    Write $\rho_n$ for the positional marginal of $f_n$.  By
    \eqref{eq:kinetic-velocity-tube-bound},
    \[
        0\ge\mathcal F(f_n)
        \ge\mathcal E(\rho_n)
        -\|\nabla^2\g\|_{L^\infty}\varepsilon_n^2.
    \]
    Therefore
    \[
        \mathcal E(\rho_n)
        \le\|\nabla^2\g\|_{L^\infty}\varepsilon_n^2
        \longrightarrow0.
    \]
    Since every nonzero Fourier coefficient of $\g$ is positive, this implies
    $\widehat\rho_n(k)\to\widehat\mu(k)$ for every $k$.  Trigonometric
    polynomials are dense in $C(\mathbb T^d)$, hence
    $\rho_n\rightharpoonup\mu$.  Compactness of the torus upgrades this to
    $W_2(\rho_n,\mu)\to0$.

    Finally, coupling each $(\theta,v)$ with $(\theta,0)$ gives
    \[
        W_2(f_n,\rho_n\otimes\delta_0)^2
        \le M_2(f_n)\le\varepsilon_n^2.
    \]
    The triangle inequality proves \eqref{eq:kinetic-minimizer-convergence}.
\end{proof}

\subsection{Proof of the well-posedness theorem}

\paragraph{Notation.}
We identify the
target with $\mu\otimes\delta_0$ and use the characteristic flow
$Z=(\Theta,V)$ defined in \eqref{eq:MMD-characteristics}; the solution is
represented by \eqref{eq:MMDFlow}.  We write
$\V=(\V_\theta,\V_v)$ for the characteristic velocity and
$M_k(t):=M_k(f_t)$ for its velocity moments.
The products $vv^T$ and $VV^T$ are $d\times d$ matrices, obtained by
summing the outer products of the columns. Derivatives in
$v\in\R^{d\times r}$ are Fr\'echet derivatives, with operator norms
induced by the Frobenius norm.

Formally, the equation \eqref{eq:MMD} is equivalent to
\begin{equation}\label{eq:MMDFlow}
    f(t) = Z(t,\cdot ;\, f)_\#f_0,
\end{equation}
where $Z(t,\theta, v) = (\Theta(t,\theta, v), V(t,\theta, v))$ is the flow map
\begin{equation}\label{eq:MMD-characteristics}
\begin{cases}
    \displaystyle \frac{d\Theta}{ds} =\V_\theta(s,\Theta, V):= - \nab\g\ast (\rho- \mu)(s,\Theta) - \frac{1}{2} \nab_\theta (\nab^2\g : VV^T) \ast (f-\mu) (s,\Theta, V) , \\[10pt]
    \displaystyle \frac{d V}{ds} = \V_v(s,\Theta, V):= - (\nab^2\g \, V) \ast (f-\mu)(s,\Theta ,V).
\end{cases}
\end{equation}

\begin{theorem}[Well-posedness and regularity of the MMD kinetic equation]\label{thm:well-posedness}
    Let $\g\in C^{3,1}_b$ {be even}, $f_0 \in \mathcal{P}_2(\R^{d}\times \R^{d\times r})$ compactly supported in the $v$ variable and $\mu \in \mathcal{P}(\R^d)$.

    For all $T>0$, there exists a unique solution $f \in C([0, T], \mathcal{P}_2(\R^{d}\times \R^{d\times r}))$ to \eqref{eq:MMDFlow} with compact support in $v$. This solution furthermore satisfies
    \begin{equation}\label{eq:kinetic-dissipation}
       \frac{d}{dt}\mathcal{F}(f) + \int_{\R^d\times \R^{d\times r}} \bigg( \bigg\vert \nab_\theta \frac{\d\mathcal{F}}{\d f}\bigg\vert^2 + \bigg\vert \nab_v \frac{\d\mathcal{F}}{\d f}\bigg\vert^2 \bigg) \, df \le 0.
    \end{equation}

    Finally, if $\g$ is smooth enough, the regularity of $f_0$ is propagated in time.
\end{theorem}

\subsubsection{Study of the flow map}
In this section, we study the flow \eqref{eq:MMD-characteristics} and establish its regularity properties.
Specifically, the goal is to prove the following two lemmas.

\begin{lemma}[Regularity of the velocity fields]\label{lem:Lip}
    Let $n\ge0$, $f\in\mathcal P_2(\R^d\times\R^{d\times r})$,
    $\mu\in\mathcal P_2(\R^d)$, and $\g\in C_b^{n+2,1}(\R^d)$.  For
    every $0\le a\le n$ and $b\ge0$, there is a constant $C_{a,b}$ such that
    \[
        \sup_{\substack{|v|\le R\\\theta\in\R^d}}
        \left(
            \|\nab_\theta^a\nab_v^b\mathcal V_\theta\|
            +\|\nab_\theta^{a+1}\nab_v^b\mathcal V_v\|
        \right)
        \le C_{a,b}\bigl(1+R^2+M_2(f)\bigr).
    \]
    The constant depends only on the indicated derivatives of $\g$.
\end{lemma}

\begin{lemma}[Stability of the velocity fields]\label{lem:measureStability}
    Let $R>0$ and $|v|\le R$. Assume $\g \in C^{3,1}_b(\R^d)$. For any $f_1, f_2 \in \mathcal{P}_4(\R^{d}\times \R^{d\times r})$, we have
    \begin{align*}
        &| \mathcal{V}_\theta (t,\theta, v; f_1)
             - \mathcal{V}_\theta (t,\theta, v; f_2) | \\
        &\quad\le C\Bigl[1+R+R^2
          +(1+R)\bigl(M_2(f_1)^\hal+M_2(f_2)^\hal\bigr)
          +M_4(f_1)^\frac14M_4(f_2)^\frac14\Bigr]W_2(f_1,f_2),
    \end{align*}
    and
    \[
        | \mathcal{V}_v (t,\theta, v; f_1)
           - \mathcal{V}_v (t,\theta, v; f_2) |
        \le C\bigl(1+R+M_2(f_1)^\hal+M_2(f_2)^\hal\bigr)
        W_2(f_1,f_2),
    \]
    where the constant $C>0$ depends on $\|\g\|_{C^{3,1}}$.
\end{lemma}

\begin{remark}
    Once these two lemmas are established, the Cauchy--Lipschitz theorem gives
    existence and uniqueness of the flow in \eqref{eq:MMD-characteristics}.
    The flow representation \eqref{eq:MMDFlow} and its stability with respect
    to $f$ then give the corresponding well-posedness theory for
    \eqref{eq:MMD}.  It remains only to control the velocity moments needed in
    \cref{lem:Lip,lem:measureStability}; this is done in
    \cref{prop:estimates}.
\end{remark}

\begin{proof}[Proof of \cref{lem:Lip}]
    We fix some $R>0$ and $|v|\le R$. H\"older's inequality gives
    \begin{align*}
        \|\nab^k \nab\g\ast (\rho - \mu )\|_{L^\infty} \le 2 \| \nab^{k+1}\g\|_{L^\infty}.
    \end{align*}
    Moreover, expanding the convolution gives
    \begin{align*}
      \|\nab^k_\theta  (\nab^2\g : vv^T)\ast (f-\mu)\|
      &\le \|\nab^{k+2}\g \|_{L^\infty}
        \bigl(2R^2+2R M_1(f)+M_2(f)\bigr)\\
      &\le \|\nab^{k+2}\g \|_{L^\infty}
        \bigl(3R^2+2M_2(f)\bigr),
    \end{align*}
    where we have used $ab \le a^2/2  + b^2/2$ to obtain the last line. Now, we differentiate with respect to $v$; for any matrix increment $h\in\R^{d\times r}$,
    \begin{align*}
        \nab_v\big[(\nab^2 \g : vv^T) \ast (f-\mu)\big] &= 2 (\nab^2 \g \, v)\ast (f-\mu), \\
        \nab_v^2\big[(\nab^2 \g : vv^T) \ast (f-\mu)\big]h &= 2 \big[\nab^2 \g \ast (\rho - \mu)\big]h.
    \end{align*}
    The higher-order derivatives in $v$ vanish identically. Therefore,
    \begin{align*}
        & \| \nab_v \nab_\theta^k (\nab^2 \g : vv^T) \ast (f-\mu) \| \\
        &\le 4 R \|\nab^{k+2}\g \|_{L^\infty} + 2\|\nab^{k+2} \g \|_{L^\infty} \int_{\R^d\times \R^{d\times r}} |v| f(\theta, v)\, d\theta dv,
    \end{align*}
    and
    \begin{align*}
         \| \nab_v^2 \nab_\theta^k (\nab^2 \g : vv^T) \ast (f-\mu) \|_{L^\infty} \le 4 \|\nab^{k+2}\g\|_{L^\infty}.
    \end{align*}

    Let us now focus on the regularity of $\mathcal{V}_v$. We have
    \begin{align*}
        \| \nab_\theta^{k} (\nab^2\g\, v)\ast (f-\mu) \| &\le 2 R \|\nab^{k+2}\g\|_{L^\infty} + \|\nab^{k+2}\g\|_{L^\infty} \int_{\R^d\times \R^{d\times r}} |v|f(\theta, v)\, d\theta dv \\
       &\le \left(2R+\hal\right)\|\nab^{k+2}\g\|_{L^\infty} + \hal \|\nab^{k+2}\g\|_{L^\infty} \int_{\R^d\times \R^{d\times r}} |v|^2f(\theta, v)\, d\theta dv.
    \end{align*}
    Moreover, differentiating in $v$ in the matrix direction $h$ gives
    \begin{align*}
        \nab_v\big[(\nab^2\g\, v) \ast (f-\mu)\big]h = \big[\nab^2 \g\ast (\rho - \mu)\big]h.
    \end{align*}
    Therefore,
    \begin{align*}
        \| \nab_v \nab_\theta^{k} (\nab^2\g\, v)\ast (f-\mu) \|_{L^\infty} \le 2\|\nab^{k+2}\g\|_{L^\infty}.
    \end{align*}
    Again, all the higher-order terms in $v$ identically vanish.
\end{proof}

\begin{proof}[Proof of \cref{lem:measureStability}]
    We have
    \begin{align*}
        \mathcal{V}_\theta (t,\theta, v; f_1) - \mathcal{V}_\theta (t,\theta, v; f_2) = -\nab\g\ast (\rho_1-\rho_2) - \hal \nab_\theta (\nab^2 \g: vv^T) \ast (f_1-f_2),
    \end{align*}
    and
    \begin{align*}
         \mathcal{V}_v (t,\theta, v; f_1) - \mathcal{V}_v (t,\theta, v; f_2) = -(\nab^2 \g\, v) \ast (f_1 -f_2).
    \end{align*}
    For the first term of $\mathcal{V}_\theta$, Kantorovich--Rubinstein duality gives
    \begin{align*}
        |\nab\g\ast (\rho_1-\rho_2)| \le \|\nab^2\g\|_{L^\infty} W_1(\rho_1, \rho_2) \le \|\nab^2\g\|_{L^\infty} W_2(f_1, f_2).
    \end{align*}
    For the oscillatory part, let $\pi$ be a coupling between $f_1$ and $f_2$ for the $W_2$ distance. We have
    \begin{align*}
        & \nab_\theta (\nab^2 \g: vv^T) \ast (f_1-f_2) = \int \nab^3 \g(\theta - \theta') : (v - v')(v - v')^T \, d(f_1 - f_2)(\theta', v') \\
        &= \int \big[ \nab^3 \g(\theta - \theta_1') : (v - v_1')(v - v_1')^T - \nab^3 \g(\theta - \theta_2') : (v - v_2')(v - v_2')^T \big] \, d\pi(Z_1', Z_2') \\
        &= \int \big[ \nab^3 \g(\theta - \theta_1') : (v_2'-v_1')(v-v_1')^T + \nab^3 \g(\theta - \theta_1') : (v - v_2')(v-v_1')^T \\
        &- \nab^3\g(\theta - \theta_2'): (v - v_2')(v-v_1')^T - \nab^3\g(\theta - \theta_2') : (v - v_2')(v_1' -v_2')^T   \big]\, d\pi(Z_1', Z_2') \\
        &=  \int  \nab^3 \g(\theta - \theta_1') : (v_2'-v_1')(v-v_1')^T \, d\pi(Z_1',Z_2') \\
        &+ \int \big[ \nab^3\g(\theta - \theta_1' ) - \nab^3\g(\theta - \theta_2') \big] : (v - v_2')(v-v_1')^T \, d\pi(Z_1', Z_2') \\
        &- \int \nab^3\g(\theta - \theta_2') : (v - v_2')(v_1' -v_2')^T   \, d\pi(Z_1', Z_2').
    \end{align*}
    The first integral is bounded by
    \begin{align*}
       & \|\nab^3\g \|_{L^\infty} \int  \big( |v| + |v_1'|\big) |v_2'-v_1'| \, d\pi(Z_1', Z_2') \\
       &\le \|\nab^3\g \|_{L^\infty}  \bigg( \int (|v|+ |v_1'|)^2 \, d\pi(Z_1', Z_2')\bigg)^\hal \bigg( \int |v_2'-v_1'|^2 \, d\pi(Z_1', Z_2')\bigg)^\hal.
    \end{align*}
    The second one is bounded by
    \begin{align*}
        &\|\nab^4\g\|_{L^\infty} \int |\theta_1' - \theta_2'| |v-v_2'||v-v_1'| \, d\pi(Z_1', Z_2') \\
        &\le \|\nab^4\g\|_{L^\infty} \bigg( \int|v-v_2'|^2|v-v_1'|^2 \, d\pi(Z_1', Z_2')\bigg)^\hal \bigg( \int |\theta_1'-\theta_2'|^2 \, d\pi(Z_1', Z_2')\bigg)^\hal.
    \end{align*}
    We then use $|v|\le R$ to obtain
    \begin{align*}
        &\|\nab^4\g\|_{L^\infty}
          \left(\int|v-v_2'|^2|v-v_1'|^2\,d\pi\right)^\hal
          \left(\int|\theta_1'-\theta_2'|^2\,d\pi\right)^\hal\\
        &\quad\le 4 \|\nab^4\g\|_{L^\infty}
          \left(R^4+R^2\bigl(M_2(f_1)+M_2(f_2)\bigr)
          +M_4(f_1)^\hal M_4(f_2)^\hal\right)^\hal\\
        &\hspace{7em}\times
          \left(\int |\theta_1'-\theta_2'|^2\,d\pi\right)^\hal.
    \end{align*}
    Finally, the last term is bounded similarly to the first one. Taking the infimum over the couplings $\pi$, we obtain the stability result for $\mathcal{V}_\theta$.

    For $\mathcal{V}_v$, we have
    \begin{align*}
       & \mathcal{V}_v (t,\theta, v; f_1) - \mathcal{V}_v (t,\theta, v; f_2)  \\
       &= -\int \big[ \nab^2\g(\theta - \theta_1) (v-v_1) - \nab^2 \g(\theta - \theta_2) (v-v_2) \big] \, d\pi (Z_1, Z_2) \\
       &= -\int \big[ (\nab^2\g(\theta - \theta_1) - \nab^2 \g(\theta - \theta_2)) (v-v_1) + \nab^2 \g(\theta - \theta_2) (v_2-v_1) \big] \, d\pi (Z_1, Z_2).
    \end{align*}
    We therefore have
    \begin{align*}
        & |\mathcal{V}_v (t,\theta, v; f_1) - \mathcal{V}_v (t,\theta, v; f_2) | \\
        &\le \|\nab^3\g\|_{L^\infty} \int |\theta_1 - \theta_2| |v-v_1|  \, d\pi (Z_1, Z_2) + \|\nab^2\g\|_{L^\infty} \int |v_1 - v_2 | \, d\pi (Z_1, Z_2) \\
        &\le \|\nab^3\g\|_{L^\infty} \bigg(\int |\theta_1 - \theta_2|^2 \, d\pi (Z_1, Z_2) \bigg)^\hal ( 2R^2 + 2M_2(f_1) )^\hal \\
        &+ \|\nab^2\g\|_{L^\infty}\bigg(\int |v_1 - v_2|^2 \, d\pi (Z_1, Z_2) \bigg)^\hal .
    \end{align*}
    This allows us to conclude by considering the infimum over couplings.
\end{proof}

\subsubsection{Well-posedness of the equation}

We now establish the estimates needed to globalize the characteristic flow.
The only point is to control the velocity moments appearing in
\cref{lem:Lip,lem:measureStability}.

\begin{proposition}[A priori estimates]\label{prop:estimates}
    Let $\g \in C^{2,1}_b(\R^d)$ be even. Consider $f_0\in \mathcal{P}_2(\R^{d}\times \R^{d\times r})$ and $\mu \in \mathcal{P}(\R^d)$ such that $f_0$ is compactly supported in the $v$ variable.
    Let $f \in C([0, T], \mathcal{P}_2(\R^{d}\times \R^{d\times r}))$ be a locally Lipschitz solution to \eqref{eq:MMD} that is compactly supported in the $v$ variable for all $t\in [0, T]$. Then the following properties hold:
    \begin{enumerate}[label=\roman*.]
        \item \emph{Conservation of mass:}
        $\int \dd f_t=1$ for every $t\in[0,T]$.
        \item \emph{Propagation of moments:} for every $k\ge1$ and
        $t\in[0,T]$,
        \[
            M_k(t)\le M_k(0)
            \exp\!\left(3k\|\nabla^2\g\|_{L^\infty}t\right).
        \]
        \item \emph{Finite propagation speed in velocity:} if
        $\supp_v f_0\subset B_R(0)$, then
        \[
            \supp_v f_t\subset B_{(R+Ct)e^{Ct}}(0),
            \qquad 0\le t\le T,
        \]
        where $C$ depends only on $\|\nabla^2\g\|_{L^\infty}$ and
        $M_1(0)$.
        \item \emph{Free-energy dissipation:} the free energy defined in
        \eqref{eq:kinetic-free-energy} is nonincreasing and satisfies
        \eqref{eq:kinetic-dissipation}.
    \end{enumerate}
\end{proposition}
\begin{proof}
    The conservation of mass follows from the divergence form of
    \eqref{eq:MMD}. To establish the propagation of velocity moments, let
    $p\ge 1$ and differentiate $M_p(t)$ with respect to time:
    \begin{align*}
        \frac{\dd}{\dd t}M_p(t)
        &=-p\int |v|^{p-2}v:
          \big[(\nabla^2\g\,v)*(f_t-\mu)\big](\theta,v)\,\dd f_t(\theta,v)\\
        &\le p \|\nab^2\g\|_{L^\infty} \left( M_{p-1}(t) M_1(t) + 2 M_p(t) \right).
     \end{align*}
     Using Jensen's inequality to obtain $M_{p-1} M_1 \le M_p$, we have
     $\frac{\dd}{\dd t}M_p(t) \le 3p \|\nab^2\g\|_{L^\infty } M_p(t)$,
     and the result follows by Gr\"onwall's lemma.

    The finite speed of propagation in the velocity variable is obtained by considering a characteristic curve $(\Theta(s), V(s))$, whose $v$ component satisfies
    \begin{align*}
        \frac{\dd V}{\dd s}(s) &= -(\nab^2\g V)\ast( f-\mu) (s,\Theta, V) \\
        &= - \nab^2\g \ast (\rho - \mu )(s,\Theta) V(s) + \int_{\R^d\times \R^{d\times r}} \nab^2\g (\Theta - \Theta') V' f(s,\Theta', V') \, d\Theta'dV'.
    \end{align*}
    Since $M_1$ is bounded on $[0,T]$, we have
    $|\frac{\dd V}{\dd s}(s)| \le C |V(s)| + C$, where $C$ depends on
    $\|\nab^2 \g\|_\infty$ and the first moment. Gr\"onwall's lemma then
    implies $|V(t)| \le (|V(0)| + Ct) e^{Ct}$, so that the support remains
    compact for all $t\in [0,T]$.

    Since $\g$ is even, the interaction kernel in \eqref{eq:kinetic-free-energy} is symmetric under exchanging the two phase-space variables. Hence $\mathcal V=-\nabla_{\theta,v}(\delta\mathcal F/\delta f)$.
    Finally, differentiating the free energy along the equation gives
    \begin{align*}
        \frac{\dd}{\dd t}\mathcal{F}(f_t)
        &=-\int \left| \nab\g\ast (\rho_t - \mu)
          + \hal \nab_\theta (\nab^2\g : vv^T)
            \ast (f_t-\mu) \right|^2\,\dd f_t \\
        &\phantom{={}}-\int | (\nab^2\g \,  v )
          \ast (f_t-\mu) |^2 \, \dd f_t.
    \end{align*}
    Since both terms are non-positive, $\mathcal{F}$ is non-increasing.
\end{proof}

\begin{proposition}[Uniqueness]
    Let $f_0 \in \mathcal{P}_2(\R^d \times \R^{d \times r})$ be a probability measure with compact support in the $v$ variable. There exists at most one solution $f \in C([0, T], \mathcal{P}_2(\R^d \times \R^{d \times r}))$ to \eqref{eq:MMDFlow} which maintains compact support in $v$ for all $t \in [0, T]$.
\end{proposition}
\begin{proof}
    Let $f_1, f_2$ be two such solutions starting from $f_0$. Let $Z_1, Z_2$ be their respective characteristic flows defined by \eqref{eq:MMD-characteristics}. Since $f_i(t) = Z_i(t, \cdot)_\# f_0$, the measure $\pi_t := (Z_1(t, \cdot), Z_2(t, \cdot))_\# f_0$ is a valid coupling between $f_1(t)$ and $f_2(t)$. By definition of the Wasserstein distance as an infimum over couplings, we have
    \begin{equation*}
        W_2^2(f_1(t), f_2(t)) \le  \int_{\R^d \times \R^{d \times r}} |Z_1(t, z) - Z_2(t, z)|^2 df_0(z),
    \end{equation*}
    where $z = (\theta, v)$. The difference in trajectories satisfies
    \begin{align*}
        \frac{d}{dt} |Z_1 - Z_2| &\le |\mathcal{V}(t, Z_1; f_1) - \mathcal{V}(t, Z_2; f_2)| \\
        &\le |\mathcal{V}(t, Z_1; f_1) - \mathcal{V}(t, Z_1; f_2)| + |\mathcal{V}(t, Z_1; f_2) - \mathcal{V}(t, Z_2; f_2)|.
    \end{align*}
    By the compact support assumption, there exists $R > 0$ such that $\supp_v f_1(t) \cup \supp_v f_2(t) \subset B_R$ for all $t \in [0, T]$.
    Using \cref{lem:measureStability}, the first term is bounded by $C_R W_2(f_1, f_2)$, while
    the second term is bounded by $C_R |Z_1 - Z_2|$ thanks to \cref{lem:Lip}.
    Combining these, we have
    \begin{align*}
        \frac{d}{dt} \int |Z_1 - Z_2|^2 \, df_0 &= 2 \int (Z_1 - Z_2) \cdot \frac{d}{dt}(Z_1 - Z_2) \, df_0 \\
        &\le 2C_R \int |Z_1 - Z_2| ( |Z_1 - Z_2| +  W_2(f_1, f_2)) \, df_0 \\
        &\le 4C_R \int |Z_1 - Z_2|^2 \, df_0 + C_R W_2^2(f_1, f_2),
    \end{align*}
    where we used $ab\le a^2/2 + b^2/2$. We conclude by Gr\"onwall's lemma that $W_2(f_1(t), f_2(t)) = 0$ for all $t \in [0, T]$, which implies uniqueness.
\end{proof}

\subsection{Conditional convergence to the target}
\label{app:kinetic-long-time}

For classical MMD gradient flow, \citet[Proposition~7]{arbel2019maximum}
prove global convergence under a uniform bound on a weighted negative
Sobolev norm. Kernel-specific convergence results have also been obtained
for Sobolev discrepancies \citep{chizat2026quantitative}, while
\citet{chodron2026coulomb} and \citet{rosenzweig2026energy} study convergence and
rigidity of Lagrangian critical points for Coulomb and energy kernels.
Here we work with a smooth kernel and quantify how kinetic fluctuations
perturb the MMD convergence argument.

For probability measures $\rho,\mu$ on $\mathbb T^d$, write
\[
    \mmd_{\g}(\rho,\mu)
    :=\int\g*(\rho-\mu)\,\dd(\rho-\mu),
\]
and define the weighted negative Sobolev norm
\begin{equation}\label{eq:kinetic-weighted-dual-norm}
    \|\rho-\mu\|_{\dot H^{-1}(\rho)}
    :=\sup_{\substack{\varphi\in C^1(\mathbb T^d)\\
              \int|\nabla\varphi|^2\,\dd\rho\le1}}
       \left|\int\varphi\,\dd(\rho-\mu)\right|.
\end{equation}

\begin{theorem}[Conditional convergence of the kinetic flow]
\label{thm:kinetic-target-convergence}
Let $\g\in C^{3,1}(\mathbb T^d)$ be even, with
$\widehat\g(k)\ge0$ for all $k$ and $\widehat\g(k)>0$ for $k\ne0$,
and let $\mu\in\mathcal P(\mathbb T^d)$. Let $f_t$ be the characteristic
solution of \eqref{eq:MMD} on $\mathbb T^d\times\R^{d\times r}$, from
initial data $f_0\in\mathcal P_2(\mathbb T^d\times\R^{d\times r})$
compactly supported in $v$, and denote its positional
marginal by $\rho_t$. Set
\[
    M_2(t)=\int|v|^2\,d f_t.
\]
Suppose that, for some $C_0>0$,
\begin{equation}\label{eq:kinetic-dual-bound}
    \sup_{t\ge0}\|\rho_t-\mu\|_{\dot H^{-1}(\rho_t)}\le C_0.
\end{equation}
Then
\begin{equation}\label{eq:kinetic-perturbed-mmd-dissipation}
    \frac{\dd}{\dd t}\mmd_{\g}(\rho_t,\mu)
    \le-\frac{2}{C_0^2}\bigl[\mmd_{\g}(\rho_t,\mu)\bigr]^2
       +2C_{\g}M_2(t)\sqrt{\mmd_{\g}(\rho_t,\mu)},
\end{equation}
where $C_{\g}=\frac32\|\nabla^3\g\|_\infty\sqrt{-\Delta\g(0)}$.
In particular, if $M_2(t)\to0$, then
\begin{equation}\label{eq:kinetic-target-convergence}
    W_2(f_t,\mu\otimes\delta_0)\longrightarrow0.
\end{equation}
No time-integrability assumption on $M_2$ is required.
\end{theorem}

\begin{proof}
We first estimate the correction to the classical MMD dissipation.
Write
\[
    \phi_t=\g*(\rho_t-\mu),\qquad
    b_t(\theta,v)=\frac12\int\nabla^3\g(\theta-\theta'):
       (v-v')(v-v')^T\,\dd(f_t-\mu)(\theta',v').
\]
In this formula, $\mu$ is lifted to $\mu\otimes\delta_0$ and
$\V_\theta=-\nabla\phi_t-b_t$. Since
$\iint|v-v'|^2\,\dd f_t\,\dd f_t'\le2M_2(t)$,
\[
    \int|b_t|\,\dd f_t
    \le\frac12\|\nabla^3\g\|_\infty
       \left(\iint|v-v'|^2\,\dd f_t\,\dd f_t'+M_2(t)\right)
    \le\frac32\|\nabla^3\g\|_\infty M_2(t).
\]
Differentiating the squared MMD along the characteristics gives
\begin{equation}\label{eq:kinetic-mmd-energy-identity}
    \frac{\dd}{\dd t}\mmd_{\g}(\rho_t,\mu)
    =-2J_t-2\int\nabla\phi_t\cdot b_t\,\dd f_t,
    \qquad
    J_t:=\int|\nabla\phi_t|^2\,\dd\rho_t.
\end{equation}
By homogeneity in the test function, \eqref{eq:kinetic-dual-bound} gives
\[
    \mmd_{\g}(\rho_t,\mu)=\int\phi_t\,\dd(\rho_t-\mu)
    \le C_0\sqrt{J_t},\qquad
    J_t\ge\frac{1}{C_0^2}\bigl[\mmd_{\g}(\rho_t,\mu)\bigr]^2.
\]
To bound the remaining term, Fourier expansion and Cauchy--Schwarz yield
\[
    \|\nabla\phi_t\|_\infty^2
    \le\left(\sum_{k\ne0}|k|^2\widehat\g(k)\right)
       \left(\sum_{k\ne0}\widehat\g(k)
                   |\widehat{\rho_t-\mu}(k)|^2\right)
    =-\Delta\g(0)\,\mmd_{\g}(\rho_t,\mu).
\]
Combining these estimates with \eqref{eq:kinetic-mmd-energy-identity}
proves \eqref{eq:kinetic-perturbed-mmd-dissipation}.

Suppose now that $M_2(t)\to0$. For any $\varepsilon>0$, choose $T$ so
that $C_{\g}M_2(t)\le\varepsilon^{3/2}/(2C_0^2)$ for $t\ge T$.
Whenever $\mmd_{\g}(\rho_t,\mu)\ge\varepsilon$,
\eqref{eq:kinetic-perturbed-mmd-dissipation} then gives
\[
    \frac{\dd}{\dd t}\mmd_{\g}(\rho_t,\mu)
    \le-\frac{1}{C_0^2}\bigl[\mmd_{\g}(\rho_t,\mu)\bigr]^2
    \le-\frac{\varepsilon^2}{C_0^2}.
\]
Thus the squared MMD enters $[0,\varepsilon]$ in finite time and cannot
subsequently cross $\varepsilon$ upwards. Letting $\varepsilon\to0$
proves $\mmd_{\g}(\rho_t,\mu)\to0$.

Finally, compactness of the torus gives a narrowly convergent subsequence
of any sequence $(\rho_{t_n})$ with $t_n\to\infty$. The Fourier formula for
$\mmd_{\g}(\rho_t,\mu)$ and strict positivity of $\widehat\g(k)$ imply that every such limit
has the same Fourier coefficients as $\mu$, and hence equals $\mu$.
Thus $\rho_t\to\mu$ narrowly and in $W_2$. For the product metric,
\[
    W_2(f_t,\mu\otimes\delta_0)^2
    =M_2(t)+W_2(\rho_t,\mu)^2,
\]
because the target fluctuation is always zero and an optimal positional
coupling can be lifted using the conditional law of $v$ given $\theta$.
This proves \eqref{eq:kinetic-target-convergence}.
\end{proof}

\begin{corollary}[Rates from fluctuation decay]
\label{cor:kinetic-mmd-rate}
Under the assumptions of \cref{thm:kinetic-target-convergence}, suppose
$M_2(t)\le A(1+t)^{-\alpha}$ for some $A,\alpha>0$. Set
$q=\min\{1,2\alpha/3\}$. Then, for a constant $C$ depending only on
$\g,C_0,A,\alpha$ and $\mmd_{\g}(\rho_0,\mu)$,
\begin{equation}\label{eq:kinetic-mmd-polynomial-rate}
    \mmd_{\g}(\rho_t,\mu)\le C(1+t)^{-q}.
\end{equation}
\end{corollary}
\begin{proof}
Choose $C>\mmd_{\g}(\rho_0,\mu)$ large enough that
$2C^2/C_0^2-2C_{\g}A\sqrt C\ge qC$.
Since $3q/2\le\alpha$ and $q\le1$, the positive function
$y_t=C(1+t)^{-q}$ satisfies
\[
    \frac{\dd y_t}{\dd t}
    \ge-\frac{2}{C_0^2}y_t^2
       +2C_{\g}A(1+t)^{-\alpha}\sqrt{y_t}.
\]
Comparison with \eqref{eq:kinetic-perturbed-mmd-dissipation} gives
$\mmd_{\g}(\rho_t,\mu)\le y_t$. Since $y_t>0$, the right-hand side is
locally Lipschitz at any possible contact point.
\end{proof}

This is a squared-MMD rate, not a rate in $L^p$ norms of densities.
When $\alpha\ge3/2$, the estimate recovers the classical $t^{-1}$ squared-MMD
bound; no optimality of these exponents is asserted. The dual bound
\eqref{eq:kinetic-dual-bound} and the decay of $M_2$ are assumptions on the
trajectory, not consequences of well-posedness.

\subsection{Propagation of higher regularity}

We show that regularity in the full variable $z=(\theta,v)$, identifying $\R^{d\times r}$ with $\R^{dr}$, is propagated by the dynamics. 

A multiindex
$\alpha=(\alpha_1,\ldots,\alpha_{d+dr})\in\mathbb N_0^{d+dr}$ records
the number of derivatives in each coordinate:
\[
    |\alpha|=\sum_{j=1}^{d+dr}\alpha_j,
    \qquad
    D_z^\alpha=\partial_{z_1}^{\alpha_1}\cdots
               \partial_{z_{d+dr}}^{\alpha_{d+dr}}.
\]
Here $W^{s,q}=W^{s,q}(\R^d\times\R^{d\times r})$ includes all weak
derivatives of total order at most $s$, including mixed derivatives, and
we use the norm
$\|h\|_{W^{s,q}}=\sum_{|\alpha|\le s}\|D_z^\alpha h\|_{L^q}$.

\begin{proposition}[Propagation of joint regularity]\label{prop:regularity}
    Let $s\ge0$ be an integer, $q\in[1,\infty]$, $\mu\in\mathcal P(\R^d)$,
    and let $\g\in C_b^{s+4}(\R^d)$ be even. Suppose that
    $f_0\in\mathcal P_2(\R^d\times\R^{d\times r})$ has a density in
    $W^{s,q}$ supported in $\{|v|\le M\}$. Then the characteristic
    solution of \eqref{eq:MMDFlow} has a density satisfying
    \[
        \|f_t\|_{W^{s,q}}\le e^{C_Tt}\|f_0\|_{W^{s,q}},
        \qquad 0\le t\le T,
    \]
    for every $T<\infty$. The constant $C_T$ depends only on
    $T,M,s,q,d,r$ and $\|\g\|_{C_b^{s+4}}$.
\end{proposition}

\begin{proof}
Fix $T>0$ and consider the transport velocity
$\mathcal V(t,z;f_t)=(\mathcal V_\theta,\mathcal V_v)(t,z;f_t)$
defined by the right-hand sides of \eqref{eq:MMD-characteristics}.
By \cref{prop:estimates}, there is $R_T$ such that
$\supp_v f_t\subset B_{R_T}$ for $0\le t\le T$.
The $\theta$-derivatives of this velocity act on the convolution kernels,
not on $f_t$, and its dependence on $v$ is at most quadratic.
Since its expression uses at most three derivatives of $\g$, the
assumption $\g\in C_b^{s+4}$ bounds all its derivatives in $z=(\theta,v)$
through total order $s+1$ on $\R^d\times B_{R_T+1}$.
Choose a smooth cutoff $\chi(v)$ equal to one on $B_{R_T}$ and zero
outside $B_{R_T+1}$, and set
\[
    b_t(z)=\chi(v)\mathcal V(t,z;f_t).
\]
This cutoff velocity agrees with the original transport velocity
$\mathcal V(t,z;f_t)$ on the support of $f_t$ and satisfies
\[
    \sup_{0\le t\le T}\|b_t\|_{W^{s+1,\infty}}\le C_T.
\]
Let $\Phi_t$ be the characteristic flow generated by $b_t$.
It is a $C^{s+1}$ diffeomorphism; the derivatives of $\Phi_t$ and
$\Phi_t^{-1}$ of orders $1,\ldots,s+1$ are bounded on $[0,T]$,
and their Jacobian determinants are bounded away from zero.
The characteristic formula
\[
    f_t=(f_0\circ\Phi_t^{-1})\,|\det D_z\Phi_t^{-1}|
\]
and the Sobolev chain and product rules therefore give $f_t\in W^{s,q}$.
We now estimate this norm. For $|\alpha|\le s$, differentiating
$\partial_t f+b_t\cdot\nabla_z f+(\divergence_z b_t)f=0$
in the sense of distributions gives
\begin{gather*}
    \partial_t D_z^\alpha f+b_t\cdot\nabla_z D_z^\alpha f
      +(\divergence_z b_t)D_z^\alpha f=R_\alpha,\\
    R_\alpha=-\sum_{0<\beta\le\alpha}\binom\alpha\beta
    \left[(D_z^\beta b_t)\cdot\nabla_zD_z^{\alpha-\beta}f
    +D_z^\beta(\divergence_z b_t)D_z^{\alpha-\beta}f\right],
\end{gather*}
where $\beta\le\alpha$ is understood coordinatewise and
$\binom\alpha\beta=\prod_j\binom{\alpha_j}{\beta_j}$.
In the first term, the derivative of $f$ has order
$|\alpha|-|\beta|+1\le|\alpha|$; in the second, at most $s+1$
derivatives fall on $b_t$. Consequently,
\[
    \|R_\alpha\|_{L^q}\le C_T\|f\|_{W^{s,q}}.
\]

For $1<q<\infty$, computing the $L^q$ norm of $g=D_z^\alpha f$
along the characteristics of $b_t$ gives
\begin{align*}
    \frac1q\frac{\dd}{\dd t}\|g\|_{L^q}^q
    &= -\left(1-\frac1q\right)\int(\divergence_z b_t)|g|^q\,\dd z
       +\int R_\alpha |g|^{q-2}g\,\dd z\\
    &\le C_T\|g\|_{L^q}^q
       +\|R_\alpha\|_{L^q}\|g\|_{L^q}^{q-1}.
\end{align*}
Dividing by $\|g\|_{L^q}^{q-1}$, with a regularization when this
norm vanishes, gives
$\frac{\dd}{\dd t}\|g\|_{L^q}\le C_T\|g\|_{L^q}+\|R_\alpha\|_{L^q}$.
For $q=1$ and $q=\infty$, the same inequality follows by taking the
corresponding norm in the characteristic formula for $g$.
Summing over $|\alpha|\le s$ and applying
Gr\"onwall's lemma proves the estimate.
\end{proof}

\paragraph{Regularity in $\theta$ alone need not propagate.}
We record an explicit example where an initial jump in $v$ can make $\partial_\theta f_t$ singular. 

Take $d=r=1$, $\g(\theta)=\cos\theta$,
$\mu=\delta_0$, and
\[
    f_0(\theta,v)=\frac{h(\theta)}{2a}\mathbf1_{[-a,a]}(v),
    \qquad a>0,
\]
where $h$ is a smooth, even, compactly supported probability density,
with $h(\pi/2)>0$. Every $\theta$-derivative of
$f_0$ belongs to every $L^q$, but $f_0$ has jumps at $v=\pm a$.
Since $\int_{-a}^a v'\,\dd v'=0$ and
$\int\sin\theta'\,h(\theta')\,\dd\theta'=0$, substitution into
\eqref{eq:MMD-characteristics} gives
\begin{align*}
    \mathcal V_v(0,\theta,v)
    &=v\left(\int\cos(\theta-\theta')h(\theta')\,\dd\theta'
                   -\cos\theta\right)\\
    &=-c\,v\cos\theta,
    \qquad c=1-\int\cos\theta'\,h(\theta')\,\dd\theta'>0.
\end{align*}
The jump initially at $v=a$ is carried by the characteristics;
write its position at small times as $v=a_t(\theta)$.
Since $a_0(\theta)=a$ is constant, its initial velocity is
$\left.\partial_t a_t(\theta)\right|_{t=0}
=\mathcal V_v(0,\theta,a)=-a\,c\cos\theta$.
Expanding $a_t$ and its $\theta$-derivative in time gives
\[
    a_t(\theta)=a\bigl(1-c\,t\cos\theta\bigr)+o(t),
    \qquad
    \partial_\theta a_t(\pi/2)=a\,c\,t+o(t)>0
\]
for sufficiently small $t>0$. Thus the jump location, initially
independent of $\theta$, now varies with $\theta$.

Near this jump and $\theta=\pi/2$, the transported density has the form
$f_t(\theta,v)=A_t(\theta,v)\mathbf1_{\{v<a_t(\theta)\}}$, where
$A_t$ is smooth and positive. Differentiating in $\theta$, one obtains:
\[
    \partial_\theta f_t
    =(\partial_\theta A_t)\mathbf1_{\{v<a_t(\theta)\}}
     +A_t(\theta,a_t(\theta))\,\partial_\theta a_t(\theta)\,
       \d\bigl(v-a_t(\theta)\bigr).
\]
The second term is singular in $v$, and its coefficient is nonzero
near $\theta=\pi/2$. Hence, for every sufficiently small $t>0$,
the first weak derivative $\partial_\theta f_t$ is not an $L^q$
function for any $1\le q\le\infty$, although all $\theta$-derivatives
of $f_0$ are in $L^q$.

\subsection{Proof of the mean-field limit}
\label{app:kinetic-limit-proof}

In this section, we prove \cref{thm:kinetic-mean-field-limit}. 

The passage to the limit on a fixed phase space follows the classical
Braun--Hepp--Dobrushin argument \citep{braun1977vlasov,dobrushin1979vlasov}.
For more recent accounts of mean-field limits, we refer to
\citet{chaintron2022propagation,golse2016dynamics,jabin2014review}.
The non-standard point has to do with the dimension of the fluctuation covariance: the
full $\Sigma$ is a $dN\times dN$ matrix, whereas the kinetic
equation is posed on the $N$-independent phase space
$\R^d\times\R^{d\times r}$.  We overcome this obstruction by introducing an
intermediate rank-$r$ particle system and proving that the full dynamics stays
close to it.  This first step is essentially a stability estimate for two
finite-dimensional ODEs.  Once it is established, the classical mean-field
argument applies directly.

We first restate the theorem:

\begin{theorem}[Validity of the kinetic limit]
\label{thm:kinetic-mean-field-limit-previous-appendix}
    Fix $\eta>0$ and $r$ independently of $N$, let $\g\in C_b^4(\R^d)$ be even, and let
    $\mu\in\mathcal P(\R^d)$.  Let
    $f_0\in\mathcal P_2(\R^d\times\R^{d\times r})$ be supported in
    $\{\|v\|\le M\}$.  For each $N$, choose
    $z_i^{N,0}=(\theta_i^{N,0},v_i^{N,0})$ with
    $\max_i\|v_i^{N,0}\|\le M$ and set
    \[
        f_N^0=\frac1N\sum_{i=1}^N\delta_{z_i^{N,0}},
        \qquad
        \Sigma^N(0)=V^N(0)V^N(0)^T,
    \]
    where the $i$-th block row of $V^N(0)$ is $v_i^{N,0}$.
    For a matrix $S$ with $d\times d$ blocks, write
    $\|S\|_{\max}=\max_{i,j}\|S_{ij}\|$.

    Let $(\theta^N,\Sigma^N)$ solve the full system \eqref{eq:particleMF}.
    Let $(\bar\theta^N,V^N)$ solve the rank-$r$ system obtained by deleting the
    term quadratic in the Hessian, with the same initial positions and factor, and
    define
    \[
        \bar f_N(t)=\frac1N\sum_{i=1}^N
        \delta_{(\bar\theta_i^N(t),v_i^N(t))}.
    \]
    Finally, let $f_t$ be the solution of \eqref{eq:MMDFlow} with initial law
    $f_0$, and let $\rho_t$ be its $\theta$-marginal.  Then, for every $T>0$,
    there is a constant $C_T$, independent of $N$, such that
    \begin{align}
        &\sup_{0\le t\le T}W_1(\bar f_N(t),f_t)
        \le C_T W_1(f_N^0,f_0),
        \label{eq:kinetic-mean-field-limit}\\
        &\sup_{0\le t\le T}\left(
        \max_i|\theta_i^N(t)-\bar\theta_i^N(t)|
        +\|\Sigma^N(t)-V^N(t)V^N(t)^T\|_{\max}\right)
        \le \frac{C_T}{N}.
        \label{eq:kinetic-rank-approximation}
    \end{align}
    Consequently, for
    $\rho_N^{\rm full}(t)=N^{-1}\sum_i\delta_{\theta_i^N(t)}$,
    \begin{equation}
        \sup_{0\le t\le T}W_1(\rho_N^{\rm full}(t),\rho_t)
        \le C_T\left(W_1(f_N^0,f_0)+\frac1N\right).
        \label{eq:kinetic-full-marginal-limit-previous-appendix}
    \end{equation}
    In particular, if $(z_i^0)_{i\ge1}$ is an i.i.d. sequence with law $f_0$
    and $z_i^{N,0}=z_i^0$ for $i\le N$, then all three conclusions hold almost
    surely with right-hand sides tending to zero.
\end{theorem}

\paragraph{Full and fixed-rank particle systems.}

For $\boldsymbol\theta=(\theta_1,\ldots,\theta_N)$, set
\[
    \rho_N=\frac1N\sum_{i=1}^N\delta_{\theta_i},
    \qquad
    E_N(\boldsymbol\theta)
    =\frac12\int \g*(\rho_N-\mu)\,\dd(\rho_N-\mu).
\]
We suppress the superscript $N$ on all particle variables.  In the time scale
of \eqref{eq:particleMF}, the full system is
\begin{equation}
\label{eq:kinetic-full-particle-system}
\begin{cases}
    \displaystyle
    \frac{\dd\theta_i}{\dd t}
    =-N\nabla_{\theta_i}E_N(\boldsymbol\theta)
     -\frac N2\nabla_{\theta_i}\nabla^2E_N(\boldsymbol\theta):\Sigma,
    &1\le i\le N,\\[8pt]
    \displaystyle
    \frac{\dd\Sigma_{ij}}{\dd t}
    =-N\sum_{k=1}^N\left(
       \nabla^2_{\theta_i\theta_k}E_N(\boldsymbol\theta)\Sigma_{kj}
       +\Sigma_{ik}\nabla^2_{\theta_k\theta_j}E_N(\boldsymbol\theta)
       \right)
    &\\[-2pt]
    \displaystyle\hspace{34mm}
     +\eta N\sum_{k,l=1}^N
       \nabla^2_{\theta_i\theta_k}E_N(\boldsymbol\theta)
       \Sigma_{kl}
       \nabla^2_{\theta_l\theta_j}E_N(\boldsymbol\theta),
    &1\le i,j\le N.
\end{cases}
\end{equation}
Starting from $\Sigma(0)=V(0)V(0)^T$, where the $i$-th block row of
$V$ is $v_i\in\R^{d\times r}$, we compare it with the fixed-rank system
\begin{equation}
\label{eq:kinetic-reduced-particle-system}
\begin{cases}
    \displaystyle
    \frac{\dd\bar\theta_i}{\dd t}
    =-N\nabla_{\theta_i}E_N(\bar{\boldsymbol\theta})
     -\frac N2\nabla_{\theta_i}\nabla^2E_N(\bar{\boldsymbol\theta})
       :VV^T,
    &1\le i\le N,\\[8pt]
    \displaystyle
    \frac{\dd v_i}{\dd t}
    =-N\sum_{j=1}^N
       \nabla^2_{\theta_i\theta_j}E_N(\bar{\boldsymbol\theta})v_j,
    &1\le i\le N.
\end{cases}
\end{equation}
Writing $Q_{ij}=v_i v_j^T$, the second equation implies
\[
    \frac{\dd Q_{ij}}{\dd t}
    =-N\sum_{k=1}^N\left(
       \nabla^2_{\theta_i\theta_k}E_N(\bar{\boldsymbol\theta})Q_{kj}
       +Q_{ik}\nabla^2_{\theta_k\theta_j}
          E_N(\bar{\boldsymbol\theta})\right),
    \qquad \rank Q(t)\le r.
\]
Thus each particle in the intermediate dynamics evolves in the fixed phase space $\R^d\times\R^{d\times r}$.  The next
proposition contains the non-classical part of the proof.  It states that
every $(i,j)$ block of the full covariance remains $O(N^{-1})$-close to the
corresponding block of the rank-$r$ matrix $Q$, although the term quadratic in the Hessian in
\eqref{eq:kinetic-full-particle-system} need not preserve algebraic rank.

\begin{proposition}[Stability of the fixed-rank reduction]
\label{prop:kinetic-rank-approximation}
    Under the assumptions of \cref{thm:kinetic-mean-field-limit}, suppose that
    \eqref{eq:kinetic-full-particle-system} and
    \eqref{eq:kinetic-reduced-particle-system} start from the same positions,
    from $\Sigma(0)=V(0)V(0)^T$, and from
    $\max_i\|v_i(0)\|\le M$.  Then, for every $T>0$,
    \[
        \sup_{0\le t\le T}\left(
        \max_i|\theta_i(t)-\bar\theta_i(t)|
        +\|\Sigma(t)-V(t)V(t)^T\|_{\max}\right)
        \le \frac{C_T}{N}.
    \]
\end{proposition}

\begin{proof}
For a block matrix, write $\|S\|_{\max}=\max_{i,j}\|S_{ij}\|$.  The rescaled
Hessian can be written blockwise as
\begin{equation}
    N\nabla^2_{\theta_i\theta_j}E_N(\boldsymbol\theta)
    =\delta_{ij}\nabla^2\g*(\rho_N-\mu)(\theta_i)
    -\frac1N\nabla^2\g(\theta_i-\theta_j).
    \label{eq:kinetic-hessian-blocks}
\end{equation}
Hence, uniformly in $N$,
\[
\begin{aligned}
    \max_i\sum_{j=1}^N
    \big\|N\nabla^2_{\theta_i\theta_j}
       E_N(\boldsymbol\theta)\big\|_{\mathrm{op}}
    &\le C,\\
    \max_i\sum_{j=1}^N
    \big\|N\nabla^2_{\theta_i\theta_j}E_N(\boldsymbol\theta)
          -N\nabla^2_{\theta_i\theta_j}E_N(\bar{\boldsymbol\theta})
    \big\|_{\mathrm{op}}
    &\le C\max_i|\theta_i-\bar\theta_i|.
\end{aligned}
\]
Bounded fourth derivatives of $\g$ similarly show that, whenever
$\|S\|_{\max}+\|\bar S\|_{\max}\le K$, the maximal difference between the
right-hand sides of the position equations evaluated at
$(\boldsymbol\theta,S)$ and $(\bar{\boldsymbol\theta},\bar S)$ is bounded by
\[
    C_K\left(\max_i|\theta_i-\bar\theta_i|
             +\|S-\bar S\|_{\max}\right).
\]
The constants are independent of $N$: each off-diagonal block carries a
factor $N^{-1}$, while every diagonal term is a normalized sum.

By symmetry, the same bounds hold when one sums over $i$ with $j$ fixed.
They therefore control multiplication by $N\nabla^2E_N$ on either side in
the $\|\cdot\|_{\max}$ norm.  Gronwall's inequality yields, on $[0,T]$,
\[
    \max_i\|v_i(t)\|\le Me^{Ct},
    \qquad
    \|\Sigma(t)\|_{\max}+\|Q(t)\|_{\max}\le C_T.
\]

Set $R=\Sigma-Q$ and
\[
    u(t)=\max_i|\theta_i(t)-\bar\theta_i(t)|,
    \qquad r(t)=\max_{i,j}\|R_{ij}(t)\|.
\]
For any $i$, the integral form of the two position equations, their common
initial value, and the Lipschitz estimate above give
\[
    |\theta_i(t)-\bar\theta_i(t)|
    \le C_T\int_0^t\big(u(s)+r(s)\big)\,\dd s.
\]
On the other hand, subtraction of the equations for $\Sigma$ and $Q$ gives
\[
\begin{aligned}
    \dot R
    ={}&-N\nabla^2E_N(\boldsymbol\theta)R
        -NR\nabla^2E_N(\boldsymbol\theta)\\
      &-N\big(\nabla^2E_N(\boldsymbol\theta)
               -\nabla^2E_N(\bar{\boldsymbol\theta})\big)Q\\
      &-NQ\big(\nabla^2E_N(\boldsymbol\theta)
               -\nabla^2E_N(\bar{\boldsymbol\theta})\big)\\
      &+\eta N\nabla^2E_N(\boldsymbol\theta)\Sigma
               \nabla^2E_N(\boldsymbol\theta).
\end{aligned}
\]
All terms except the last one are Lipschitz in the current error; the last is
$O(N^{-1})$ in block norm.  Since $R(0)=0$, for any $i,j$,
\[
    \|R_{ij}(t)\|
    \le C_T\int_0^t\big(u(s)+r(s)\big)\,\dd s
       +\frac{C_Tt}{N}.
\]
Taking the maximum first over $i$, then over $(i,j)$, and using $t\le T$, we
obtain
\[
    u(t)+r(t)
    \le C_T\int_0^t\big(u(s)+r(s)\big)\,\dd s+\frac{C_T}{N}.
\]
Gronwall's inequality proves \eqref{eq:kinetic-rank-approximation}.
\end{proof}

\paragraph{Mean-field limit.}

It remains to pass to the limit in the fixed-rank system.  This is the
classical Dobrushin argument, which we give for completeness.

\begin{proof}[Proof of \cref{thm:kinetic-mean-field-limit}]
Set $z=(\theta,v)$ and identify the target with
$\widehat\mu=\mu\otimes\delta_0$.  The interaction kernel on the product
space is
\[
    K(z,z')
    =\g(\theta-\theta')
     +\frac12\nabla^2\g(\theta-\theta'):
       (v-v')(v-v')^T.
\]
For a probability measure $h$, write
\[
    \mathcal V[h](z)
    =-\int\nabla_zK(z,z')\,\dd(h-\widehat\mu)(z').
\]
If $z_i(t)=(\bar\theta_i(t),v_i(t))$ and
\[
    \bar f_N(t)=\frac1N\sum_{i=1}^N\delta_{z_i(t)},
\]
then direct substitution in the preceding formula shows that
\eqref{eq:kinetic-reduced-particle-system} is exactly
\[
    \frac{\dd z_i}{\dd t}=\mathcal V[\bar f_N](z_i),
    \qquad 1\le i\le N.
\]
Consequently, if $Z_N$ denotes the characteristic flow generated by
$\mathcal V[\bar f_N]$, then
\[
    \bar f_N(t)=Z_N(t,\cdot)_\#f_N^0.
\]
Likewise, if $Z$ is the characteristic flow generated by $\mathcal V[f]$,
then $f_t=Z(t,\cdot)_\#f_0$.

We first record the Lipschitz estimates used below.  The bounds on the
$v_i$ obtained in the proof of
\cref{prop:kinetic-rank-approximation}, together with
\cref{prop:estimates}, give $R_T>0$, independent of $N$, such that
\[
    \supp_v\bar f_N(t)\cup\supp_v f_t\subset B_{R_T},
    \qquad 0\le t\le T.
\]
Since $\g\in C_b^4$, differentiation of $K$ gives, for
$z,\widetilde z,z',\widetilde z'\in\R^d\times B_{R_T}$,
\begin{align}
    |\nabla_zK(z,z')-\nabla_zK(\widetilde z,z')|
    &\le L_T|z-\widetilde z|,
    \label{eq:dobrushin-kernel-space-lipschitz}\\
    |\nabla_zK(z,z')-\nabla_zK(z,\widetilde z')|
    &\le L_T|z'-\widetilde z'|.
    \label{eq:dobrushin-kernel-measure-lipschitz}
\end{align}
Indeed, the $\theta$-component of $\nabla_zK$ contains derivatives of $\g$
up to order three multiplied by at most two factors $v-v'$, while its
$v$-component is linear in $v-v'$.  One further derivative therefore uses
at most four derivatives of $\g$, and all velocity factors are bounded by
$2R_T$.

It follows from \eqref{eq:dobrushin-kernel-space-lipschitz} that, for every
probability measure $h$ supported in $v\in B_{R_T}$,
\begin{equation}
    |\mathcal V[h](z)-\mathcal V[h](\widetilde z)|
    \le 2L_T|z-\widetilde z|.
    \label{eq:dobrushin-field-space-lipschitz}
\end{equation}
The factor two accounts for the total variation of
$h-\widehat\mu$.  Moreover, if $h_1,h_2$ have such support and
$\pi\in\Pi(h_1,h_2)$, then the target terms cancel and
\begin{align}
    |\mathcal V[h_1](z)-\mathcal V[h_2](z)|
    &\le \int
       |\nabla_zK(z,z')-\nabla_zK(z,\widetilde z')|
       \,\dd\pi(z',\widetilde z')\notag\\
    &\le L_T\int|z'-\widetilde z'|\,\dd\pi(z',\widetilde z').
    \label{eq:dobrushin-field-measure-lipschitz}
\end{align}

We now construct the coupling.  Let
$\pi_0\in\Pi(f_N^0,f_0)$ be arbitrary and transport its two coordinates by
the two flows:
\[
    \pi_t
    =\big((z,\widetilde z)\mapsto
       (Z_N(t,z),Z(t,\widetilde z))\big)_\#\pi_0.
\]
Then $\pi_t\in\Pi(\bar f_N(t),f_t)$.  Define its transportation cost by
\[
    D(t)=\int
    |Z_N(t,z)-Z(t,\widetilde z)|\,\dd\pi_0(z,\widetilde z).
\]
In particular,
\begin{equation}
    W_1(\bar f_N(t),f_t)\le D(t).
    \label{eq:dobrushin-transported-coupling}
\end{equation}
For fixed $(z,\widetilde z)$, the integral form of the characteristic
equations gives
\begin{align*}
    |Z_N(t,z)-Z(t,\widetilde z)|
    \le{}&|z-\widetilde z|\\
    &+\int_0^t
      |\mathcal V[\bar f_N(s)](Z_N(s,z))
       -\mathcal V[f_s](Z(s,\widetilde z))|\,\dd s.
\end{align*}
Split the last difference by inserting
$\mathcal V[\bar f_N(s)](Z(s,\widetilde z))$.  The spatial estimate
\eqref{eq:dobrushin-field-space-lipschitz} controls the first part by
$2L_T|Z_N(s,z)-Z(s,\widetilde z)|$.  For the second part, use
\eqref{eq:dobrushin-field-measure-lipschitz} with the coupling $\pi_s$ to
obtain the bound $L_TD(s)$.  Integrating with respect to $\pi_0$ therefore
yields
\[
    D(t)\le D(0)+3L_T\int_0^tD(s)\,\dd s.
\]
Gronwall's inequality gives
\[
    D(t)\le e^{3L_Tt}D(0)
    =e^{3L_Tt}\int|z-\widetilde z|\,\dd\pi_0(z,\widetilde z).
\]
Combining this estimate with
\eqref{eq:dobrushin-transported-coupling} and finally taking the infimum over
$\pi_0\in\Pi(f_N^0,f_0)$ proves Dobrushin's estimate
\[
    \sup_{0\le t\le T}W_1(\bar f_N(t),f_t)
    \le e^{3L_TT}W_1(f_N^0,f_0),
\]
which is \eqref{eq:kinetic-mean-field-limit}.

The stability estimate of
\cref{prop:kinetic-rank-approximation} gives
\eqref{eq:kinetic-rank-approximation}.  Finally, coupling the full and
fixed-rank particles with the same index and projecting onto the $\theta$
coordinate gives
\begin{align*}
    W_1(\rho_N^{\rm full}(t),\rho_t)
    &\le \frac1N\sum_{i=1}^N
       |\theta_i(t)-\bar\theta_i(t)|
       +W_1\big((\mathsf p_\theta)_\#\bar f_N(t),\rho_t\big)\\
    &\le \max_i|\theta_i(t)-\bar\theta_i(t)|
       +W_1(\bar f_N(t),f_t).
\end{align*}
This proves \eqref{eq:kinetic-full-marginal-limit}.  For independent initial
samples with law $f_0$, the empirical measures $f_N^0$ converge almost surely
to $f_0$ in $W_1$ because $f_0\in\mathcal P_2$.
\end{proof}

\section{Additional remarks}
\label[appendix]{app:remarks}

\subsection{Scaling of the fluctuations}
\label{app:mf-scaling}

A natural question concerns the scaling of the fluctuation covariance with respect to the parameters, in particular the learning rate $\eta$. 
We argue here that the natural scaling is, at the edge of stability, $\tr(\Sigma) \approx \eta$ for a fixed loss scale, as is the case in the gradient-flow limit \cref{prop:mf-limits}. 

Indeed, the natural quantity we monitor, as we have seen, is 
\begin{equation*}
    F(\theta,\Sigma) := E(\theta) + \frac12 H(\theta):\Sigma.
\end{equation*}
At the edge of stability, we have $H : \Sigma \approx 2\tr(\Sigma)/\eta $, so that to equilibrate the energy part and the fluctuating part, one has
\begin{equation*}
    \boxed{\tr(\Sigma) \approx \eta E.}
\end{equation*}

\subsection{On the onset failure of Rod flow}
\label{app:onset-failure}

As observed by \citet{marion2026edge}, Rod flow can fail to settle at the edge of stability when initialized near the threshold: oscillations do not develop, just as in gradient flow.
In this section, we show that because the mean--fluctuation dynamics retains covariance in all directions, it can be more robust than Rod flow to this failure mode. 

First, we show that this failure mode is robust to numerical refinement, then assess the impact of flooring the covariance amplitude (a procedure used in Edge flow) and finally show that the mean--fluctuation dynamics does not always suffer from the same failure mode as Rod flow.
The explanation of the latter is that the evolution of MF retains the full covariance matrix, which can amplify initially small oscillations aligned with the top Hessian eigenvector.
This is crucial, because at the onset of EoS, the top eigenvector of $\Sigma$ (which is the one retained in Rod flow) can be nearly orthogonal to the top Hessian direction.

\begingroup
\def\rodonsetgeometrylabel{fig:rod-onset-geometry-appendix}
\providecommand{\rodonsetgeometrylabel}{fig:rod-onset-geometry}
\begin{figure}[htbp]
    \centering
    \includegraphics[width=\textwidth]{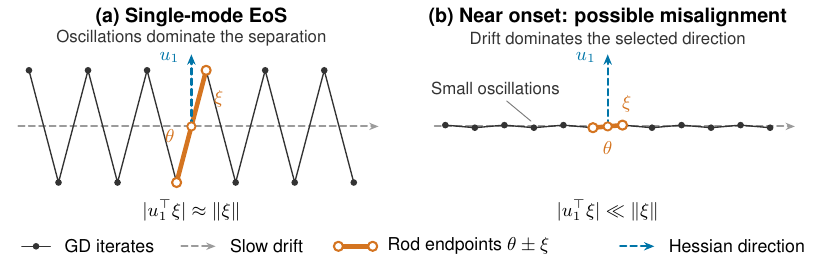}
    \caption{Rod's onset failure. The selected displacement is
    $\xi=\sqrt{s_1}v_1$, where $(s_1,v_1)$
    is the leading eigenpair of $\Sigma$.
    (a) At single-mode EoS, the oscillations are along the leading Hessian eigenvector $u_1$
    \citep[see also Fig.~2]{marion2026edge}.
    (b) Near onset, the selected direction can instead be nearly orthogonal
    to $u_1$, even if the full covariance contains smaller fluctuations
    along $u_1$. MF can amplify these smaller fluctuations.}
    \label{\rodonsetgeometrylabel}
\end{figure}

\endgroup

\paragraph{Onset comparison and numerical refinement.} We first check that this failure mode persists under numerical refinement. 
To this end, we initialize from the same saved GD window at
$\eta\lambda_1\simeq1.999$ and evolve for $4000$ steps. MF uses\footnote{We refer to \cref{app:solver} for the definition of these numerical parameters for MF. $K$ is the number of substeps per gradient descent step. }
$(r,p,K)=(3,2,4)$ as a baseline and $(4,3,8)$ for the refined run; Rod uses auxiliary rank three with
$K=4$ and rank four with $K=8$. 
The initial MF covariance and Rod
auxiliary matrix $\Sigma$ are fixed and have rank three in both runs. 
Neither MF nor Rod has an amplitude floor; Edge retains its prescribed
floor and uses $K=D=4$ in both runs.\footnote{See \cref{app:edge-resolution} for the numerical parameters of Edge flow.} In \cref{fig:rod-onset-floor},
MF recovers the growth of the GD fluctuations whereas Rod does not,
at both resolutions.
The maximum relative full-covariance differences over the $81$ common
checkpoints are $1.86\%$ for MF and $0.0506\%$ for Rod's endpoint
covariance, using the normalization of \eqref{eq:solver-covariance-gap}.
The replayed GD states agree exactly.

\paragraph{Persistent amplitude floors.}
After checking that this behavior persists under numerical refinement, we examine the role of a floor on the fluctuation amplitude. 
Indeed, at each step of Algorithm~1 \citep{marion2026edge}, the oscillation amplitude $x$ is replaced by $\max(a,x)$, which may explain why Edge flow avoids this failure mode.

We repeat the baseline Rod run with RMS floors
$a\in\{10^{-16},10^{-12},10^{-8},10^{-6},10^{-4}\}$, using the same
cached GD window, initial center, retained basis and $4000$-step horizon.

Write Rod's auxiliary matrix as
$\Sigma=\sum_{i=1}^r s_i v_i v_i^T$, with orthonormal $v_i$ and
$s_1\ge\cdots\ge s_r$. At initialization and after every covariance
substep, the floor replaces
\[
    s_i\leftarrow\max\{s_i,a^2\},\qquad i=1,\ldots,r.
\]
The two endpoints remain $\theta\pm\xi$, with
$\xi=\sqrt{s_1}\,v_1$: their covariance is $\xi\xi^T$, not the full
auxiliary matrix $\Sigma$. Along the leading Hessian eigenvector $u_1$
at the GD-window mean, the plotted endpoint variance is therefore
\[
    (u_1^T\xi)^2=s_1|u_1^Tv_1|^2,
\]
which need not exceed $a^2$ if $\xi$ and $u_1$ are nearly orthogonal. 
The floor also leaves directions outside the retained
subspace unchanged. Thus $a=10^{-6}$ matches Edge's RMS floor magnitude,
but not its rule for choosing the fluctuation direction.

\Cref{fig:rod-persistent-floors} separates two outcomes. For every tested
$a\le10^{-6}$, Rod's loss and sharpness curves overlap with the unfloored
run, and the critical endpoint variance remains negligible. Even at
$a=10^{-6}$, the principal RMS amplitude stays above
$4.5\,10^{-6}$: the floor raises the smaller extent eigenvalues, rather
than directly increasing $\|\xi\|$. At the last diagnostic sample,
$|u_1^Tv_1|\simeq2.4\,10^{-5}$ and
$(u_1^T\xi)^2\simeq1.2\,10^{-20}$, despite the nonzero extent.

For $a=10^{-4}$, substantial critical variance does develop, and Rod's
effective sharpness returns close to $2$. At the last diagnostic sample, its
critical variance $(u_1^T\xi)^2$ is $2.60\,10^{-7}$, compared with $1.99\,10^{-7}$ for GD.
However, the variance grows and the loss peaks much earlier than in GD.
This stronger floor also changes the initial extent: all three retained
eigenvalues become $a^2$, so the largest is initially multiple.

\begin{figure}[ht]
    \centering
    \includegraphics[width=\textwidth]{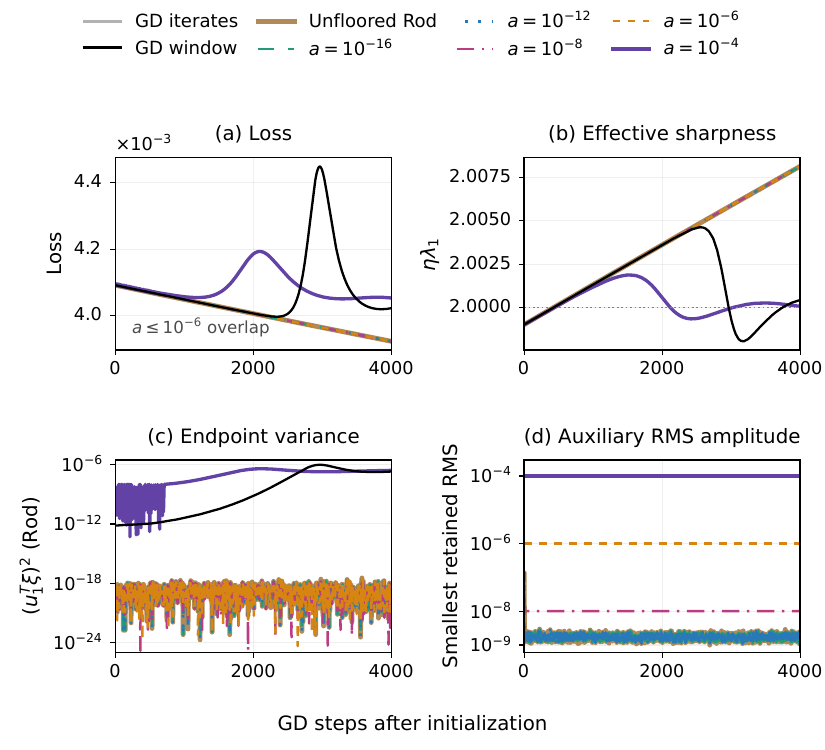}
    \caption{Floored Rod flows at $\eta=0.002$, from the
    same cached GD window as \cref{fig:rod-onset-floor}, with $\Sigma$ of rank three, $K=4$ and $4000$ steps. The floor acts at initialization
    and throughout the evolution. The tested floors up to $10^{-6}$ leave the onset
    failure unchanged; $10^{-4}$ produces critical fluctuations, but
    earlier than in GD.}
    \label{fig:rod-persistent-floors}
\end{figure}

\paragraph{Critical fluctuations are discarded by Rod and amplified by MF.}

The experiment of \cref{fig:rod-onset-floor} shows that the top Hessian direction $u_1$ and Rod's retained direction $\xi$ are nearly orthogonal at the onset of EoS (see also \cref{fig:rod-onset-geometry}).

This may be the reason why Rod flow fails to develop oscillations, as explained by \cite{marion2026edge}.
In contrast, MF can still increase the small oscillations aligned with $u_1$, because all directions of $\operatorname*{Ran}\Sigma$ are retained for its evolution.
We also note that Edge flow does not suffer from this failure mode because the oscillation direction is \emph{defined} to align with the top Hessian direction through a Rayleigh-quotient flow.

To fix ideas, freeze the Hessian $H$ and let $(\lambda_1,u_1)$ denote its top eigenpair: $Hu_1=\lambda_1 u_1$ with $\|u_1\|=1$.
Suppose that
\begin{equation*}
    \Sigma = s_{\rm drift} vv^T + s_{\rm crit} u_1u_1^T, \quad \| v\|=1, \quad v\perp u_1, \quad s_{\rm drift} > s_{\rm crit} >0.
\end{equation*}
Rod flow selects $\xi = \sqrt{s_{\rm drift}} v$. Therefore, linearizing the evolution around this initial condition gives
\begin{align*}
    \frac{ds_{\rm crit}}{dt} \bigg\vert_{\rm Rod} &\approx \frac{\eta^2}{2} |u_1^T \nabla E(\theta) |^2  - 2 s_{\rm crit}, \\
    &\approx   - 2 s_{\rm crit},
\end{align*}
assuming $\eta^2|u_1^T\nabla E(\theta)|^2\ll 4s_{\rm crit}$.
As noted by \citet{marion2026edge}, the critical variance $s_{\rm crit}$ therefore cannot grow under these assumptions.

In contrast, the mean--fluctuation evolution gives
\begin{align*}
     \frac{ds_{\rm crit}}{dt} \bigg\vert_{\rm MF} &\approx - 2\eta \lambda_1 s_{\rm crit}   +\eta^2 \lambda_1^2 s_{\rm crit} \\
     &= \eta \lambda_1 (-2 + \eta \lambda_1) s_{\rm crit}.
\end{align*}
There is therefore no need for a rotation of $v$ to allow for growing oscillations in $s_{\rm crit}$: as soon as $\eta \lambda_1 > 2$, the critical variance grows exponentially.

MF can therefore amplify critical fluctuations before they become
dominant; the linearized Rod forcing does not amplify them when its
selected direction is orthogonal to $u_1$.
We finally note that this is a frozen-Hessian explanation: the directions can rotate along the nonlinear dynamics, so that we have not provided the full mechanism. 

We do not claim that MF is totally immune to this failure mode: for fixed $H$, an
initially zero critical variance remains zero, and strong subcritical
damping can delay its recovery.

\FloatBarrier

\section{Additional experimental details}
\label[appendix]{app:experiments}

\subsection{Experimental protocol and summary}

All experiments use full-batch GD on subsets of CIFAR-10,
mean-squared loss, and no data augmentation. The MLP, CNN and ResNet have
widths $256$, $32$ and $16$, respectively, with GELU activations; the ResNet
has three stages of three residual blocks. Gradients may be accumulated
over smaller batches for memory efficiency, without stochastic sampling.
The six main comparisons use seed $0$ and the settings in
\cref{tab:experiment-summary}.

For these six comparisons, coarse MF uses covariance rank cap
three, Krylov depth two and four substeps per GD step (the solver's parameters are described in \cref{app:solver}). 
Rod uses auxiliary rank three and four substeps; Central uses four substeps, and Edge uses
Algorithm~1 of \cite{marion2026edge} with $K=D=4$. The initial MF covariance and Rod auxiliary extent
have rank three for MLP, one/two for CNN, and three/two for ResNet in the
one-/two-mode cases. No amplitude floor is added to MF or Rod; Edge retains
its prescribed amplitude rule. Arithmetic precision and solver-refinement
checks are detailed in \cref{app:solver}, including the full displayed
horizons of the two-mode comparisons.

For reproducibility, the accompanying code provides Python
scripts that prepare initial states and launch the simulations. Saved configuration
files record the parameter choices. Initialization checkpoints are not included
in the supplement; when rerunning a comparison, the same prepared state must be
used for all flows and numerical resolutions. A separate record links each
plotted curve to its simulation outputs and source-code version.

\begin{table}[ht]
    \centering
    \small
    \setlength{\tabcolsep}{3pt}
    \begin{tabular}{@{}lrrrrrrrcrl@{}}
        \toprule
        Experiment & Modes & Seeds & $n$ & $\eta$ & $m$ & $r$ & $p$ & $K$ & Iterates & Figures \\
        \midrule
        MLP & 1 & 1 & 8000 & .002 & 16 & 3 & 2 & 4 & 4000 & \ref{fig:intro-one-mode-loss}, \ref{fig:mlp-eos} \\
        MLP & 2 & 1 & 8000 & .002 & 16 & 3 & 2 & 4 & 4000 & \ref{fig:mlp-eos} \\
        CNN & 1 & 1 & 8000 & .002 & 4 & {3} & {2} & 4 & 1500 & \ref{fig:intro-one-mode-loss}, \ref{fig:architecture-eig1} \\
        CNN & 2 & 1 & 8000 & .002 & 4 & {3} & 2 & 4 & 2000 & \ref{fig:architecture-eig12}, \ref{fig:cnn-closure-accumulation}, \ref{fig:cnn-mf-central-relaxation} \\
        ResNet & 1 & 1 & 3000 & .004 & 8 & 3 & {2} & {4} & 411 & \ref{fig:intro-one-mode-loss}, \ref{fig:architecture-eig1} \\
        ResNet & 2 & 1 & 3000 & .004 & 8 & {3} & 2 & 4 & 200 & \ref{fig:architecture-eig12} \\
        \midrule
        MLP seeds & 1 & 10 & 8000 & .002 & 16 & 3 & 2 & 4 & 4000 & \ref{fig:seed-ensemble}, \ref{fig:seed-energy-errors} \\
        MLP seeds & 2 & 9 & 8000 & .002 & 16 & 3 & 2 & 4 & 4000 & \ref{fig:seed-ensemble}, \ref{fig:seed-energy-errors} \\
        MLP rates & 1 & 1 & 8000 & .0025 & 16 & 3 & 2 & 4 & 1000 & \ref{fig:mlp-robustness} \\
        MLP rates & 1 & 1 & 8000 & .004 & 16 & 3 & 2 & 4 & 1000 & \ref{fig:mlp-robustness} \\
        MLP rates & 2 & 1 & 8000 & .0025 & 16 & 3 & 2 & 4 & 1000 & \ref{fig:mlp-robustness} \\
        MLP rates & 2 & 1 & 8000 & .004 & 16 & 3 & 2 & 4 & 1000 & \ref{fig:mlp-robustness} \\
        \midrule
        MLP onset & 1 & 1 & 8000 & .002 & 16 & 3 & 2 & 4 & 4000 & \ref{fig:rod-onset-floor} \\
        CNN sweep & 1 & 1 & 8000 & .002 & 4 & 1 & -- & $2,4,8,16$ & 200 & \ref{fig:cnn-rod-substeps} \\
        \bottomrule
    \end{tabular}
    \caption{Summary of the experiments. The Modes column counts near-EoS Hessian
    modes {at initialization}, independently of the retained covariance rank $r$. Here $n$ is
    the sample count, $m$ the averaging-window length, $p$ the MF Krylov
    depth, and $K$ the MF/Rod substeps per GD step (Rod only in the sweep).
    Seeds counts completed trajectories. The seed cohorts attempt
    seeds $0,\ldots,9$; two-mode seed $4$ does not reach its initialization gate
    and is not replaced. All other rows use seed $0$.
    Iterates gives the retained baseline comparison horizon per seed.
    Matched coarse-MF refinements for the two-mode comparisons are listed
    separately in \cref{tab:mf-network-refinement}.
    The MLP onset row
    gives the unfloored baseline; \cref{fig:rod-onset-floor} also shows its
    rank-four, eight-substep refinement, with MF Krylov depth three.}
    \label{tab:experiment-summary}
\end{table}

The one-mode comparisons start after the first instability, when
$\tfrac12 H:\Sigma/E$ reaches approximately $10^{-3}$, with $H$ and $E$
evaluated at the right endpoint of the stored GD window and $\Sigma$ its
empirical covariance. The two-mode comparisons continue the underlying GD
trajectory until the second Hessian mode reaches EoS. Within each comparison,
all flows share the stored initialization window. The separate MLP onset
control from \cref{app:onset-failure} starts just below EoS, at $\eta\lambda_1\simeq1.999$, without the
additional fluctuation-size condition; its baseline and refined
MF/Rod runs have no amplitude floor. The CNN sweep reuses
the one-mode CNN initialization.

GD Hessian eigenpairs are computed by warm-started LOBPCG every five steps
for the MLP, every $100$ for the one-mode CNN/ResNet, and every ten for their
two-mode runs. Independent GD-window diagnostics are saved every five steps in all six cases, with both GD-center eigenvalues in every two-mode comparison. 
The covariance-closure tests in \cref{app:architecture-closure}
compare the covariance measured from GD with predictions obtained by evolving
the covariance equations along the GD-window means. The predicted covariance
and Hessian coefficients are updated every GD step; the Hessian basis used
for projection is refreshed every five steps.
The CNN Rod sweep evaluates sharpness every $20$ steps. Instantaneous GD
sharpness and sharpness at the window mean are distinct observables; flow
sharpness is evaluated at the flow center.

\Cref{fig:intro-one-mode-loss} reuses the three one-mode runs, while
\cref{fig:cnn-closure-accumulation,fig:cnn-mf-central-relaxation} reuse the
two-mode CNN trajectory. 
Some figures combine flows computed in separate runs, each
with a GD reference restarted from the same saved state. Each flow is
compared with the GD reference computed in its own run, so we need to check reproducibility between GD runs. 
The repeated GD loss records are identical for MLP; for the long CNN, their maximum relative
loss difference from the displayed GD references (absolute error normalized by the reference loss) is $0.00497\%$; the corresponding projected empirical
GD covariance traces differ by at most
$0.232\%$. For \cref{fig:cnn-mf-central-relaxation}, GD, MF and
Central were computed together in one run, with the two-mode CNN settings
in \cref{tab:experiment-summary}. Both flows use the same GD reference, and
their mutual distance is computed directly from their saved parameter
vectors.

\subsection{Center error and oscillatory-state distance}
\label{app:distance-observables}

The relative center error used in the main figures,
\begin{equation}\label{eq:relative-center-error}
    e_\theta(p)=\frac{\|\theta_p-\overline\theta_{\rm GD}\|_2}
    {\|\overline\theta_{\rm GD}\|_2},
\end{equation}
compares the predicted mean with the GD-window mean in the full parameter
space. It does not compare their fluctuations. Moreover, its denominator
is the total parameter norm, which can be much larger than the oscillation
amplitude. A small value of $e_\theta$ therefore does not by itself establish
accurate fluctuation dynamics.

We complement it with a comparison of means and covariances in the
leading Hessian directions. At each sampled time, let $U$ contain the
leading one or two orthonormal Hessian eigenvectors at
$\overline\theta_{\rm GD}$, according to the one- or two-mode comparison.
For a flow $p$ with mean $\theta_p$ and covariance $\Sigma_p$, set
\[
    m_p=U^T(\theta_p-\overline\theta_{\rm GD}),
    \qquad X_p=U^T\Sigma_pU.
\]
For Rod, $\Sigma_p$ is the covariance of its two endpoint positions, not
its auxiliary covariance. GD is represented by $m_{\rm GD}=0$ and
$X_{\rm GD}=U^T\Sigma_{\rm GD}U$, where $\Sigma_{\rm GD}$ is the empirical
covariance of the iterates in the same window. All states in a comparison
use the same basis and GD-window mean.

To compare positive semidefinite covariance matrices, we use the Bures
distance
\[
    d_{\rm B}^2(X,Y)=\operatorname{tr}X+\operatorname{tr}Y
    -2\operatorname{tr}\!\left[(X^{1/2}YX^{1/2})^{1/2}\right],
\]
where matrix square roots are positive semidefinite. Our normalized
state discrepancy is
\begin{equation}\label{eq:eos-state-error}
    \mathcal D^2(p,q)
    =\frac{2\bigl(\|m_p-m_q\|^2+d_{\rm B}^2(X_p,X_q)\bigr)}
    {\|m_p\|^2+\operatorname{tr}X_p+\|m_q\|^2+\operatorname{tr}X_q}.
\end{equation}
For a positive denominator, this compares the mean and covariance
differences with the average second moment of the two projected states
about the GD-window mean, rather than with the full parameter norm.
It does not detect errors orthogonal to the selected Hessian subspace,
so it complements rather than replaces $e_\theta$.

\Cref{fig:state-distance-comparison} uses the same trajectories and sampled
times as \cref{fig:mlp-eos,fig:architecture-eig12}. At the last two-mode
MLP sample, MF and Rod have nearly identical center errors, about
$0.00043\%$, but their discrepancies $\mathcal D$ from GD are $0.0263$
and $0.00113$, respectively. Thus agreement of the means alone does not
characterize agreement of the mean--covariance states.

\begin{figure}[ht!]
    \centering
    \includegraphics[width=\textwidth]{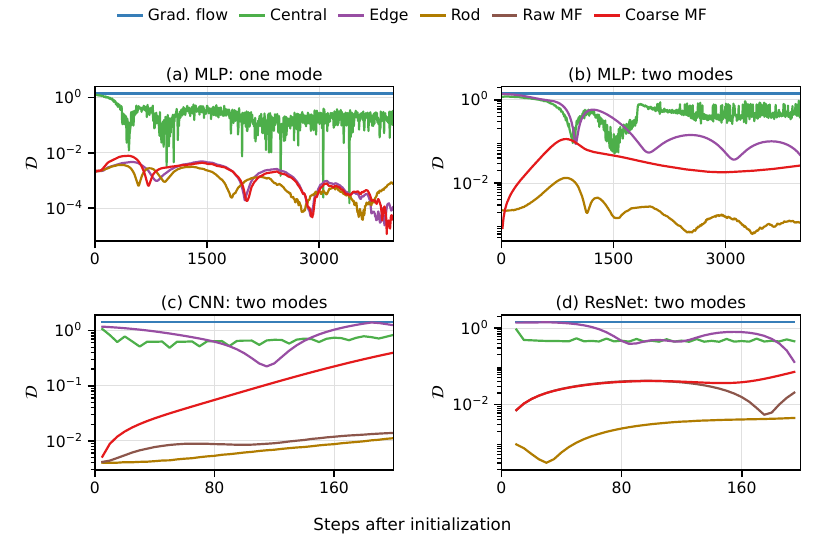}
    \caption{Normalized mean--covariance discrepancy
    $\mathcal D(p,\mathrm{GD})$ for the same runs as the main figures.
    Top: one- and two-mode MLP. Bottom: two-mode CNN and ResNet. Each
    comparison uses the leading one or two Hessian directions at the
    GD-window mean. The trajectories, horizons and sampled times are
    unchanged. Gradient flow has zero covariance, giving
    $\mathcal D=\sqrt2$.}
    \label{fig:state-distance-comparison}
\end{figure}

\paragraph{MF--Central comparison.}
\label{app:cnn-mf-central}
For the shared $2000$-step CNN run in \cref{fig:cnn-mf-central-relaxation},
the center-to-center curve shows
$\|\theta_{\rm MF}-\theta_{\rm Central}\|_2/\|\theta_{\rm Central}\|_2$
at the $41$ saved checkpoints. To complement this comparison, we evaluate
\eqref{eq:eos-state-error} in the same two GD-center Hessian directions
for both flows. The discrepancy
$\mathcal D(\mathrm{MF},\mathrm{Central})$ decreases from $0.457$ at
step $195$ to $0.00117$ at step $1995$, while their final discrepancies
from GD are both about $0.697$ (\cref{fig:cnn-mf-central-state}).

\begin{figure}[ht!]
    \centering
    \includegraphics[width=\textwidth]{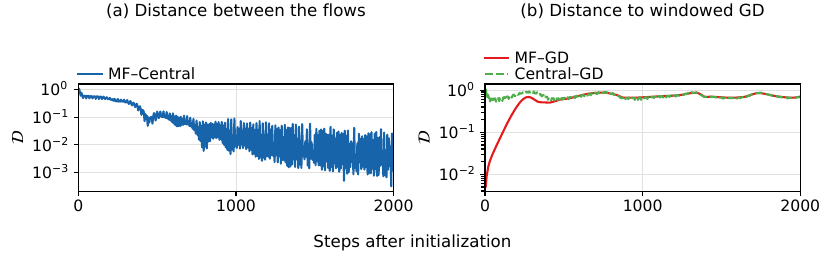}
    \caption{Mean--covariance comparison for the same
    $2000$-step CNN run as \cref{fig:cnn-mf-central-relaxation}, using
    $\mathcal D$ from \eqref{eq:eos-state-error}. Left: discrepancy
    between MF and Central flow. Right: discrepancy of each flow from
    windowed GD. All comparisons use the same two Hessian directions
    at the GD-window mean and all $399$ recorded diagnostic samples.}
    \label{fig:cnn-mf-central-state}
\end{figure}

\paragraph{The constant value for gradient flow.}
If $X_p=0$, then $d_{\rm B}^2(X_p,X_{\rm GD})=\operatorname{tr}X_{\rm GD}$.
Since $m_{\rm GD}=0$, \eqref{eq:eos-state-error} gives
\[
    \mathcal D^2(p,\mathrm{GD})
    =\frac{2\bigl(\|m_p\|^2+\operatorname{tr}X_{\rm GD}\bigr)}
    {\|m_p\|^2+\operatorname{tr}X_{\rm GD}}=2,
\]
whenever the denominator is positive, in particular when $X_{\rm GD}\ne0$.
This applies to gradient flow, whose assigned covariance is zero.

\FloatBarrier

\subsection{One-mode CNN and ResNet experiments}
\label{app:architecture-eig1}

\Cref{fig:architecture-eig1} complements the two-mode comparisons with
one-mode CNN and ResNet runs. The available trajectories contain $1500$
and $411$ iterates, respectively. 
Both use the baseline of \cref{tab:experiment-summary}. 

In both one-mode runs, the coarse-MF free energy and sharpness follow the GD loss
and the sharpness at the GD-window mean, respectively. Writing $u_1$ for the
leading Hessian eigenvector at the GD-window mean, the predicted
variance $u_1^T\Sigma_{\rm MF}u_1$ also follows the corresponding GD variance.

\begin{figure}[t]
    \centering
    \includegraphics[width=\textwidth]{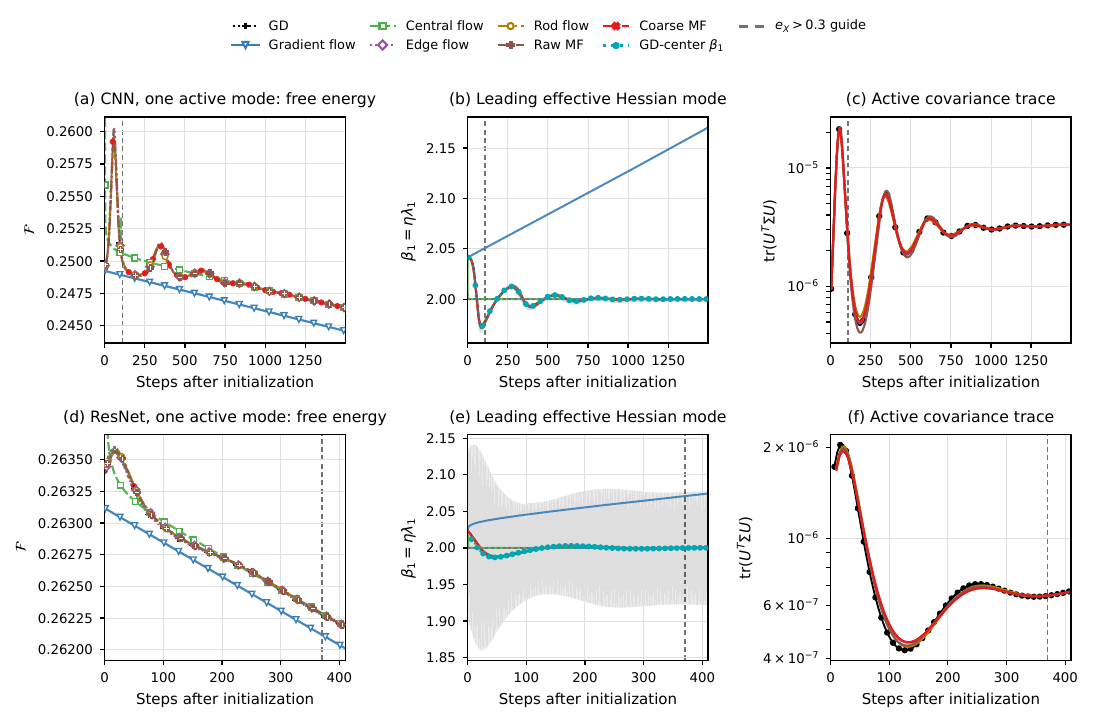}
    \caption{Additional one-mode architecture experiments.  Top: CNN with
    $\eta=0.002$.  Bottom: ResNet with $\eta=0.004$.
    Columns show free energy, the leading effective Hessian
    eigenvalue, and the variance $u_1^T\Sigma u_1$, where $u_1$ is the
    leading Hessian eigenvector at the GD-window mean. For the Rod curve,
    $\Sigma$ is the covariance of its two endpoint positions.
    Instantaneous GD sharpness is drawn lightly; cyan markers show sharpness
    at the window center, and black markers show the
    empirical GD-window variance along $u_1$.
    The vertical guides mark the first
    sampled forced-covariance error above $0.3$.}
    \label{fig:architecture-eig1}
\end{figure}
\FloatBarrier

\subsection{Closure diagnostics for the six architecture/onset experiments}
\label{app:architecture-closure}

In this section, we numerically test our closure hypothesis, especially for the covariance evolution.
Specifically, we test whether the covariance evolution, driven by the empirical GD Hessian, reproduces the GD fluctuations. 
Similarly, we evolve the mean equation using the empirical GD covariance projected onto the leading Hessian directions.

We report these diagnostics for the six one- and two-mode MLP, CNN and ResNet runs in \cref{tab:experiment-summary}.

\paragraph{Closure diagnostics procedure.}
Let $\tilde\theta_t$ denote the GD iterates. The mean and covariance of
the window of $m$ consecutive iterates ending at step $t$ are
\[
    \begin{aligned}
        \overline\theta_{{\rm GD},t}
        &=\frac1m\sum_{j=0}^{m-1}\tilde\theta_{t-j},\\
        \widehat\Sigma_t
        &=\frac1m\sum_{j=0}^{m-1}
        (\tilde\theta_{t-j}-\overline\theta_{{\rm GD},t})
        (\tilde\theta_{t-j}-\overline\theta_{{\rm GD},t})^T.
    \end{aligned}
\]
These empirical quantities are measured from GD, not predicted by MF; the window
lengths are given in \cref{tab:experiment-summary}.
Let $U_t=(u_1(t),\ldots,u_q(t))$ have orthonormal
columns approximating the eigenvectors associated with the $q$ largest
eigenvalues of $H(\overline\theta_{{\rm GD},t})$. We use the empirical projected quantities
\[
    \widehat X_t=U_t^T\widehat\Sigma_tU_t,
    \qquad B_t=U_t^TH(\overline\theta_{{\rm GD},t})U_t.
\]
All six runs use the same diagnostic dimension $q=3$. 
Note that the basis is constructed from
GD alone. A Hessian eigensolve refreshes the tracked subspace every five
steps; between solves, we diagonalize the current Hessian within that
subspace.

\paragraph{Accumulated covariance errors.}
We evolve two covariance controls along the prescribed GD-window means which we now describe.
The continuous control starts from the projected empirical covariance at initialization and integrates
\[
    \frac{dX}{ds}=-\eta(B_tX+XB_t)+\eta^2B_tXB_t,
    \qquad s\in[t,t+1],
\]
exactly with $B_t$ held fixed over the step. Before the next step,
the covariance is transferred to the updated basis by
$X\leftarrow(U_{t+1}^TU_t)X(U_t^TU_{t+1})$.
The discrete control starts from a rank-truncated factor $Q$ of the
same initial empirical covariance and evolves
\[
    Q_{t+1}=(\Id-\eta H(\overline\theta_{{\rm GD},t}))Q_t
\]
in the full parameter space. Its initial factor rank is three for MLP
and one-mode ResNet, one for one-mode CNN, and two for two-mode CNN/ResNet.
We compare its projection $U_t^TQ_tQ_t^TU_t$ with $\widehat X_t$.
These controls differ in covariance representation, initialization and basis transfers, so their comparison does not isolate the effect of replacing discrete evolution by continuous time.

Neither control is reset to the empirical covariance at later windows so that
the resulting errors accumulate over time. For either    
predicted covariance $X_{\rm pred}$, \cref{fig:closure-covariance-all}
shows
\[
    e_X=\frac{\|\widehat X-X_{\rm pred}\|_F}
    {\max\{\|\widehat X\|_F,\epsilon_X\}}.
\]
To avoid division by nearly zero values, we set the fixed floor
$\epsilon_X$ to one thousandth of the median covariance
norm $\|\widehat X_t\|_F$, with a minimum of $10^{-12}$.

The continuous
crossings of $e_X= 0.3$ define the vertical guides in
\cref{fig:architecture-eig12,fig:architecture-eig1}.

\begin{figure}[ht!]
    \centering
    \includegraphics[width=\textwidth]{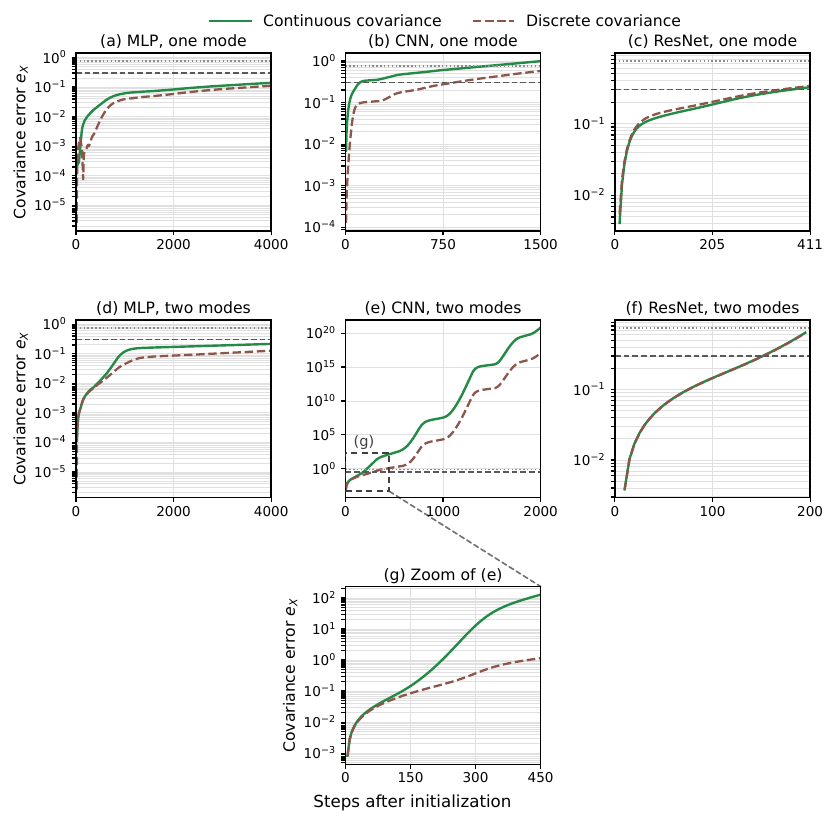}
    \caption{Accumulated covariance errors $e_X$ along the
    prescribed GD-window means. Panels (a--c): one-mode initializations;
    (d--f): two-mode initializations. Columns: MLP, CNN and ResNet. Solid green:
    continuous MF covariance equation in the $q=3$ diagnostic subspace.
    Dashed brown: discrete linearized covariance evolution in the full
    parameter space, projected only for comparison. Both controls are
    initialized once, then evolved without resetting to GD. Values are
    recorded every five steps; horizontal lines in (a--f) mark $e_X=0.30,0.75$.
    }
    \label{fig:closure-covariance-all}
\end{figure}

\paragraph{Accumulated mean errors.}

We test the mean equation similarly, using the projected empirical covariance.
To this end, we evolve a forced mean $\theta^{\rm f}$, initialized from the GD-window mean. 
Over each GD interval, it solves
\begin{equation}
    \frac{d\theta^{\rm f}}{ds}
    =-\eta\nabla E(\theta^{\rm f})
     -\frac\eta2\nabla H(\theta^{\rm f})
       :(U_t\widehat X_tU_t^T),
    \qquad s\in[t,t+1].
    \label{eq:closure-forced-mean}
\end{equation}
Only the empirical covariance is held fixed over the interval. The covariance is refreshed at every GD step,
but the forced mean is never reset, so that we measure accumulated mean errors. 
We use four Euler substeps per GD step,
float64 accumulation of the mean updates, and float32 network and
derivative evaluations. 

The accumulated error is
\begin{equation}
    e_\theta^{\rm forced}(t)
    =\frac{\|\theta^{\rm f}(t)-\overline\theta_{{\rm GD},t}\|_2}
    {\max\{\|\overline\theta_{{\rm GD},t}\|_2,10^{-12}\}}.
    \label{eq:closure-forced-mean-error}
\end{equation}

\Cref{fig:closure-mean-completed} reports this control for all six runs,
using their original initialization caches: $4000$ steps for each MLP
run, $1500/2000$ steps for one-/two-mode CNN, and $411/200$ for ResNet.
The sampled relative error stays below $1.02\times10^{-5}$ throughout.
In the two-mode CNN and ResNet runs, its maxima are
$3.85\times10^{-6}$ and $1.73\times10^{-6}$, respectively, even though
the corresponding covariance controls develop large errors.

\begin{figure}[ht]
    \centering
    \includegraphics[width=\textwidth]{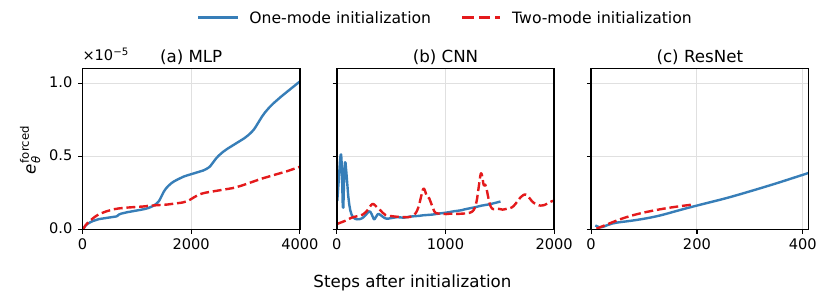}
    \caption{Accumulated mean errors
    \eqref{eq:closure-forced-mean-error}, with
    projected empirical GD covariance forcing and $q=3$.
    Curves end at their respective run horizons; the two-mode ResNet run covers
    $200$ steps, compared with $411$ for the one-mode run.}
    \label{fig:closure-mean-completed}
\end{figure}

\paragraph{Local mismatches of the mean and covariance equations.}

{Beyond accumulated errors, the recorded diagnostics compare backward window increments with the current mean and covariance fields.}
Writing $\widehat\theta_t=\overline\theta_{{\rm GD},t}$, the MF predictions are
\[
    \begin{aligned}
        G_{\theta,t}
        &=-\eta\nabla E(\widehat\theta_t)
          -\frac{\eta}{2}\nabla H(\widehat\theta_t)
          :(U_t\widehat X_tU_t^T),\\
        G_{X,t}
        &=-\eta(B_t\widehat X_t+\widehat X_tB_t)
          +\eta^2B_t\widehat X_tB_t.
    \end{aligned}
\]
We compare them with the changes over one GD step,
\[
    \delta_t\widehat\theta
    =\widehat\theta_t-\widehat\theta_{t-1},
    \qquad
    \delta_t\widehat X
    =U_t^T(\widehat\Sigma_t-\widehat\Sigma_{t-1})U_t.
\]
Both empirical covariances are projected onto the same current basis
before subtraction, so basis rotations alone do not contribute to the
residual. For $z\in\{\theta,X\}$, define
\begin{equation}
    R_z(t)=\frac{\|\delta_t\widehat z-G_{z,t}\|}
    {\max\{\|\delta_t\widehat z\|,\|G_{z,t}\|,\epsilon_z\}},
    \label{eq:closure-vector-field-residual}
\end{equation}
using the Euclidean norm for the mean, the Frobenius norm for the
covariance, and the same floor rule applied to
$\max\{\|\delta_t\widehat z\|,\|G_{z,t}\|\}$.
\Cref{fig:closure-vector-fields-all} records these residuals every five
steps, but each difference uses adjacent windows, not windows five
steps apart. The mean residual is measured in the full parameter space, while the covariance is projected. 

A more natural comparison would have been to record the defect with respect to the \emph{forward} changes $\delta_t \wh\theta = \wh\theta_{t+1} - \wh \theta_{t}$. 
The available logs use backward rather than forward differences; the two conventions have not been compared here. 

\begin{figure}[t]
    \centering
    \includegraphics[width=\textwidth]{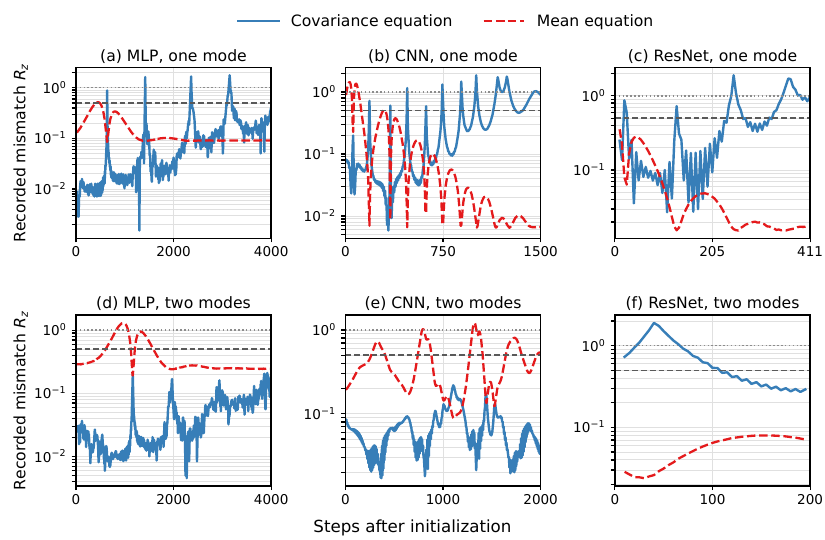}
    \caption{Local residual
    \eqref{eq:closure-vector-field-residual} for the same six runs and
    $q=3$ diagnostic spaces, arranged as panels (a--f) of
    \cref{fig:closure-covariance-all}.
    Solid blue: covariance
    equation. Dashed red: mean equation. Horizontal lines mark $R_z=0.50,1.00$.}
    \label{fig:closure-vector-fields-all}
\end{figure}
\FloatBarrier

\begin{figure}[ht]
    \centering
    \includegraphics[width=\textwidth]{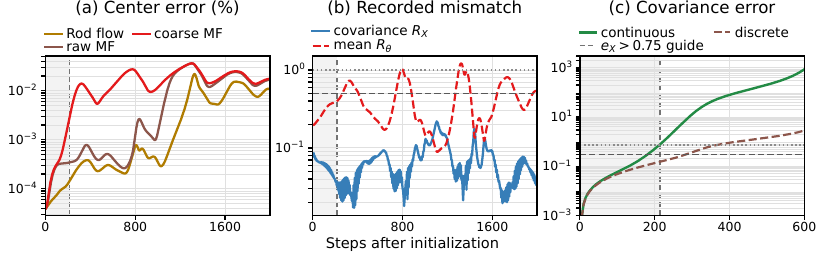}
    \caption{Two-mode CNN: (a) relative center errors of the
    autonomous flows with respect to gradient descent; (b) local residuals $R_X,R_\theta$;
    (c) accumulated covariance errors for the continuous and discrete
    controls, over the first $600$ steps. Vertical guides mark the first
    sampled crossing of $e_X=0.75$ by the continuous control; horizontal
    guides show the reference levels specified above.}
    \label{fig:cnn-closure-accumulation}
\end{figure}

\FloatBarrier

\subsection{MLP sensitivity to seed and learning rate}
\label{app:mlp-robustness}

We vary the random seed at $\eta=0.002$ and the learning rate at seed $0$,
using both one- and two-mode initializations. The MLP width, sample count
and averaging window remain $256$, $8000$ and $m=16$.
For each seed and rate, GD is continued to the corresponding first- or
second-eigenvalue threshold, following \cref{tab:experiment-summary};
all compared flows then use the same sampled GD window.
The horizons are $4000$ steps for the seed cohort and $1000$ for the
learning-rate comparisons. MF denotes the coarse dynamics, with baseline
$(r,p,K)=(3,2,4)$. The other solver settings are unchanged.
The available refinement checks are summarized in
\cref{tab:mf-network-refinement}.

\paragraph{Seed cohort.}
\label{app:seed-ensemble}
We attempt seeds $0,\ldots,9$ at $\eta=0.002$. All ten one-mode runs
and nine two-mode runs complete. Two-mode seed $4$ does not reach the
second-eigenvalue threshold within $500000$ additional GD steps and is
not replaced. All completed runs are retained. The largest sampled next
effective GD eigenvalue is $1.9422$ in the one-mode cohort and $1.9264$ in the
two-mode cohort; these checks exclude an additional unstable mode only
at the sampled times.

\Cref{fig:seed-ensemble} compares the mean loss/free energy and leading
eigenvalues of GD, MF and Rod. Each seed is aligned at its own
initialization. We compute the relative energy error
$e_F=|F-\overline E_{\rm GD}|/|\overline E_{\rm GD}|$
against that seed's window-averaged GD loss before taking means or
medians. \Cref{fig:seed-energy-errors} shows the individual errors as
well as both summaries, and \cref{tab:mlp-observables} reports cohort
medians for MF.

At the last common sample, the one-mode median energy errors for MF,
raw MF and Rod are $0.00277\%,0.0138\%,0.0102\%$, respectively.
The corresponding two-mode medians are $0.151\%,0.0581\%,0.0313\%$,
whereas the means reach $3.27\%,5.02\%,0.242\%$.

\paragraph{Learning-rate sensitivity.}
\Cref{fig:mlp-robustness} compares seed $0$ at
$\eta=0.0025,0.004$, including Edge, Central and gradient flow.
For one-mode initializations, the sampled MF energy error stays below
$0.224\%$ and $0.73\%$, respectively. For two-mode initializations,
these bounds become $0.068\%$ and $5.42\%$.
The next effective GD-window eigenvalue remains below $2$ at the sampled times:
its maxima are below $1.803,1.839$ for one mode and $1.978,1.994$ for two modes.
The deterioration at the larger rate therefore does not require an
additional unstable mode in the recorded spectrum.

\paragraph{Dependence on the observable.}
\label{app:mlp-observable-choice}
Small center errors need not imply accurate fluctuations.
For the two-mode $\eta=0.004$ run, the final relative center error
$e_\theta$ is $0.0228\%$, whereas the mean--covariance discrepancy
$\mathcal D$ is $0.760$ (\cref{app:distance-observables}).
The former divides by the full parameter norm; the latter compares
means and covariances on the oscillation scale in the tracked Hessian
directions.

\begin{table}[ht]
    \centering
    \small
    {
    \begin{tabular}{@{}rrrrrrrr@{}}
        \toprule
        Modes & $\eta$ & Seeds & Steps & $\max e_F$ (\%) & Final $e_F$ (\%) &
        Final $e_\theta$ (\%) & Final $\mathcal D$ \\
        \midrule
        1 & .002 & 10 & 4000 & .249 & .00277 & .000511 & .000337 \\
        2 & .002 & 9 & 4000 & .757 & .151 & .00182 & .0263 \\
        \midrule
        1 & .0025 & 1 & 1000 & .224 & .00431 & .000133 & .000845 \\
        1 & .004 & 1 & 1000 & .729 & .00603 & .000201 & .00200 \\
        2 & .0025 & 1 & 1000 & .0680 & .00672 & .000274 & .00190 \\
        2 & .004 & 1 & 1000 & 5.41 & 2.51 & .0228 & .760 \\
        \bottomrule
    \end{tabular}}
    \caption{Accuracy of the baseline MF trajectories:
    free-energy error $e_F$, relative center error $e_\theta$
    \eqref{eq:relative-center-error}, and mean--covariance discrepancy
    $\mathcal D$ \eqref{eq:eos-state-error}.
    Modes refers to the initialization Hessian spectrum; Seeds is the
    number of completed trajectories. The first two rows give medians
    across seeds: each error is first computed within its own trajectory,
    and $\max e_F$ is maximized in time before taking the median.
    The remaining rows use seed $0$ only. }
    \label{tab:mlp-observables}
\end{table}

\begin{figure}[p]
    \centering
    \includegraphics[width=\textwidth]{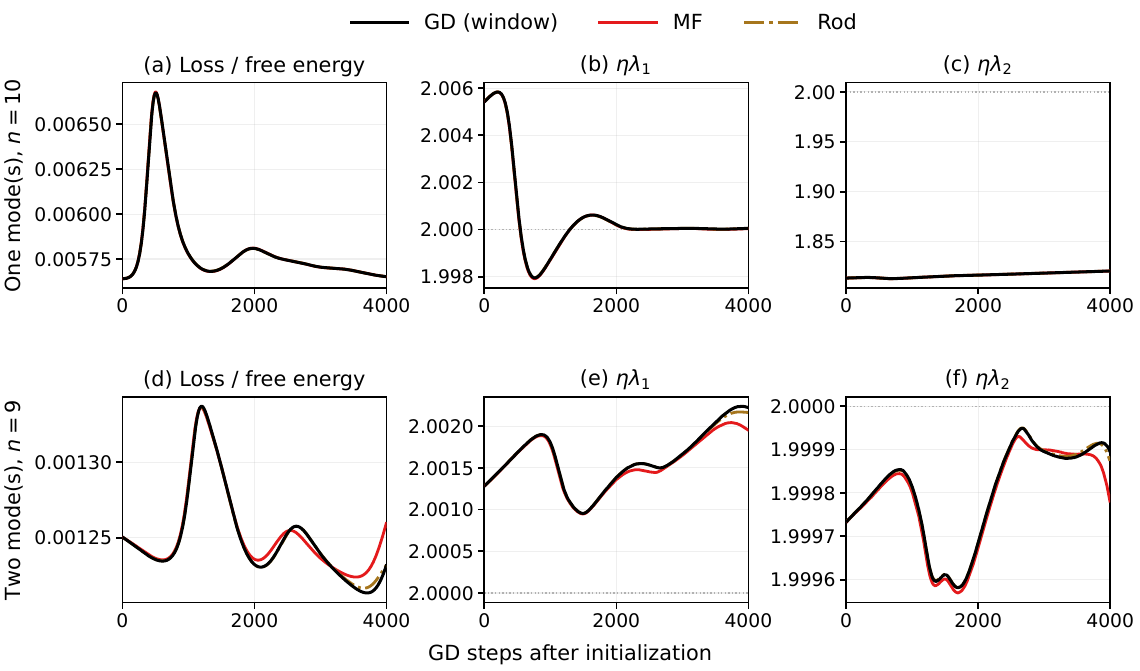}
    \caption{MLP seed cohort at $\eta=0.002$: one-mode ($10$ seeds, top) and two-mode ($9$ seeds, bottom) initializations. Columns show averaged loss/free energy and the first two effective Hessian eigenvalues for GD, MF and Rod. }
    \label{fig:seed-ensemble}
\end{figure}

\begin{figure}[p]
    \centering
    \includegraphics[width=\textwidth]{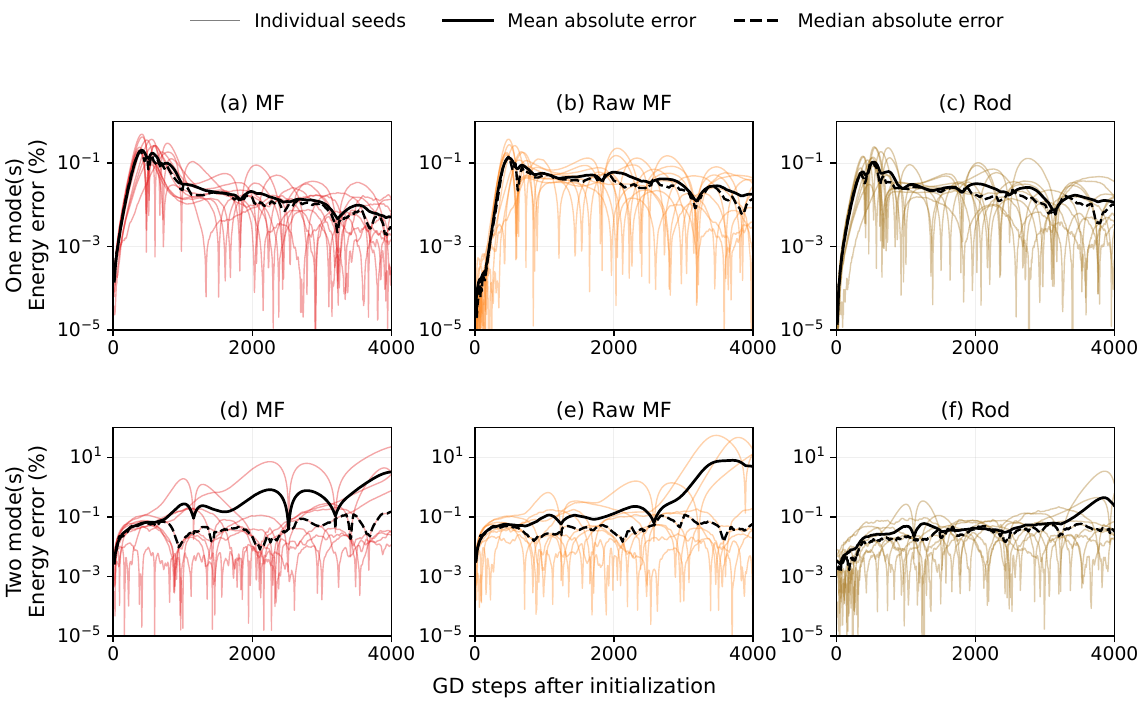}
    \caption{{Per-seed absolute relative energy errors (thin colored lines), with their mean (solid black) and median (dashed black). Errors are computed against each seed's window-averaged GD loss before aggregation. Two-mode seed $4$ has no comparison trajectory because its initialization gate was not reached.}}
    \label{fig:seed-energy-errors}
\end{figure}

\begin{figure}[p]
    \centering
    \includegraphics[width=\textwidth]{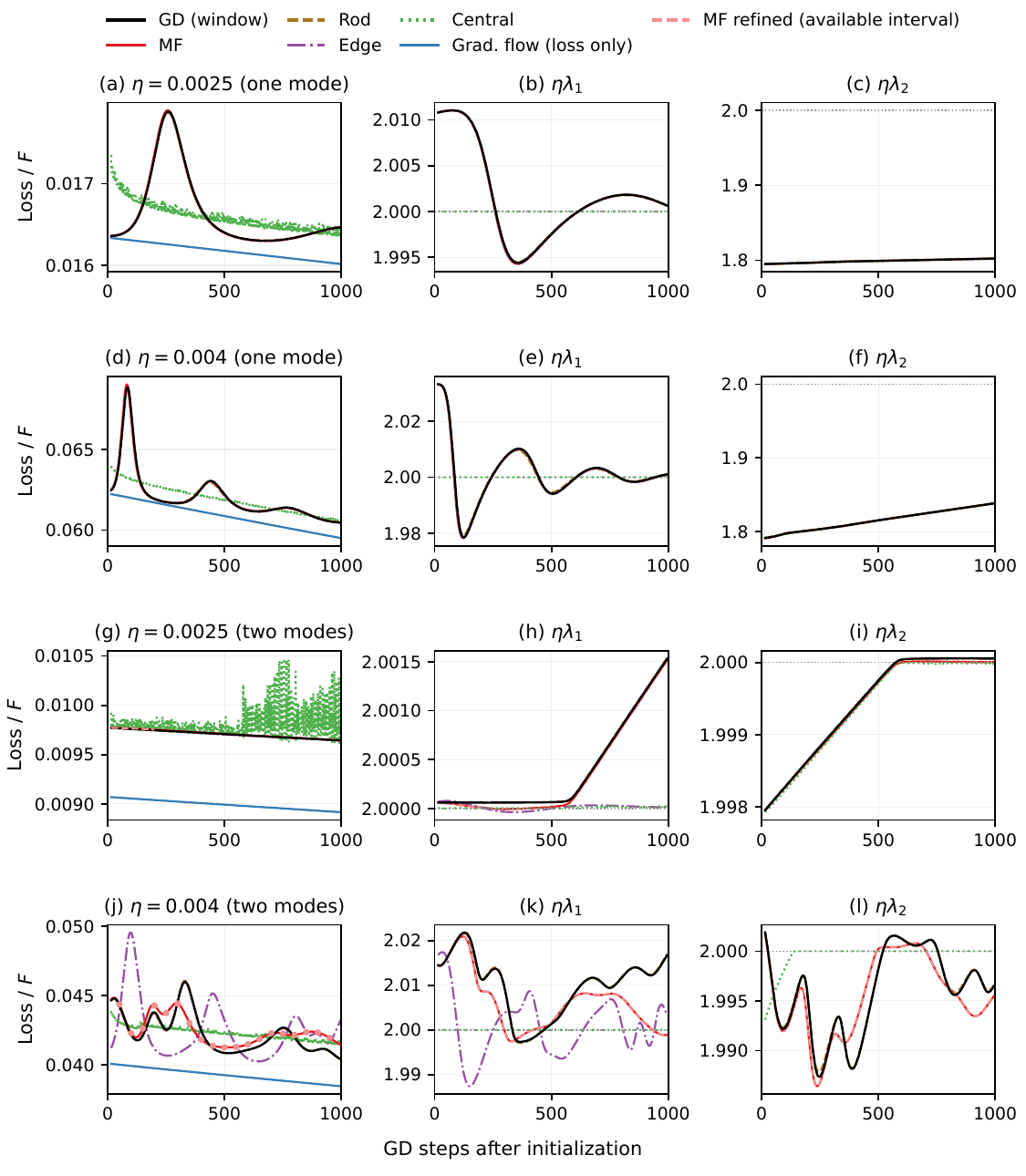}
    \caption{MLP learning-rate sensitivity, seed $0$.
    One-mode initializations at $\eta=0.0025,0.004$ (first two rows),
    then two-mode initializations at the same rates (last two rows),
    each over $1000$ steps.
    Left: GD loss (gray), its window average (black), and flow free energies.
    Middle/right: first/second effective Hessian eigenvalues at the flow
    centers and the GD-window center (black). For Edge, only the proxy
    $\eta\|Hu\|$ is available, shown in the middle column. Gradient flow
    appears only in the loss panels. Light dashed red shows the refinement over its
    available horizon: $200$ steps in the third row and $1000$ in the
    bottom row, where refined energies are available only at checkpoints
    and are shown as red markers. The first two rows have no matched
    full-horizon refinement.}
    \label{fig:mlp-robustness}
\end{figure}
\FloatBarrier

\subsection{Rod-flow solver resolution}
\label{app:rod-substeps}

The Rod solver evolves an auxiliary covariance matrix $\Sigma$, whose
leading eigenpair determines the endpoints $\theta\pm\xi$. With $K$
substeps per GD step, its forward-Euler update before rank truncation is
\begin{equation}
    \Sigma^+
    =\left(1-\frac{2}{K}\right)\Sigma
    +\frac{\eta^2}{4K}\sum_{s=\pm1}
      \nabla E(\theta+s\xi)\nabla E(\theta+s\xi)^T.
    \label{eq:rod-euler-substep}
\end{equation}
For the isolated decay $\dot\Sigma=-2\Sigma$, the exact multiplier over
one substep is $e^{-2/K}$. At $K=2$, the Euler multiplier instead vanishes;
the endpoint-gradient term still depends on the previous state. A
timestep check is therefore needed before interpreting a Rod--GD
discrepancy as a limitation of the model.

\paragraph{One-mode timestep sweep.}
We use the one-mode CNN initialization of \cref{tab:experiment-summary},
fix the auxiliary rank to one and compare $K=2,4,8,16$ over $200$ steps
(\cref{fig:cnn-rod-substeps}). This control uses single-precision
arithmetic and changes only $K$. Taking the $K=16$ trajectory as reference,
the maximum relative differences in free energy and effective sharpness
are $4.97\%$ and $2.18\%$ for $K=2$, but $0.0454\%$ and $0.0385\%$ for
$K=4$. Each difference is divided by the reference value at the same
recorded time. The largest relative covariance-eigenvalue difference for
$K=4$ is $2.99\%$. Thus the scalar observables are substantially less
sensitive to further refinement beyond $K=4$, although this rank-one
control does not validate the higher-rank baseline by itself.

\begin{figure}[ht]
    \centering
    \includegraphics[width=\textwidth]{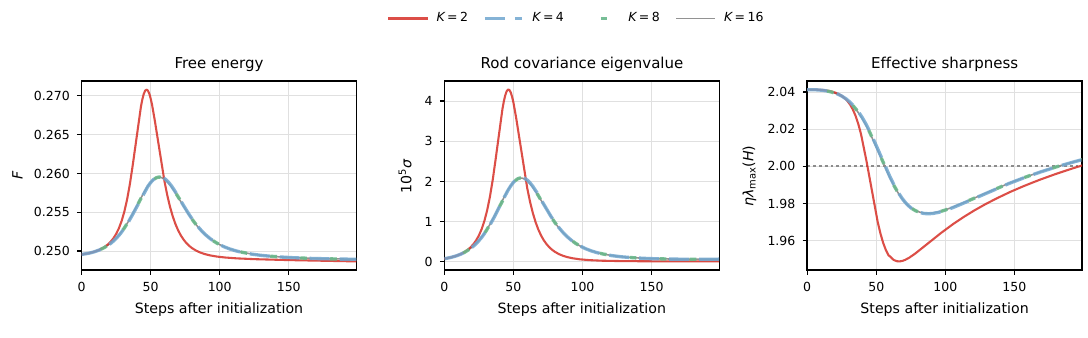}
    \caption{Rod timestep refinement from the one-mode CNN initialization,
    with auxiliary rank one and $200$ steps after initialization.
    Left to right: free energy $F$, the retained covariance eigenvalue
    $\sigma$, and effective sharpness $\eta\lambda_1(H(\theta))$ at the Rod
    center. Here $K$ is the number of Euler substeps per GD step.
    Light, interleaved curves distinguish the nearly overlapping
    $K=4,8,16$ trajectories. The horizontal line in the right panel marks EoS.}
    \label{fig:cnn-rod-substeps}
\end{figure}

\paragraph{Two-mode baseline refinement.}
\label{app:rod-two-mode-resolution}
We next compare auxiliary rank three with $K=4$ against rank four with
$K=8$, using the two-mode CNN initialization over $200$ steps. Both runs
start from the same mean and rank-two auxiliary matrix, so increasing
the rank cap does not change the initial state. Both use double-precision
mean accumulation and small eigensolves, with single-precision network
evaluations, and neither applies an amplitude floor.

At step $200$, the relative Frobenius difference between the endpoint
covariances $\xi\xi^T$ is $0.140\%$, normalized by the larger covariance
norm. The center difference is $0.00768\%$ of the larger displacement
from the common initial center, as in \eqref{eq:solver-mean-gap}.
The full auxiliary matrices also differ by $0.140\%$.
These comparisons support the rank-three, $K=4$ Rod baseline on this
$200$-step CNN interval; they do not establish its accuracy over the
$2000$-step extension or for other architectures.

\subsection{Edge-flow timestep and direction-update controls}
\label{app:edge-resolution}

For Edge flow, increasing the number of substeps changes more than the
integration step size. Algorithm~1 of \citet{marion2026edge} advances
the center and amplitude with step $\rho=\eta/K$, but replaces the
oscillation direction by $Hu/\|Hu\|$ at every substep, where
$H=\nabla^2E(\theta)$ is evaluated at the Edge center. This direction
update has no factor $\rho$. Denoting its frequency per GD step by $D$,
the published algorithm has $D=K$. We therefore test integration
resolution and direction-update frequency separately.

All controls use the one- or two-mode CNN initializations of
\cref{tab:experiment-summary} and cover $200$ GD steps. Within each case,
the initial center, direction and amplitude are identical. Mean
increments are accumulated in double precision; network and
Hessian--vector evaluations remain single precision. The amplitude
floor $10^{-6}$ is inactive throughout. At the five-step diagnostic
times, the second GD-center effective eigenvalue lies in $[1.788,1.815]$
for the one-mode case and $[1.9969,2.0004]$ for the two-mode case.

\paragraph{Separating the two update frequencies.}
With Algorithm~1, changing $K=4$ to $8$ has little effect on the direction
in the one-mode case, but produces a $61.74^\circ$ angle in the two-mode
case (\cref{tab:edge-resolution}). To isolate timestep sensitivity, we
fix $D=4$ and vary $K=4,8,16$. At the start of each direction-update
interval of length $\eta/4$, we compute $Hu/\|Hu\|$. We then advance the
center and amplitude for $K/4$ substeps with the old direction held
fixed, and apply the computed direction at the interval end. At $K=4$,
this reproduces Algorithm~1 exactly.

With $D=4$ fixed, the covariance and center differences approximately
halve when the pair $K=4,8$ is replaced by $K=8,16$, at every saved
nonzero checkpoint. In contrast, changing only $D=4$ to $8$ at fixed
$K=8$ gives a $105.34\%$ covariance difference and a $61.78^\circ$
direction angle at step $200$. The large two-mode sensitivity of
Algorithm~1 is therefore predominantly due to the direction updates
in this test, rather than the center and amplitude time discretization.

\paragraph{Integrating the direction equation.}
Fixing $D$ still leaves discrete direction updates. As a separate check,
we replace the power iteration used in Algorithm~1 \citep{marion2026edge} by a normalized Euler step for Marion's
continuous direction equation:
\[
    \widetilde u=u+2\rho\bigl(Hu-(u^THu)u\bigr),
    \qquad u^+=\widetilde u/\|\widetilde u\|,
    \qquad \rho=\eta/K.
\]
Only this update changes: the center and amplitude equations, including
the sharpness estimate $\|Hu\|$ used for the amplitude, are retained.
At step $200$, the $K=8,16$ covariance difference is $0.0865\%$, and the
$K=16$ result differs from Algorithm~1 with $K=4$ by $0.150\%$
(\cref{tab:edge-resolution}).

\begin{table}[ht]
    \centering
    \small
    \begin{tabular}{@{}llrr@{}}
        \toprule
        CNN case & Comparison & Covariance gap (\%) & Angle (${}^\circ$) \\
        \midrule
        One-mode & Algorithm~1: $K=4,8$ & $2.49$ & $0.0027$ \\
        Two-mode & Algorithm~1: $K=4,8$ & $105.34$ & $61.74$ \\
        Two-mode & Algorithm~1: $K=8,16$ & $120.68$ & $74.00$ \\
        Two-mode & Fixed $D=4$: $K=4,8$ & $0.181$ & $0.0402$ \\
        Two-mode & Fixed $D=4$: $K=8,16$ & $0.0902$ & $0.0201$ \\
        Two-mode & Direction ODE: $K=8,16$ & $0.0865$ & $0.0283$ \\
        Two-mode & \shortstack[l]{Direction ODE $K=16$\\ vs. Algorithm~1 $K=4$} & $0.150$ & $0.0358$ \\
        \bottomrule
    \end{tabular}
    \caption{Edge resolution controls at step $200$, from identical
    initial states within each CNN case. For the two compared runs
    $a,b$, the covariance gap is
    $\|\Sigma_a-\Sigma_b\|_{\rm F}/\max\{\|\Sigma_a\|_{\rm F},\|\Sigma_b\|_{\rm F}\}$,
    with $\Sigma_i=x_i^2u_i u_i^T$; it can exceed $100\%$.
    The angle $\arccos|u_a^Tu_b|$ compares normalized oscillation
    directions independently of their signs. These are differences
    between numerical solutions, not errors relative to GD.}
    \label{tab:edge-resolution}
\end{table}

These controls support retaining Algorithm~1 with $K=D=4$ as the Edge
comparator on the tested CNN interval. They neither establish
long-horizon convergence nor validate the approximation
$\lambda_1(H)\simeq\|Hu\|$. In particular, sensitivity to $D$ alone
does not demonstrate a failure of the continuous Edge equations.
\FloatBarrier

\section{Numerical solver for the mean--fluctuation dynamics}
\label{app:solver}

For a network with $d$ parameters, storing the covariance in \eqref{eq:MF}
would require $d^2$ numbers.  This is not feasible for the networks considered
here.  Our solver instead keeps a low-rank factor in parameter space and uses
a small Krylov space to advance it.  Neither the Hessian nor the covariance is
ever formed as a dense $d\times d$ matrix.

\subsection{Description of the solver}
\label{app:solver-description}

\paragraph{A covariance factor rather than a covariance matrix.}
We write
\begin{equation}
    \Sigma\simeq QQ^T,
    \qquad Q=(q_1,\ldots,q_r)\in\R^{d\times r},\qquad r\ll d.
    \label{eq:solver-covariance-factor}
\end{equation}
The columns of $Q$ carry the oscillation directions and amplitudes.
The initial factor is obtained from the truncated SVD of the centered GD window. Its initialization rank $r_{\rm init}$ is specified separately from the evolution rank cap $r$: in a resolution comparison, we keep the initial covariance fixed while increasing the number of directions the solver may retain.
More precisely, if $\xi_1,\ldots,\xi_m$ are the centered iterates and
\begin{equation}
    \Xi=\frac1{\sqrt m}(\xi_1,\ldots,\xi_m)\in\R^{d\times m},
\end{equation}
then $Q_0Q_0^T$ is the rank-$r_{\rm init}$ spectral truncation of
$\widehat\Sigma_0=\Xi\Xi^T$.  Computing the thin SVD of the $m$ stored
oscillations gives $Q_0$, and $\wh \Sigma_0$ is never assembled.

The sign $\simeq$ in \eqref{eq:solver-covariance-factor} is important because the
full MF equation need not preserve rank (\cref{rem:mf-rank}), so that fixing $r$ is an
additional numerical approximation.  What makes this approximation useful is
that the empirical EoS covariance may concentrate along a few directions, although this does not establish convergence to the full MF equation. 
Indeed, if $P_\perp$ projects orthogonally to $\operatorname{Ran}\Sigma$, then
\[
    P_\perp\dot\Sigma P_\perp
    =\eta^2P_\perp H\Sigma HP_\perp,
\]
so covariance can grow outside its current range. The solver below captures new directions by Krylov enrichment, advances the projected covariance by a frozen-Hessian exponential, and then truncates to rank at most $r$. Numerical refinement therefore varies the rank cap $r$, enrichment depth $p$ and number $K$ of time substeps.

\paragraph{Krylov update of the covariance.}
Fix the mean during one covariance substep and write $H=H(\theta)$.  Even when
$\Sigma=QQ^T$ has small rank, the terms $H\Sigma$ and $H\Sigma H$ immediately
probe the new directions $HQ$.  We retain these directions without forming
$H$ by building
\begin{equation}
    \mathcal K_p(Q,H)=\operatorname{span}\{Q,HQ,\ldots,H^pQ\}.
    \label{eq:solver-krylov-space}
\end{equation}
Let the columns of $V\in\R^{d\times s}$ form an orthonormal basis of this space and define the small matrices
\begin{equation}
    B=V^THV,\qquad S=V^TQQ^TV.
\end{equation}
For fixed $V$, whose construction is described next, the projected covariance equation is
\begin{equation}
    \frac{\dd S}{\dd t}
    =-\eta(BS+SB)+\eta^2BSB.
    \label{eq:solver-krylov-covariance}
\end{equation}
We integrate this small equation exactly while $H$ is frozen, using float64 precision because this step combines values of widely differing magnitudes. Since $B$ is symmetric, write
\[
    B=R\operatorname{diag}(\lambda_1,\ldots,\lambda_s)R^T,
    \qquad Y=R^TSR,
\]
where the columns of $R$ are orthonormal eigenvectors of $B$ and $\lambda_a$ are its eigenvalues.
In these coordinates, the covariance entries evolve independently. A substep of length $\Delta t$ gives
\begin{equation}
    Y_{ab}^{+}
    =\exp\!\left(\Delta t
       [-\eta(\lambda_a+\lambda_b)+\eta^2\lambda_a\lambda_b]
      \right)Y_{ab}.
    \label{eq:solver-modal-update}
\end{equation}
Here and below, $Y^+$, $S^+$ and $Q^+$ denote the updated covariance variables after this substep.
The rate of the variance along an eigenvector of the projected Hessian $B$ is $(\eta\lambda_a)(\eta\lambda_a-2)$, which makes explicit the change from
decay to growth at the EoS threshold. We then return to the original Krylov coordinates:
\[
    S^+=RY^+R^T.
\]
To truncate the covariance rank, we perform a second eigendecomposition,
$S^+=W\Lambda W^T$. The columns of $W$ are orthonormal eigenvectors of $S^+$,
and $\Lambda$ is diagonal with its covariance eigenvalues. We retain at most the
$r$ largest nonnegative eigenvalues in the diagonal matrix $\Lambda_r$, and collect
their eigenvectors in $W_r$. The updated factor is
\begin{equation}
    Q^+=VW_r\Lambda_r^{1/2}.
    \label{eq:solver-refactorization}
\end{equation}
Thus diagonalizing $B$ serves to integrate the projected equation, whereas
diagonalizing $S^+$ selects the covariance directions retained by the solver.
The reconstructed covariance $Q^+(Q^+)^T$ is positive semidefinite.

In the implementation, covariance eigenvalues smaller than $10^{-10}$ times the largest one are also dropped when no amplitude floor is imposed. Thus $r$ is an upper bound on the retained rank, not a prescribed rank at every step.
Before rank truncation, $VS(t)V^T$ solves the full covariance equation with
frozen $H$ if $H$ maps $\mathcal K_p$ into itself.  Otherwise, directions
outside $\mathcal K_p$ are discarded and the exponential update is exact only
for the projected equation \eqref{eq:solver-krylov-covariance}.

\paragraph{Construction of the orthonormal Krylov basis.}

There are two issues with the construction of the Krylov basis $V$. First, at EoS, we have roughly $\lambda_1(H)\approx 2/\eta$ and therefore $H^2Q$ could be about $10^6$ times larger than $Q$ when $\eta=0.002$.
A naive SVD cutoff on $[Q,HQ,\ldots, H^pQ]$ can thus discard directions carried by $Q$. 

Replacing these blocks by $[Q,(\eta H)Q, \ldots, (\eta H)^p Q]$ gives the same Krylov space with more comparable block magnitudes. Nevertheless, the generated columns can still be almost linearly dependent, which makes basis extraction sensitive to roundoff errors. 

Our implementation constructs $V$ incrementally by the \emph{Arnoldi} algorithm applied to the scaled, frozen Hessian $q\mapsto\eta H(\theta)q$ \citep{saad2003iterative,gutknecht2007block}, which we describe now.
First, the initial $V_0$ is obtained by QR factorization, i.e. $V_0R_0$ denotes the matrix formed by normalized columns of $Q$ while $V_0^TV_0=\Id$ and $R_0$ is upper triangular.  
Then, given $V=[V_0,\ldots,V_j]$, we compute $Z=\eta H(\theta)V_j$ through Hessian--vector products (HVPs), applying the scaled Hessian to each column of the newest block $V_j$. We remove its components along already existing basis vectors by applying
\[
    Z\leftarrow Z-V(V^TZ)
\]
twice to reduce loss of orthogonality from roundoff. 
A thin SVD of this residual identifies the new directions: we retain its left singular vectors whose singular values exceed $10^{-5}\|\eta HV_j\|_{\rm F}$. 
This explicit cutoff uses the norm of the current HVP block \emph{before} orthogonalization, not the norm of a concatenation of differently scaled powers. 
This cutoff selects only new columns to add to $V$; it does not remove any existing column, including those of $V_0$. If no singular value passes the cutoff, we stop enlarging the basis and use the current $V$ for the covariance update. Otherwise, we apply the above projection twice more to the selected left singular vectors to remove roundoff components along $V$. A QR factorization then makes these new columns mutually orthonormal. They form the next block $V_{j+1}$, and we enlarge the basis by $V\leftarrow[V,V_{j+1}]$.

The depth $p$ limits the number of enrichment rounds, so that the temporary basis has at most $(p+1)r$ columns, whereas at most $r$ covariance directions are retained in \eqref{eq:solver-refactorization}. Normalizing the starting columns changes only the basis construction: $S=V^TQQ^TV$ still uses the original $Q$, so no covariance amplitude is rescaled or floored. 

\paragraph{Mean equation and free energy.}
The same factor makes the remaining high-dimensional contractions
matrix-free.  Namely,
\begin{align}
    \nabla H(\theta):\Sigma
    &\simeq \sum_{a=1}^r D^3E(\theta)[q_a,q_a,\,\cdot\,],
    \label{eq:solver-factor-mean}\\
    F(\theta,\Sigma)
    &\simeq E(\theta)+\frac12\sum_{a=1}^r q_a^TH(\theta)q_a.
    \label{eq:solver-factor-free-energy}
\end{align}
Automatic differentiation evaluates each third-order directional derivative
and Hessian--vector product separately.  Thus we form neither $D^3E$, $H$, nor
$\Sigma$. The coarse-MF mean correction uses the full-space factor $Q$, not its projection onto the tracked Hessian eigenvectors.

\paragraph{Energy observables in the figures.}
In the main-text architecture comparisons, coarse-MF energy curves use full-factor logs $F(\theta,QQ^T)$ for MLP and the projected logs $F(\theta,UXU^T)$ for CNN and ResNet, where $U$ is the tracked Hessian basis and $X=U^TQQ^TU$. This concerns only the displayed observable: the coarse-MF mean drift uses $QQ^T$ in all cases. 
Furthermore, for these CNN and ResNet coarse-MF runs, the difference between the two energies is below $0.008\%$ of $|F(\theta,QQ^T)|$ at the saved checkpoints. Across the architecture, seed, learning-rate and onset experiments, the relative difference is below $0.13\%$ at the available logged or checkpointed states, so that the conclusions of the energy plots are unchanged. 
\paragraph{Time stepping, controls, and cost.}
One unit of algorithmic time corresponds to one GD iteration.  We divide it
into $K$ substeps and use a symmetric splitting: a half covariance update with the current Hessian, an explicit Euler mean update, then a half covariance update with the new Hessian. 
Despite the symmetric ordering, the Euler mean step makes the coupled scheme first order.
For comparison, Strang splitting with exact subflows is second order \citep{strang1968construction}.

For the neural-network flow runs, the mean increments are accumulated in a double-precision copy of $\theta$, which is carried across substeps and saved in checkpoints. The gradient and HVP evaluations use a rounded single-precision copy. This prevents small increments from being lost each time they are added to a much larger parameter value, without converting all network computations to double precision. In particular, factors, Krylov bases and HVPs remain single precision.

The raw-MF control instead uses the full-space
linearized-GD map
\begin{equation}
    Q^+=(\Id-\eta H)Q,
    \qquad
    \Sigma^+=(\Id-\eta H)\Sigma(\Id-\eta H),
    \label{eq:solver-raw-covariance}
\end{equation}
once per algorithmic step.
Contrary to the coarse MF model, this map preserves rank when $\Id - \eta H$ is injective on $\operatorname*{span}\{Q\}$. 

The covariance state requires $O(dr)$ memory.  At Krylov depth $p$, its
temporary workspace has at most $(p+1)r$ ambient vectors and a dense matrix of
size at most $(p+1)r$.  The dominant operations are therefore a small number
of Hessian--vector products and one third-order directional derivative per
factor column; all dense linear algebra is low-dimensional.

\paragraph{Projection and truncation diagnostics.}

At each covariance substep we record
\begin{equation}
    \rho_H=\frac{\|HV-VB\|_{\rm F}}{\|HV\|_{\rm F}}.
    \label{eq:solver-projection-diagnostics}
\end{equation}
This measures how far $HV$ lies outside the Krylov space, and vanishes when $H(\operatorname*{Ran}V)\subseteq \operatorname*{Ran}V$. We set $\rho_H=0$ if $HV=0$.

To measure the covariance discarded by truncation, let $I_r$ index the retained eigenvalues of $S^+$ and put $\mu_a=\max\{\lambda_a(S^+),0\}$. Before any amplitude floor, we record
\begin{equation}
    d_{\rm tr}=\sum_{a\notin I_r}\mu_a.
    \label{eq:solver-truncation-diagnostics}
\end{equation}
This is the trace of the discarded positive covariance part. In the dense checks below, we report the sum of $d_{\rm tr}$ and the largest $\rho_H$ over all covariance substeps of the run.

These diagnostics help assess the covariance representation. We complement them with comparisons of the full mean and covariance under refinement of $(r,p,K)$ to assess numerical convergence.

\subsection{Numerical validation}
\label{app:mf-solver-validation}

We first test the covariance solver against exact solutions for quadratic losses, then the coupled mean--covariance solver against a high-accuracy dense numerical reference for a nonlinear loss. These low-dimensional benchmarks use float64 throughout.
Then, for neural networks, we instead compare MF trajectories at different resolutions of $(r,p,K)$ from the same initial state. 
Agreement under these refinements supports the tested numerical settings: neither check should be confused with agreement between MF and GD, which tests the model approximation.

\paragraph{Dense reference checks.}
The first two checks use quadratic losses $E(w)=\tfrac12w^THw$ and initial mean $\theta_0=0$. Thus $H$ is constant and the mean stays zero: specifying $H$, $\eta$ and $Q_0$ suffices to determine the covariance evolution.
For $\eta=1$, $H=\operatorname{diag}(2,1)$ and $Q_0=(1,1)^T$, the exact covariance is
\[
    \Sigma(t)=\begin{pmatrix}1&e^{-t}\\e^{-t}&e^{-t}\end{pmatrix},
    \qquad \det\Sigma(t)=e^{-t}-e^{-2t}>0\quad(t>0).
\]
At $t=1$, the solver's covariance with $r=2,p=1$ agrees with this solution to relative Frobenius error below $10^{-11}$. At $r=1,p=1$, that error is $0.182$ for $K=8$ and $0.186$ for $K=64$: time refinement does not cure rank bias. At $p=0,r=2,K=8$, the error is $0.470$ despite zero discarded trace.
Here the missing enrichment is detected by $\max\rho_H=0.316$, whereas this residual is below $10^{-14}$ in the $p=1$ tests.

A second quadratic-loss test checks invariance under rescaling the loss and inversely rescaling the learning rate. In dimension six we set
\[
    E_\eta(w)=\frac{1}{2\eta}w^TAw,
    \qquad A=\operatorname{diag}(2.02,1.97,1.5,1.1,0.7,0.3),
\]
so that $H=A/\eta$ and $\eta H=A$ stays fixed. 
We use the same rank-two initial factor, $r=6$, $p=2$ and substep length $1/8$. Changing $\eta$ from $1$ to $0.002$, while rescaling the loss as above, leaves the exact covariance update unchanged. At both scalings, the Arnoldi construction retains six directions and the relative covariance error is below $2\times10^{-15}$. Thus the solver reproduces the same exact update to floating-point accuracy. 

The third check tests the coupled dynamics, where the Hessian changes along the mean trajectory. 
We compare the solver's covariance and mean updates with a dense numerical integration of the full coupled ODE in dimension six, with $\eta=1$ and the nonlinear loss
\[
    E(w)=\tfrac12w^TH_0w+\frac{0.08}{6}\sum_i(Mw)_i^3
                           +\frac{0.02}{4}\sum_iw_i^4,
    \qquad H_0=\operatorname{diag}(2.02,1.8,1.1,0.7,0.4,0.1).
\]
Here $M$ is orthogonal and the initial covariance has rank two; the generator seed $17$, matrix and initial state are saved with the benchmark. The dense reference uses DOP853 with relative tolerance $10^{-12}$ and absolute tolerance $10^{-14}$; a second integration at $10^{-10},10^{-12}$ changes the final state by less than $3\times10^{-13}$ in maximum norm. \Cref{tab:dense-mf-validation} separates time refinement from rank and Krylov-depth effects. 
These small-system checks validate the numerical building blocks.

\begin{table}[ht]
    \centering
    \small
    \begin{tabular}{@{}rrrcccc@{}}
        \toprule
        $r$ & $p$ & $K$ & \shortstack{Relative mean\\error} & \shortstack{Relative covariance\\error} & $\sum d_{\rm tr}$ & $\max\rho_H$ \\
        \midrule
        6 & 3 & 4  & $2.14\,10^{-2}$ & $1.38\,10^{-4}$ & $0$ & $<2\,10^{-15}$ \\
        6 & 3 & 8  & $1.06\,10^{-2}$ & $6.79\,10^{-5}$ & $0$ & $<2\,10^{-15}$ \\
        6 & 3 & 16 & $5.25\,10^{-3}$ & $3.37\,10^{-5}$ & $0$ & $<2\,10^{-15}$ \\
        2 & 3 & 16 & $5.14\,10^{-3}$ & $4.68\,10^{-1}$ & $4.61\,10^{-3}$ & $<2\,10^{-15}$ \\
        4 & 3 & 16 & $5.25\,10^{-3}$ & $1.34\,10^{-2}$ & $1.02\,10^{-4}$ & $<2\,10^{-15}$ \\
        6 & 0 & 16 & $5.42\,10^{-3}$ & $6.58\,10^{-1}$ & $0$ & $0.549$ \\
        6 & 1 & 16 & $5.25\,10^{-3}$ & $4.37\,10^{-5}$ & $0$ & $0.117$ \\
        \bottomrule
    \end{tabular}
    \caption{Dense coupled MF check at $t=2$. Errors use the dense-reference Euclidean and Frobenius norms as denominators. Discarded trace is summed over the whole run. The last column is the largest $\rho_H$ over all covariance substeps. All rows start from exactly the same mean and covariance.}
    \label{tab:dense-mf-validation}
\end{table}

The first three rows of \cref{tab:dense-mf-validation} keep the full rank $r=6$ and depth $p=3$ fixed: doubling $K$ approximately halves both errors, as expected for first-order time stepping. At fixed $K=16$, reducing the rank or taking $p=0$ produces a much larger covariance error. The diagnostics distinguish these two limitations: rank truncation discards covariance despite a negligible $\rho_H$, whereas $p=0$ has zero discarded trace but a large $\rho_H$. However, $p=1$ already gives a covariance error close to that at $p=3$, despite $\max\rho_H=0.117$. 
This residual measures subspace invariance without weighting directions by their covariance amplitudes, so that its value need not track the trajectory error. We therefore assess accuracy through the mean and covariance comparisons, using the local diagnostics to identify sources of error.

\paragraph{Neural-network refinement protocol.}
Within each comparison, the loss, dataset, seed, learning rate, GD initialization window and initial state $(\theta_0,\Sigma_0)$ are fixed. In particular, changing the solver rank $r$ does not change the initial covariance rank. 
The runs use the same mixed-precision arithmetic described in \cref{app:solver-description}. 
We vary $r$, $p$ and $K$ separately in the short CNN controls, then compare the baseline $(3,2,4)$ with the combined refinement $(4,3,8)$ over the common intervals listed in \cref{tab:mf-network-refinement}.

At every common checkpoint, write $(\theta^{\rm b},\Sigma^{\rm b})$ and $(\theta^{\rm f},\Sigma^{\rm f})$ for the two MF states (baseline and refined), with $\Sigma^a=Q^a(Q^a)^T$ for $a\in\{\mathrm b,\mathrm f\}$. We measure
\begin{align}
    \delta_\theta(t)
    &=\frac{\|\theta^{\rm b}(t)-\theta^{\rm f}(t)\|_2}
    {\max\{\|\theta^{\rm b}(t)-\theta_0\|_2,\|\theta^{\rm f}(t)-\theta_0\|_2\}},
    \label{eq:solver-mean-gap}\\
    \delta_\Sigma(t)
    &=\frac{\|\Sigma^{\rm b}(t)-\Sigma^{\rm f}(t)\|_{\rm F}}
    {\max\{\|\Sigma^{\rm b}(t)\|_{\rm F},\|\Sigma^{\rm f}(t)\|_{\rm F}\}}.
    \label{eq:solver-covariance-gap}
\end{align}
We set $\delta_\theta=0$ if both means equal their common initial value $\theta_0$, and $\delta_\Sigma=0$ if both covariances vanish. 
The mean gap is therefore relative to the displacement since initialization.
The covariance comparison uses the full factors $Q$, evaluated in a shared QR basis without assembling $d\times d$ matrices so that we compare $QQ^T$, not the factors themselves or only their projections onto near-EoS Hessian modes.

The tables report $100\delta_\theta$ and $100\delta_\Sigma$, either at the last checkpoint or maximized over the common checkpoints, as specified in each caption. 

\paragraph{Separate rank, depth and timestep controls.}
\label{app:cnn-mf-resolution}
We vary each resolution parameter separately in the two-mode CNN over $200$ steps, with width $32$, $n=8000$, $\eta=0.002$, $m=4$, seed $0$ and the same rank-two initial covariance. The first four comparisons in \cref{tab:cnn-mf-resolution} change only the indicated parameter from $(r,p,K)=(3,2,4)$. Rank and depth have little effect; the mean gaps approximately halve when $K$ doubles.

\begin{table}[ht]
    \centering
    \small
    \begin{tabular}{@{}lrr@{}}
        \toprule
        Comparison & Covariance gap (\%) & Mean gap (\%) \\
        \midrule
        $r=3$ versus $4$ & $0.0195$ & $0.00139$ \\
        $p=2$ versus $3$ & $0.0146$ & $0.00134$ \\
        $K=2$ versus $4$ & $0.0602$ & $0.00738$ \\
        $K=4$ versus $8$ & $0.0255$ & $0.00366$ \\
        $(3,2,4)$ versus $(4,3,8)$ & $0.0302$ & $0.00357$ \\
        \bottomrule
    \end{tabular}
    \caption{Separate CNN resolution controls. Mean and covariance gaps use \eqref{eq:solver-mean-gap}--\eqref{eq:solver-covariance-gap} at checkpoint $200$.}
    \label{tab:cnn-mf-resolution}
\end{table}

\paragraph{Matched refinement checks.}
\label{app:cnn-long-mf-resolution}
\label{app:mlp-mf-resolution}
\label{app:resnet-mf-resolution}
\label{app:mlp-rate-resolution}
\label{app:mlp-seed-resolution}
\Cref{tab:mf-network-refinement} collects the combined $(3,2,4)$ versus $(4,3,8)$ checks: the full reference CNN, MLP and ResNet horizons, plus the seed-$1$ check and the two-mode learning-rate comparisons of \cref{fig:mlp-robustness}. The refinement covers $1000$ steps at $\eta=0.004$ and only the first $200$ at $\eta=0.0025$. The initial covariance rank is fixed within each pair.

\begin{table}[ht]
    \centering
    \small
    \begin{tabular}{@{}lrrrrrr@{}}
        \toprule
        Network & $\eta$ & Seed & Matched steps & $r_{\rm init}$ &
        \shortstack{Max.\ mean\\gap (\%)} & \shortstack{Max.\ covariance\\gap (\%)} \\
        \midrule
        CNN    & .002 & 0 & 2000 & 2 & .00752 & .0540 \\
        MLP    & .002 & 0 & 4000 & 3 & .00728 & .111 \\
        ResNet & .004 & 0 & 200  & 2 & .0245  & .126 \\
        \midrule
        MLP    & .002 & 1 & 4000 & 3 & .00583 & .113 \\
        MLP    & .004 & 0 & 1000 & 3 & .00911 & .0696 \\
        MLP & .0025 & 0 & 200 & 3 & .00391 & .0273 \\
        \bottomrule
    \end{tabular}
    \caption{Coarse-MF refinement checks for selected two-mode experiments. Entries are maxima of $100\delta_\theta$ and $100\delta_\Sigma$ over checkpoints spaced $50$ steps apart. The baseline and refinement share their initial state. Matched steps gives the common comparison horizon, which may be shorter than the baseline horizon in \cref{tab:experiment-summary}.}
    \label{tab:mf-network-refinement}
    \label{tab:cnn-long-mf-resolution}
    \label{tab:mlp-mf-resolution}
    \label{tab:resnet-mf-resolution}
    \label{tab:mlp-rate-resolution}
\end{table}

Across the matched pairs, maximum relative mean and covariance gaps remain below $0.025\%$ and $0.127\%$. 
These checks support the listed coarse-MF resolutions and horizons; they are not exact-solution error bounds and do not validate other flows or initializations.
The accompanying code and metadata retain the launchers, configurations and full checkpoint comparisons.

\end{document}